\documentclass[11pt]{amsart}
\usepackage{amsmath}
\usepackage{epsfig}
\usepackage{graphics,graphicx}
\usepackage{graphicx}
\usepackage[tight,FIGTOPCAP]{subfigure}
\usepackage{amsfonts}
\usepackage{caption}
\usepackage{verbatim}
\usepackage{amssymb}
\usepackage{latexsym}
\usepackage{amsmath, amsthm, amssymb}
\usepackage{color}

\newtheorem{corollary}{Corollary}[section]

\newtheorem{lemma}[corollary]{Lemma}

\newtheorem{remark}[corollary]{Remark}
\newtheorem{theorem}[corollary]{Theorem}
\newcommand{\mylabel}[1]{\label{#1}
            \ifx\undefined\stillediting
            \else \fbox{$#1$}\fi }
\newcommand{\BE}{\begin{equation}}

\newcommand{\EEQ}{\end{equation}}
\newcommand{\rfb}[1]{\mbox{\rm
   (\ref{#1})}\ifx\undefined\stillediting\else:\fbox{$#1$}\fi}

\newfont{\Blackboard}{msbm10 scaled 1200}

\newfont{\roma}{cmr10 scaled 1200}

\def\CC{\rm \hbox{C\kern-.56em\raise.4ex
         \hbox{$\scriptscriptstyle |$}\kern+0.5 em }}

\newcommand{\mm}    {{\hbox{\hskip 0.5pt}}}

\newcommand{\bluff} {{\hbox{\raise 15pt \hbox{\mm}}}}
\makeatletter
\def\section{\@startsection {section}{1}{\z@}{-3.5ex plus -1ex minus
    -.2ex}{2.3ex plus .2ex}{\large\bf}}
\makeatother
\def\be{\begin{equation}}
\def\ee{\end{equation}}

\begin{document}

\thispagestyle{empty}
\title[Stability of a nonlinear time-delayed  dispersive equation]{QUALITATIVE AND NUMERICAL study of the stability of A NONLINEAR TIME-DELAYED DISPERSIVE EQUATION}
\date\today
\author{Ka\"{i}s Ammari}
\address{UR Analysis and Control of PDEs, UR 13ES64, Department of Mathematics, Faculty of Sciences of Monastir, University of Monastir, Tunisia}
\email{kais.ammari@fsm.rnu.tn}

\author{Boumedi\`ene Chentouf}
\address{Kuwait University, Faculty of Science, Department of Mathematics, Safat 13060, Kuwait}
\email{boumediene.chentouf@ku.edu.kw}

\author{Nejib Smaoui}
\address{Kuwait University, Faculty of Science, Department of Mathematics, Safat 13060, Kuwait}
\email{n.smaoui@ku.edu.kw}

\begin{abstract}
This paper deals with the stability analysis of a nonlinear time-delayed dispersive equation of order four. First, we prove the well-posedness of the system and give some regularity results. Then, we show that the zero solution of the system exponentially converges to zero when the time tends to infinity provided that the time-delay is small  and the damping term satisfies reasonable conditions. Lastly, an intensive numerical study is put forward and numerical illustrations of the stability result are provided. 
\end{abstract}

\subjclass[2010]{35L05, 35M10}
\keywords{Nonlinear dispersive equation, time-delay, stability, numerical simulations}

\maketitle



\section{Introduction} \label{secintro}
\setcounter{equation}{0}

The qualitative and numerical analysis of nonlinear dispersive equations has attracted the attention of a huge number of authors from various disciplines. This is due to the fact that such equations describe miscellaneous physical phenomena, such as surface water waves in shallow water \cite{bou, kdv}, turbulent states in a distributed chemical reaction system and plane flame propagation \cite{kt, r1}, propagation of ion-acoustic waves in plasma, and pressure waves in liquid-gas bubble mixture \cite{cc, 4, jk, lig, whi1, whi2, za}.

It is worth noting that the nonlinearity in the equations governing the models mentioned above makes the mathematical problem more challenging, and its analysis often requires elaborate techniques. The situation is even more complicated when a time-delay occurs in the equation (see  for instance \cite{c2, c1, AABM, ANP1,ANP2,ANP3,ammarinicaise, c5, NPSicon06, zou} for other types of physical systems).

One particular dispersive equation is the nonlinear partial differential equation (PDE) known in literature as the Korteweg-de Vries-Burgers (KdVB) equation in a bounded interval
\[
u_t(x,t) - \nu \, u_{xx}(x,t) + u (x,t) u_x(x,t) + \mu \left[ u_{xxx}(x,t) +u_x+b(x) u \right] = 0, \,  (x,t) \in Q, \\
\]
where $Q=(0,\ell) \times (0,+\infty)$, $\nu$ and $\mu$ are positive physical parameters, while $b(x)$ is a given nonnegative function. The above equation exhibits the properties of dispersion and dissipation, and has been widely used to describe a number of physical parameters such as unidimensional propagation of small  waves in nonlinear dispersive mediums and long waves in shallow water (for instance, see  \cite{amcrep, amcrepbbm, lina, Pazoto,  Perla} for the stability, \cite{Cerpa, Cerpa_Crepeau, Crepeau, CoCre, Rosier, ro3} for the control problem, and \cite{usm} for numerical analysis). The reader can also find in \cite{capz, Cerpa, roz} the statement of the main results related to the stabilization and control problems of the KdV equation in a bounded interval. In turn, one can find in the books \cite{erd, lipo} and the references therein numerous discussions on the case of the KdV on the whole real line, the half-line, or with periodic boundary conditions.

In the case when $\mu=0$, the above (KdVB) equation is called the Burger's equation and has been the subject of many studies \cite{ba1, ba2, ba3, kr, liu2, s7, s8}. In turn, the following time-delayed Burgers equation has been considered in \cite{liu, s5, wang}:
\[
u_t(x,t) - \nu \, u_{xx}(x,t) +  u(t-\tau,x)  u_x(x,t) = 0, \,  (x,t) \in Q. \\
\]
In fact, homogeneous Dirichlet boundary conditions were used \cite{liu}, whereas  periodic boundary conditions were considered in \cite{s5}. The Lyapunov function technique has been utilized in order to establish the exponential stability of the solutions provided that the time-delay $\tau$ is sufficiently small. This outcome has been obtained in \cite{wang} by using another method, namely, the fixed point theorem and the comparison principle.

The control problem of the generalized Korteweg-de Vries Burgers $(GKdVB)$ equation (without delay)
\[
u_t(x,t) - \nu \, u_{xx}(x,t) + \mu \, u_{xxx}(x,t) +  u^\alpha(x,t)  u_x(x,t) + u_x+ b(x) u = 0, (x,t) \in Q,\, \alpha \in \mathbb{N},
\]
has also been extensively investigated by many researchers in finite and infinite domains (see for example \cite{bi1, bi2, bo1, bo2, bo3, bo4, mo1, s1, s2, s3, s4, mo2, s6}).

Inspired by the paper \cite{liu}, the present article is devoted to the qualitative and numerical analysis of the following {\bf delayed} dispersive equation in a bounded domain $[0,\ell]$, with initial and boundary conditions:
\begin{equation}
\label{1}
\left\{
\begin{array}{ll}
u_t(x,t) - \nu \, u_{xx}(x,t) + \mu \, u_{xxxx}(t,x)+  u(x,t-\tau)  u_x(x,t) +a(x) u (x,t)= 0,& (x,t) \in Q, \\
u(0,t) = u(\ell,t) = u_x (0,t) = u_x (\ell,t) =0,&  t > 0, \\
u(x,s) = v(x,s),&   (x,s) \in \tilde{Q},
\end{array}
\right.
\end{equation}
where $\tau$ is the time-delay, whereas $\nu$ and $\mu$ are positive physical parameters. Furthermore,   $\tilde{Q}=(0,\ell) \times [-\tau,0])$ and $a(x) \in L^{\infty}((0,\ell))$ is a non-negative  function. Note that the above PDE can be viewed as a perturbation (by a fourth order derivative term and the damping term $a(x)u$) of the delayed Burger's equation studied in \cite{liu} (see \cite{bo4} for the case of a damping in the generalized KdV equation on whole the real line and without delay).

The main results of this paper are twofold: first, we show that the problem (\ref{1}) is well-posedness in the integral sense in a functional space. Second, the solutions are shown to be exponentially stable as long as the delay is small. These findings complement the results in \cite{liu}, where the considered equation is Burger's equation of order two. In order to accomplish these outcomes, we shall proceed as in \cite{liu} with  of course a number of changes born out of necessity due to the higher order derivative in our case.

Last but not least, a numerical comparative study will be provided by conducting numerical simulations of the solutions of the system under different values of the time-delay $\tau$ and the physical parameters $\nu$ and $\mu$.

The remainder of the paper is organized as follows: In Section 2, we set the problem in its natural functional space and the global well-posedness of the problem is established. Section 3 is consecrated to the exponential stability of the solution by means of the Lyapunov method and under a smallness condition of the time-delay. Our results are ascertained and illustrated through numerical simulations. Finally, the article ends with concluding remarks.

\section{Well-posedness of the problem} \label{well-posedness}
\setcounter{equation}{0}
In this section, we will provide a well--posedness result for the delayed  problem (\ref{1}).

First of all, let us introduce, on one hand, a number of notations that will be systematically used in the sequel. $I$ denotes the interval $(0,\ell)$, $H_0^m (I)$ the usual Sobolev space. The norm of $L^2 (I)$ will be denoted by $\| \cdot \|$, whereas $\| \cdot \|_{\infty}$ represents the norm of  $L^{\infty} (I)$. The space $C ( J; \, H_0^2 (I) )$ denotes the space of continuous functions on a closed bounded interval $J$ with values in $H_0^2 (I)$ and will be endowed with the supremum norm $\| \cdot \|_{c}=\displaystyle \sup_{s \in [-\tau,0]} \| \cdot \|_{H_0^2 (I)}$.
On the other hand, the following Wirtinger's inequalities \cite{H} will be frequently used:
\begin{eqnarray}
\displaystyle \int_{0}^{\ell} f^2 (x) \, dx  &\leq&   \dfrac{\ell^2}{\pi^2} \int_{0}^{\ell} f_x^2 (x) \, dx, \; \forall f \in H_0^1 (I); \label{wi1} \\
\displaystyle \int_{0}^{\ell} f_x^2 (x) \, dx  &\leq&   \dfrac{\ell^2}{\pi^2} \int_{0}^{\ell} f_{xx}^2 (x) \, dx, \; \forall f \in H_0^1 (I) \cap H^2 (I). \label{wi2}
\end{eqnarray}
In view of (\ref{wi1})-(\ref{wi2}), we shall equip $H_0^2 (I)$ with an equivalent norm defined by: $\| u\|_{H_0^2 (I)}=\| u_{xx}\|$. The well-known Young's inequality will be also applied throughout this article:
\begin{equation}
\displaystyle 2 ab  \leq  \dfrac{a^2}{2 \varepsilon} + 2 \varepsilon b^2, \quad a, \, b \in \mathbb{R} \, \varepsilon >0. \label{you}
\end{equation}

Thereafter, the problem (\ref{1}) can be written as follows:
\begin{equation}\label{2}
\left\{
\begin{array}{l}
u_t={\mathcal A} u + {\mathcal B}(u^t),\\
u(\cdot,s)=v(\cdot,s), \, s \in [-\tau,0],
\end{array}
\right.
\end{equation}
where the linear operator ${\mathcal A}$ is defined by
\begin{equation}\label{opa}
\begin{array}{l}
\displaystyle
\mathcal{D} ({\mathcal A})=H^4(I) \cap H_0^2(I),\\
{\mathcal A} u=\nu u_{xx} - \mu u_{xxxx} - a \,  u, \;\; \forall u \in \mathcal{D} ({\mathcal A}).
\end{array}
\end{equation}
In turn, $u^r(\theta)=u(r+\theta),$ where $r>0$ and $\theta \in [-\tau,0]$ and the operator ${\mathcal B}$ is nonlinear defined by:
\begin{equation}\label{opb}
\begin{array}{l}
\displaystyle
\mathcal{D} ({\mathcal B})=C([-\tau,0]; \, H_0^2 (I)),\\
{\mathcal B} (z)=- z_x (0) z(-\tau), \;\; \forall z \in \mathcal{D} ({\mathcal B}).
\end{array}
\end{equation}
Next, recalling that $a \in L^{\infty}(I)$ is a non-negative  function, one can readily check that the linear operator ${\mathcal A}$ defined by (\ref{opa}) generates an exponentially stable $C_0$-semigroup $S(t)$ on $L^2(I)$. Then, the problem (\ref{2}) can be rewritten as an integral equation
\begin{equation}
\left\{
\begin{array}{ll}
u(t)=S(t)u_0 (0)+\displaystyle \int_{0}^{t} S(t-s) {\mathcal B}(u^r) \, dr, & t>0,\\
u(t)=v(t), & t \in [-\tau, 0].
\end{array}
\right.
 \label{mi}
	\end{equation}

In the sequel, any continuous solution of (\ref{mi}) is called a mild solution of (\ref{2}).

Our well-posedness result is stated below

\begin{theorem} Given an initial condition $v=v(x,s) \in C([-\tau,0]; \, H_0^2 (I))$, the system (\ref{1}), or equivalently (\ref{2}), has a unique global mild solution $ u \in C([-\tau,\infty]; \, H_0^2 (I))$.
\label{t1}
\end{theorem}

\begin{proof}
First, we claim that the nonlinear operator ${\mathcal B}$ defined by (\ref{opb}) is locally Lipschitz. Indeed, given $v \in \mathcal{D} ({\mathcal B})=C([-\tau,0]; \, H_0^2 (I))$, we have:
\begin{eqnarray}
\| {\mathcal B} (z) -{\mathcal B} (\hat{z}) \| &=& \|  \hat{z}_x (0) \hat{z}(-\tau)- z_x (0) z(-\tau)\| \nonumber \\
& \leq &  \sqrt{\ell} \left( \|\hat{z}(-\tau)- z(-\tau)\|_{\infty} \|  {z}_x (0) \| +
\| \hat{z}(-\tau)\|_{\infty} \| \hat{z}_x (0) -z_x (0) \| \right).
\label{3}
\end{eqnarray}
Applying the interpolation inequalities of Gagliardo-Nirenberg \cite{br} as well as Wirtinger's inequalities (\ref{wi1})-(\ref{wi2}), one can deduce the existence of  a positive constant $K$ such that  (\ref{3}) gives
\begin{eqnarray}
\| {\mathcal B} (v) -{\mathcal B} (\hat{z}) \| &\leq&   K \left( \|\hat{z}(-\tau)- z(-\tau)\|_{H_0^2(I)} \|  {z}_x (0) \| +
\| \hat{z}(-\tau)\|_{H_0^2(I)} \| \hat{z}_x (0) - z_x (0) \| \right) \nonumber \\
& \leq & K  \left( \|\hat{z}(-\tau)- z(-\tau)\|_{H_0^2 (I)} \|  {z} \|_c
+ \| \hat{z}\|_{c} \|  \hat{z} (0) - z(0) \|_{H_0^2 (I)} \right) \nonumber \\
& \leq &  K  \left( \| {z}\|_{c} + \| \hat{z}\|_{c} \right)  \| z-\hat{z} \|_c=\tilde{K}  \| z-\hat{z} \|_c,
\label{4}
\end{eqnarray}
where $\tilde{K} =K \left( \| {z}\|_{c} + \| \hat{z}\|_{c} \right).$

Whereupon, for each initial datum $v=v(x,s) \in C([-\tau,0]; \, H_0^2 (I))$, there exists a positive constant $T=T(u_0)$ such that the system (\ref{1}) has a unique local mild solution $ u \in C([-\tau,T]; \, H_0^2 (I))$ given by the variations of constant formula (\ref{mi}).

It remains to show that the solution $u$ is global. To do so,  the space variable $x$ will be omitted in the sequel whenever it is unnecessary. Then, taking the  inner product of (\ref{1}) in $L^{2} ( I )$  with  $u_{xxxx}$, integrating by parts, and using the boundary conditions, we have for any $ t \in [-\tau, 0]$:
\begin{eqnarray}
&& \dfrac{1}{2} \dfrac{d}{dt} \|u_{xx}(t)\|^2 - \nu \int _{0}^{\ell} u_{xx}(t) u_{xxxx} (t) ~dx  + \int _{0}^{\ell}  u(t-\tau)  u_x(t) u_{xxxx} (t) ~dx+ \mu \|u_{xxxx} (t)\|^2=0 \nonumber \\
&& - \int _{0}^{\ell} a(x) u(t) u_{xxxx} (t) ~dx.
\label{sam}
\end{eqnarray}
Recalling that $|u(t-\tau)| \leq \| v \|_{c}$ and using Young's inequality (\ref{you}), the latter becomes:
\begin{eqnarray}
\dfrac{d}{dt} \|u_{xx}(t)\|^2 &\leq & \dfrac{\nu^2}{2 \varepsilon_1}  \|u_{xx} (t)\|^2 +  \dfrac{\| v \|_{c}^2}{2 \varepsilon_2}  \|u_{x} (t)\|^2 +2 (\varepsilon_1+\varepsilon_2+\varepsilon_3- \mu) \|u_{xxxx} (t)\|^2 \nonumber\\
&& +\dfrac{\|a\|_{\infty}}{2 \varepsilon_3}  \|u (t)\|^2,
 \label{ind}
\end{eqnarray}
for any positive constant $\varepsilon_i$, $i=1,2,3$. It suffices now to choose $\varepsilon_i$ so that the coefficient of $\|u_{xxxx} (t)\|$ vanishes (for instance $\varepsilon_1=\varepsilon_2=\mu/3$) and then invoke (\ref{wi1})-(\ref{wi2}) to get:
$$
\dfrac{d}{dt} \|u_{xx}(t)\|^2 \leq \dfrac{3}{2 \mu}
\left( \nu^2+\frac{\ell^2}{\pi^2}  \| v \|_{c}^2 + \frac{\ell^4}{\pi^4 \|a\|_{\infty}^2 } \right) \|u_{xx} (t)\|^2,
$$
which yields
$$
\|u_{xx}(t)\| \leq L(\| v \|_{c}),
$$
where $L$ is a positive constant depending on $\| v \|_{c}$ and the system parameters. Finally, it amounts to repeating the above argument to show that $\|u_{xx}(t)\| \leq L (n,\| v \|_{c})$, for $t \in [n \tau, (n+1)\tau], \, n \in \mathbb{N}.$
\end{proof}

\begin{remark}
\label{rem1}
The reader can easily check that the well-posedness result stated in Theorem \ref{t1} remains valid even if $a(x)$ is identically zero. In turn, the exponential stability result requires a positive function as it will be shown in the next section.
\end{remark}

\section{Exponential stability } \label{exp}
\setcounter{equation}{0}
This section is devoted to the exponential stability result of solutions to (\ref{1}).

The following lemma will play an important role in the proof of stability result:

\begin{lemma} \cite{liu}
 Let $g, \, h$ and $y$ be three positive integrable functions on $(0,T)$. If $y'$ is integrable on $(0,T)$ such that:
\[
\begin{array}{cc}
  y'(t) \leq g(t) y(t)+h(t), \quad \forall t \in [0,T], &  \\[3mm]
  \int_{0}^{T} g(r)~dr \leq c_1, \quad
  \int_{0}^{T} e^{m s} g(r)~dr \leq c_2, \quad
 \int_{0}^{T} e^{m s} y(r)~dr \leq c_3,
\end{array}
\]
\label{lem}
\end{lemma}
for some positive constants $m,\, c_1, \, c_2$ and $c_3$, then
\[ y(t) \leq (c_2 + m c_3 + y(0)) e^{c_1-mt}, \quad \forall t \in [0,T].\]

Our stability result is
\begin{theorem}
Let $v=v(x,s) \in C([-\tau,0]; \, H_0^2 (I))$ be an initial condition.  Then, there exist positive constants $M, \, \hat{\tau}$ and $\tilde{\omega}$ such that for any  time-delay $\tau < \hat{\tau}$, the unique mild solution of the problem (\ref{1}) satisfies
\[ \|u_{xx}\|^2 \leq \dfrac{M^2}{4} e^{-\tilde{\omega} t}, \quad \forall t >0,\]
 provided that $\nu$ is sufficiently small and  $a(\cdot) \in C^4 (I)$ such that for some positive constant $a_0$ and for all $x \in [0,\ell]$, we have:  $a(x) > a_0$, $a''(x) \leq 0, \; a^{(4)} (x) \geq 0$.
\label{theo2}
\end{theorem}

\begin{proof} For sake of clarity, we shall proceed by steps. First, we define
$$T_0=\sup \{\kappa; \, \|u_{xx} (t)\| \leq M, \; \text{for all} \, t \in [0,\kappa]  \}.$$
The main objective is to show that $T_0=\infty$. If this claim were not true, then
\begin{equation}
\|u_{xx} (t)\| \leq M, \; \text{for all} \, t \in [-\tau,T_0] \;\; \text{and} \; \|u_{xx} (T_0)\| = M.
\label{con}
\end{equation}
\noindent {\bf Step 1:} First, take the  inner product of (\ref{1}) in $L^{2} (I)$  with  $u$ and integrate by parts. Then, use the boundary conditions to obtain for any $ t \in [-\tau, 0]$:
\begin{equation}
\dfrac{1}{2} \dfrac{d}{dt} \|u(t)\|^2 + \nu \|u_{x} (t)\|^2 + \mu \|u_{xx} (t)\|^2 + \|\sqrt{a(x)} u(t)\|^2 + \int _{0}^{\ell}  \left( u(t-\tau)  - u(t)\right) u(t) u_{x} (t) ~dx=0.
\label{in1}
\end{equation}
 In turn, we have thanks to the estimate $u(x,t) \leq \sqrt{\ell} \|u_x(t)\|$ and Cauchy-Schwarz inequality
 \begin{eqnarray}
 \int _{0}^{\ell}  \left( u(t-\tau)  - u(t)\right) u(t) u_{x} (t) ~dx
 &\leq& \sqrt{\ell} \|u_x(t)\|^2 \| u(t-\tau)  - u(t)\|\nonumber\\
 &\leq& \sqrt{\tau \ell} ~ \|u_x(t)\|^2 \left( \int _{t-\tau}^{t}  \|u_r(r)\|^2\, dr ~dx \right)^{1/2}.
 \label{in2}
 \end{eqnarray}
 Inserting (\ref{in2}) into (\ref{in1}) yields
\begin{equation}
\dfrac{1}{2} \dfrac{d}{dt} \|u(t)\|^2 + \nu \|u_{x} (t)\|^2 + \mu \|u_{xx} (t)\|^2 +\|\sqrt{a(x)} u(t)\|^2 \leq  \sqrt{\tau \ell}~ \|u_x(t)\|^2 \left( \int _{t-\tau}^{t}  \|u_r(r)\|^2 ~dr \right)^{1/2}.
\label{in3}
\end{equation}

\noindent {\bf Step 2:} The task ahead is to estimate $\int _{t-\tau}^{t}  \|u_r(r)\|^2 ~dr$. To do so, integrating by parts and using (\ref{1}), we have:
$$
\dfrac{\mu}{2} \dfrac{d}{dt} \|u_{xx}(t)\|^2 + \|u_{t} (t)\|^2 =
-\int _{0}^{\ell}   u_t (t) u(t-\tau)  u_{x} (t) ~dx
+\nu \int _{0}^{\ell}  u_t (t) u_{xx} (t) ~dx-\int _{0}^{\ell} a(x) u_t (t) u (t) ~dx,
$$
which implies that
 \begin{eqnarray}
&& \dfrac{\mu}{2} \|u_{xx}(t)\|^2- \dfrac{\mu}{2} \|u_{xx}(t-\tau)\|^2
+\int _{t-\tau}^{t}  \|u_{r} (r)\|^2 ~dr =
- \int _{t-\tau}^{t} \int _{0}^{\ell}   u_r (r) u(r-\tau)  u_{x} (r) ~dx dr\nonumber\\
&& +\nu \int _{t-\tau}^{t} \int _{0}^{\ell}  u_r (r) u_{xx} (r) ~dx dr-\nu \int _{t-\tau}^{t} \int _{0}^{\ell} a(x) u_r (r) u (r) ~dx dr.
 \label{in4}
\end{eqnarray}
In light of Cauchy-Schwarz inequality  as well as Young's inequality (\ref{you}),  and the fact that $\|u_{xx}(t)\| \leq M$, for any $t \in [-\tau,T_0]$, we obtain
 \begin{eqnarray*}
\int _{t-\tau}^{t} \int _{0}^{\ell}   u_r (r) u(r-\tau)  u_{x} (r) ~dx dr
&\leq& \dfrac{\ell M^2}{4 \delta_1}  \int _{t-\tau}^{t} \|u_{x} (r)\|^2 ~dr +
\delta_1  \int _{t-\tau}^{t} \|u_{r} (r)\|^2 ~dr,
\end{eqnarray*}
for any $\delta_1 >0$. This, together with (\ref{wi2}) and the boundedness of $\|u_{xx}(t)\|$, implies that
\begin{equation}\label{in5}
  -\int _{t-\tau}^{t} \int _{0}^{\ell}   u_r (r) u(r-\tau)  u_{x} (r) ~dx dr \leq
\dfrac{\ell^3 M^4 \tau}{4 \delta_1 \pi^2}  + \delta_1  \int _{t-\tau}^{t} \|u_{r} (r)\|^2 ~dr.
\end{equation}
Arguing as before, we also get:
\begin{equation}
 \nu \int _{t-\tau}^{t} \int _{0}^{\ell}  u_r (r) u_{xx} (r) ~dx dr \leq
\dfrac{\nu^2 M^2 \tau}{4 \delta_2}  + \delta_2  \int _{t-\tau}^{t} \|u_{r} (r)\|^2 ~dr, \quad \forall \delta_2 >0,
 \label{in6a}
\end{equation}
\begin{equation}
 \int _{t-\tau}^{t} \int _{0}^{\ell} a(x) u_r (r) u(r) ~dx dr \leq
\|a\|_{\infty} \dfrac{\ell^4 M^2 \tau}{4 \delta_3 \pi^4}  + \|a\|_{\infty}  \delta_3  \int _{t-\tau}^{t} \|u_{r} (r)\|^2 ~dr, \quad \forall \delta_3 >0.
 \label{in6}
\end{equation}
Amalgamating (\ref{in4})-(\ref{in6}), we have
\begin{equation}
2 (1-\delta_1-\delta_2 -\delta_3 \|a\|_{\infty} ) \int _{t-\tau}^{t} \|u_{r} (r)\|^2 ~dr \leq \dfrac{\ell^3 \tau}{2 \delta_1 \pi^2} M^4 +   \left(\mu+\dfrac{\nu^2 \tau}{2 \delta_2}+\|a\|_{\infty} \dfrac{\ell^4 M^2 \tau}{4 \delta_3 \pi^4} \right) M^2 -\mu \|u_{xx}(t)\|^2,
 \label{in7}
\end{equation}
for any $\delta_i>0$, $i=1,2,3$.

\noindent {\bf Step 3:} Inserting (\ref{in7}) into (\ref{in3}) gives
\[
\dfrac{d}{dt} \|u(t)\|^2 +2 \mu \|u_{xx} (t)\|^2+ \|\sqrt{a(x)} u(t)\|^2  \leq
\]
\begin{equation}
\left\{-2\nu +\sqrt{\dfrac{2\tau \ell}{1-\delta_1-\delta_2-\delta_3 \|a\|_{\infty}} \left( \dfrac{\ell^3 \tau}{2 \delta_1 \pi^2} M^4 +   \left[\mu+\dfrac{\nu^2 \tau}{2 \delta_2}+\|a\|_{\infty} \dfrac{\ell^4 M^2 \tau}{2 \delta_3 \pi^4} \right] M^2 \right)}\right\}\|u_x(t)\|^2 ,
\label{in8}
\end{equation}
where the positive constants $\delta_1, \, \delta_2$ must satisfy  $\delta_1+\delta_2 +\delta_3 \|a\|_{\infty} <1$. For instance, one can pick up $\delta_1=\delta_2 =1/6$ and $\delta_3 =1/(6\|a\|_{\infty})$, which transforms (\ref{in8}) as follows
\begin{equation}
\dfrac{d}{dt} \|u(t)\|^2 \leq -2 \omega \|u_x(t)\|^2 -2 \mu \|u_{xx} (t)\|^2,
\label{in9}
\end{equation}
in which $\omega =\nu -\sqrt{{\tau \ell} \left( \dfrac{3\ell^3 \tau}{\pi^2} M^4 +   \left[\mu+3\nu^2 \tau +3 \|a\|_{\infty}^2  \dfrac{\ell^4 \tau}{\pi^4} \right] M^2 \right)} >0$ provided that $\tau \in (\tau_1, \tau_2)$, where
\begin{equation}
\left\{
\begin{array}[c]{ll}
\tau_1=\dfrac{-\mu \ell M^2 -\sqrt{\mu^2 \ell^2 M^4 + 12 \nu^2 \left( \dfrac{\ell^4 }{\pi^2} M^4 +\left[\nu^2 \ell + \|a\|_{\infty}^2
\dfrac{\ell^5 }{\pi^4}\right] M^2  \right)}}{6 \left( \frac{\ell^4}{\pi^2} M^4 +\left[ \nu^2 \ell + \|a\|_{\infty}^2
\dfrac{\ell^5 }{\pi^4}\right] M^2 \right)} <0, \vspace{4mm}\\
\tau_2=\dfrac{-\mu \ell M^2 +\sqrt{\mu^2 \ell^2 M^4 + 12 \nu^2 \left( \dfrac{\ell^4 }{\pi^2} M^4 +\left[\nu^2 \ell + \|a\|_{\infty}^2
\dfrac{\ell^5 }{\pi^4}\right] M^2  \right)}}{6 \left( \frac{\ell^4}{\pi^2} M^4 +\left[\nu^2 \ell + \|a\|_{\infty}^2
\dfrac{\ell^5 }{\pi^4}\right] M^2 \right)} >0.
\end{array}
\right. \label{tau}%
\end{equation}

\noindent {\bf Step 4:}
Going back to (\ref{in9}) and using (\ref{wi1}), we obtain:
\begin{equation}
\dfrac{d}{dt} \|u(t)\|^2 \leq -2 \tilde{\omega} \|u(t)\|^2 -2 \tilde{\omega} \|u_{xx} (t)\|^2,
\label{in10}
\end{equation}
where $\tilde{\omega}=\min\{ \omega (\pi/\ell)^2, \mu \}$, which, on one hand, implies that
\begin{equation}
\|u(t)\| \leq e^{-\tilde{\omega}t} \|v (0)\|, \quad t \in [0,T_0].
\label{in11}
\end{equation}
On the other hand, (\ref{in10}) yields
\begin{equation}
\dfrac{d}{dt} \left(  e^{\tilde{\omega}t} \|u(t)\|^2 \right) \leq
\tilde{\omega} e^{\tilde{\omega}t}   \|u(t)\|^2 -2 \tilde{\omega}  e^{\tilde{\omega}t}  \|u_{xx} (t)\|^2.
\label{in12}
\end{equation}
Combining (\ref{in11}) and (\ref{in12}), we get
\begin{equation}
\dfrac{d}{dt} \left(  e^{\tilde{\omega}t} \|u(t)\|^2 \right) + 2 \tilde{\omega}  e^{\tilde{\omega}t}  \|u_{xx} (t)\|^2 \leq
\tilde{\omega} e^{-\tilde{\omega}t}   \|v (0)\|^2 .
\label{in13}
\end{equation}
A simple integration of (\ref{in13}) over $[0,T_0]$ and the utilization of (\ref{in11}) gives the following estimate:
\begin{equation}
\int_{0}^{T_0}  e^{\tilde{\omega}t}  \|u_{xx} (t)\|^2~dt \leq
\left(\tilde{\omega}\right)^{-1}  \|v (0)\|^2 ,
\label{in14}
\end{equation}
which also gives by means of (\ref{wi2})
\begin{equation}
\dfrac{\pi^2}{\ell^2} \int_{0}^{T_0}  e^{\tilde{\omega}t}  \|u_{x} (t)\|^2~dt \leq
\left(\tilde{\omega}\right)^{-1}  \|v (0)\|^2 .
\label{in15}
\end{equation}
The ultimate outcome is to estimate $\int_{0}^{
T_0}  e^{\tilde{\omega}t}  \|u_{x} (t-\tau)\|^2~dt$. To proceed, we have:
 \begin{eqnarray*}
\int_{0}^{T_0}  e^{\tilde{\omega}t}  \|u_{x} (t-\tau)\|^2~dt &=&
\int _{-\tau}^{0} e^{\tilde{\omega}(r+\tau)} \|u_{x} (r)\|^2~ dr + \int _{0}^{T_0} e^{\tilde{\omega}(r+\tau)} \|u_{x} (r)\|^2~ dr \\
&+& \int _{T_0}^{T_0-\tau} e^{\tilde{\omega}(r+\tau)} \|u_{x} (r)\|^2~ dr  \\
&\leq&
\int _{-\tau}^{0} e^{\tilde{\omega}(r+\tau)} \|u_{x} (r)\|^2 dr + \int _{0}^{T_0} e^{\tilde{\omega}(r+\tau)} \|u_{x} (r)\|^2 dr .
\end{eqnarray*}
In light of (\ref{in15}), the last estimate gives the desired result:
\begin{equation}
\int_{0}^{T_0}  e^{\tilde{\omega}t}  \|u_{x} (t-\tau)\|^2~dt \leq e^{\tilde{\omega}\tau} \| v_{x}\|_{\tau}^2
+ \ell^2 \left(\pi^2 \tilde{\omega}\right)^{-1} e^{\tilde{\omega}\tau} \|v (0)\|^2.
\label{in16}
\end{equation}
where $\| v_{x}\|_{\tau}^2= \int _{-\tau}^{0}  \|   v_{x} (r)\|^2~ dr$.

\noindent {\bf Step 5:} The main concern now is to show that
\begin{equation}\label{fin}
  \dfrac{d}{dt} \|u_{xx}(t)\|^2 \leq \gamma  \|u_x(t-\tau)\|^2  \|u_{xx} (t)\|^2,
\end{equation}
for some positive constant $\gamma$. To do so, we first rewrite (\ref{sam}) as follows:
\begin{eqnarray}
\dfrac{d}{dt} \|u_{xx}(t)\|^2 &=&2 \nu \int _{0}^{\ell} u_{xx}(t) u_{xxxx} (t) ~dx -2 \int _{0}^{\ell}  u(t-\tau)  u_x(t) u_{xxxx} (t) ~dx-2 \mu \|u_{xxxx} (t)\|^2 \nonumber\\
&& -2 \int _{0}^{\ell} a(x) u(t) u_{xxxx} (t) ~dx.
\label{17}
\end{eqnarray}
Next, it follows from Cauchy-Schwarz inequality that $|u(t-\tau)| \leq \sqrt{\ell} \|u_x(t-\tau)\|$. This, together with (\ref{wi2}), (\ref{you}) and (\ref{17}), yields
\[
\dfrac{d}{dt} \|u_{xx}(t)\|^2 \leq \left( \dfrac{\nu^2}{2 \epsilon_1}+\dfrac{\ell}{2 \epsilon_2}-2\mu \right) \|u_{xxxx} (t)\|^2 + 2 \dfrac{\epsilon_2 \pi^2}{\ell^2} \| u_x(t-\tau)\|^2  \|u_{xx} (t)\|^2
\]
\[
+2  \epsilon_1  \|u_{xx} (t)\|^2 -2 \int _{0}^{\ell} a(x) u(t) u_{xxxx} (t) ~dx,
\]
for any positive constants $\epsilon_1$ and $\epsilon_2$. In view of the properties of  $a(x)$ and simple integration by parts, the latter gives
\[
\dfrac{d}{dt} \|u_{xx}(t)\|^2 \leq \left( \dfrac{\nu^2}{2 \epsilon_1}+\dfrac{\ell}{2 \epsilon_2}-2\mu \right) \|u_{xxxx} (t)\|^2 + 2 \dfrac{\epsilon_2 \pi^2}{\ell^2} \| u_x(t-\tau)\|^2  \|u_{xx} (t)\|^2
\]
\begin{equation}
+2 (  \epsilon_1-\sqrt{a_0} ) \|u_{xx} (t)\|^2 + 4 \int _{0}^{\ell}  a''(x) u_x^2 (t) ~dx - \int _{0}^{\ell}  a^{(4)}(x) u^2 (t)~dx.
\label{18}
\end{equation} Lastly, one can choose $\epsilon_1 =p \sqrt{a_0}$, where $p$ is an arbitrary number in $(0,1]$  and then choose $\epsilon_2 =\mu - \frac{\nu^2}{4 p \sqrt{a_0}}$ so that (\ref{18}) leads to the desired inequality (\ref{fin}) with $\gamma=\frac{4 p \pi^2 \sqrt{a_0}}{\ell (4 p \mu \sqrt{a_0}-\nu^2)}$ provided that $\nu^2 < 4 p \mu \sqrt{a_0}$.

\noindent {\bf Step 6:}
Now, recalling (\ref{in14}), (\ref{in16}), (\ref{fin}) and using Lemma \ref{lem} with $y(t)=\|u_{xx}(t)\|^2 $, $g(t)=\gamma \| u_x(t-\tau)\|^2$, $h(t)=0$, $m=\tilde{\omega}$, $c_1=\gamma e^{\tilde{\omega}\tau}
\| v_{x}\|_{\tau}^2 + \gamma  \ell^2 \left(\pi^2 \tilde{\omega}\right)^{-1} e^{\tilde{\omega}\tau} \|v (0)\|^2 $, $c_2=0$ and $c_3=\left(\tilde{\omega}\right)^{-1}  \|v (0)\|^2$, we reach that for any $t \in [0,T]$:
\[
\|u_{xx}(t)\|^2 \leq \left( \|v(0)\|^2 + \|v_{xx} (0)\|^2 \right)
\exp \left[ \gamma e^{\tilde{\omega}\tau} \| v_{x}\|_{\tau}^2
+ \gamma  \ell^2 \left(\pi^2 \tilde{\omega}\right)^{-1} e^{\tilde{\omega}\tau} \|v (0)\|^2  \right] e^{-\tilde{\omega}t}
\]
which gives
\begin{equation}\label{kk}
\|u_{xx}(t)\|^2 \leq (M^2/4) e^{-\tilde{\omega}t}, \quad \forall t \in [0,T],
\end{equation}
 as long as
\[ M=\displaystyle \sup_{-\tau \leq s \leq 0} \|v_{x}(s)\|+4
\left[ \left[ \|v (0)\|^2 + \|v_{xx} (0)\|^2 \right] \exp \left( \gamma
\| v_{x}\|_{\tau}^2 + \frac{\gamma  \ell^2}{\pi^2}  \|v (0)\|^2  \right) \right]^{1/2},\]
and $\tau < \hat{\tau}=\min \{\sigma, \tau_2 \}$, where $\tau_2$ is given by (\ref{tau}) and
\[
\sigma= \sup \left\{ \kappa; \, \left[ \|v (0)\|^2 + \|v_{xx} (0)\|^2 \right] \exp \left[ \gamma e^{\tilde{\omega}\tau} \left( \| v_{x}\|_{\tau}^2
+  \ell^2 \left(\pi^2 \tilde{\omega}\right)^{-1} \|v (0)\|^2 \right) \right] \leq \frac{M^2}{4}, \; \forall \tau \in [0,\kappa]  \right\}.
\]
Taking $t=T$ in (\ref{kk}) and recalling that $\|u_{xx}(T)\|=M$, we finally reach the contradiction. Thereby, $T=\infty$ and also the conclusion of the theorem follows from (\ref{kk}).
\end{proof}

\begin{remark}
\label{rem2}
(i) There are many functions $a(x)$ satisfying the conditions of Theorem \ref{theo2}. For instance, one can take $a(x)=c_0$ ($c_0$ being any positive constant), which obviously satisfy the assumptions of Theorem \ref{theo2}. Furthermore, given a positive real number $b_0$, one can also choose  $a(x)=b_0+x$ or $a(x)=b_0+\sin(\pi x/{\ell})$, for $x \in [0,\ell]$. Then, it is easy to check that the assumptions of Theorem \ref{theo2} are fulfilled for such functions.

(ii) A careful look at the decay rate $\tilde{\omega}$ obtained in Theorem \ref{theo2} leads us to notice that the role of $a(x)$ is to accelerate the convergence of the zero solution. Indeed, the role of the damping term will be illustrated later in the numerical simulations section.
\end{remark}

\section{Numerical results}
 \label{num}
\setcounter{equation}{0}
 The aim of this section is to illustrate via numerical simulations the stability results of the time-delayed dispersive equation (1.1) {\bf{with}} and {\bf{without}} a presence of a time-delay. The main numerical simulation tool used in this section is  COMSOL Multiphysics software 5.4 which is based on the finite element method (FEM). Due to the sensitivity of the equation, an extra fine element mesh size is used, and the backward differentiation formula (BDF) as a numerical integrator with dt=0.001 is selected. The numerical solutions are computed for $\ell=1$ and for different values of $\nu$ and $\mu$, and for different functions $a(x)$.

 \subsection{The dispersive equation without a time-delay}
 In this subsection, we consider the dispersive equation without a time-delay, i.e., when $\tau=0$. In this case, the system (\ref{1}) reduces to the following:
 \begin{equation}
\label{41}
\left\{
\begin{array}{ll}
u_t(x,t) - \nu \, u_{xx}(x,t) + \mu \, u_{xxxx}(t,x)+  u(x,t)  u_x(x,t) +a(x) u (x,t)= 0,& (x,t) \in Q, \\
u(0,t) = u(1,t) = u_x (0,t) = u_x (1,t) =0,&  t > 0, \\
u(x,0) = u_0(x),&   x \in (0,1),
\end{array}
\right.
\end{equation}
 where $Q=(0,1) \times (0,+\infty)$, $\nu$ and $\mu$ are positive physical parameters, and $a(x) \in L^{\infty}(0,1)$. We study two cases: i) $a(x)$ is a non-positive function; ii) $a(x)$ is a positive function.

 \,\,\,{\bf{Case 1:}} {\bf{$a(x)$ is a non-positive function:}} let us take $\nu=0.01$ and $\mu=0.001$ and $u_0(x)=\sin( \pi x)$, and consider the following four different functions for $a(x)$, namely,  $a(x)=0$; $a(x)=-1$; $a(x)=-2$; and $a(x)=-3$. Figure 1 presents a 3-dimensional plot of the dynamics of the dispersive equation (\ref{41}) for these four  functions of $a(x)$. Figure (1a) indicates that the dynamics of the dispersive equation (\ref{41})  is exponential stable for the case  $a(x)=0$. This is verified by plotting the $L^2$-norms of the solutions $u(x,t)$ and $u_{xx}(x,t)$, $||u(x,t)||$ and $||u_{xx}(x,t)||$, respectively, versus time (see Figures (2a) and (2b)). The figures show that these norms converge exponentially to zero as $t \to \infty$.  However, when $a(x)$ is negative, the zero solution is unstable, and the dynamics of $u(x,t)$ converges to a nonzero steady-state solution (see Figures (1b)-(1d) and (2a)-(2b)). A careful look at the figures indicates that as the value of $a(x)$ decreases, the value of the nonzero steady state increases.

  \,\,\,{\bf{Case 2:}} {\bf{$a(x)$ is a positive function:}} In this case, we choose $\nu=0.01$ and $\mu=0.001$ and $u_0(x)=\sin( \pi x)$. Figure 3 presents a 3-dimensional plot of the dynamics of the dispersive equation (\ref{41}) for four different functions: $a(x)=1$; $a(x)=1+x$; $a(x)=1+\sin(\pi x)$; $a(x)=1+2x+\sin(2 \pi x)$. The figures indicate that for each of the chosen function $a(x)$, the dynamics of $u(x,t)$ is exponentially stable. This is verified by plotting the $L^2$-norms of the solutions: $||u(x,t)||$ and $||u_{xx}(x,t)||$, respectively, versus time (see Figures (4a) and (4c)). The figures show that these norms converge exponentially to zero as $t \to \infty$. Furthermore, Figures (4b) and (4d) depict semi-log plots of the $L^2$-norms, $||u(x,t)||$ and $||u_{xx}(x,t)||$ versus time. A careful look at the figures indicates that the curves of  these norms are indeed straight lines with different negative slopes. In addition, among the four selected functions of $a(x)$, the dynamics corresponding to the case when $a(x) = 1 + 2x + \sin(2 \pi x)$ has the fastest convergence rate; whereas, the dynamics corresponding to the case when $a(x) = 1$ has the slowest convergence rate. This of course is due to the fact that the function $a(x)=1+2x+\sin(2 \pi x)$ is the largest; whereas, $a(x)=1$ is the smallest among the other three function for $x \in (0,1)$.

\newpage
\vspace*{1.5in}
\begin{figure}[!h]
\centering
\subfigure[]{
\includegraphics[scale=0.4]{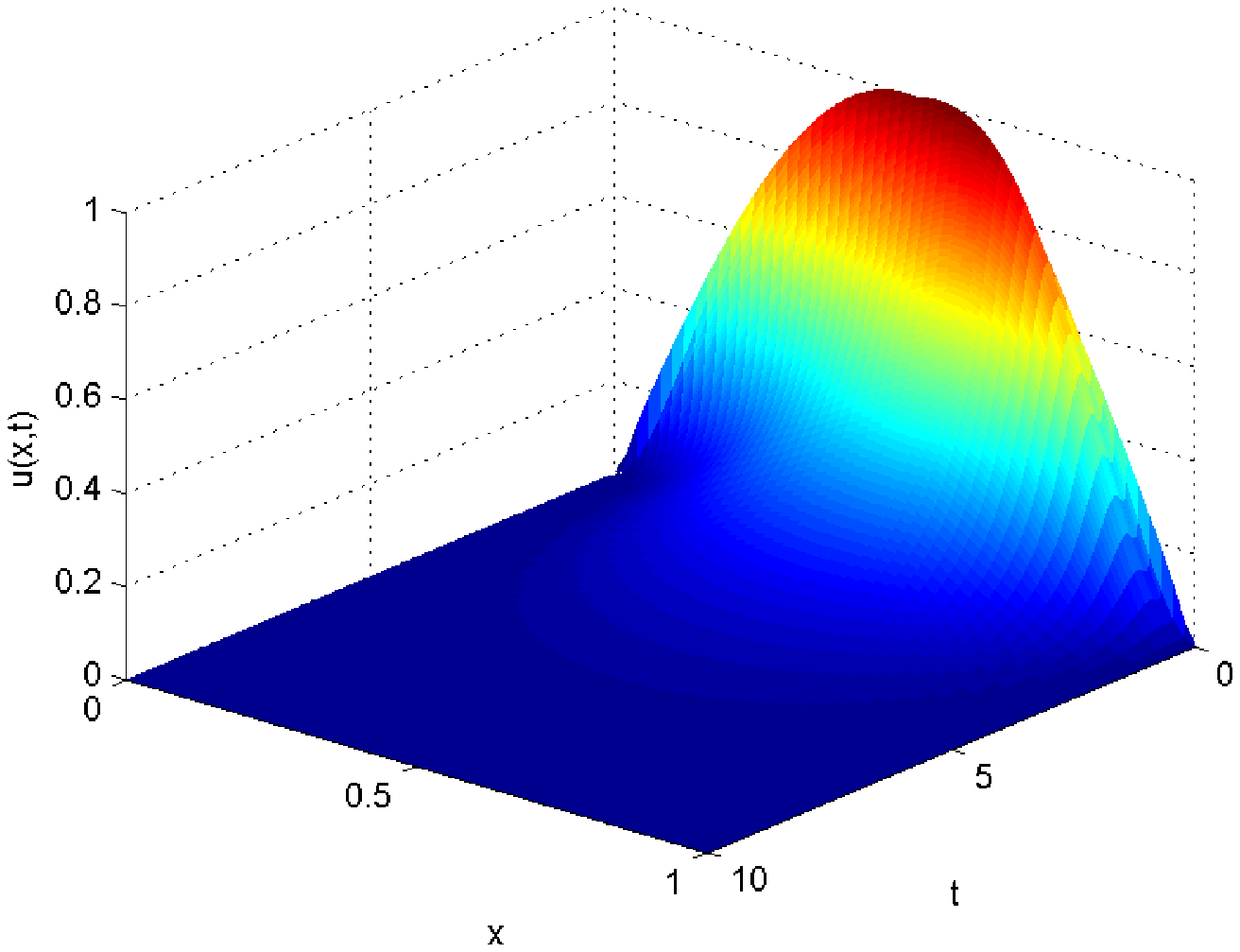}
\label{fig:subfiga}
}
\subfigure[]{
\includegraphics[scale=0.4]{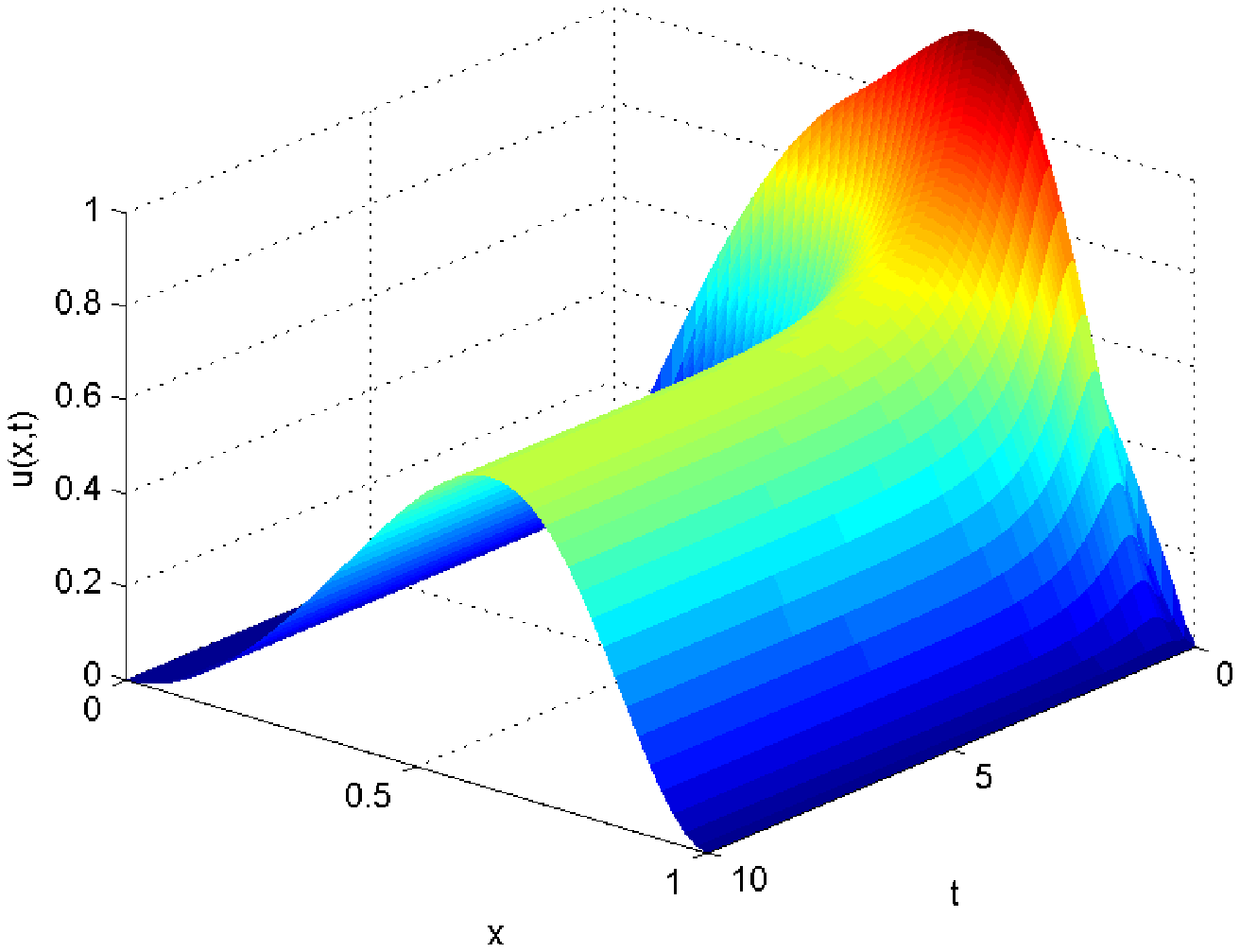}
\label{fig:subfigb}
}
\subfigure[]{
\includegraphics[scale=0.4]{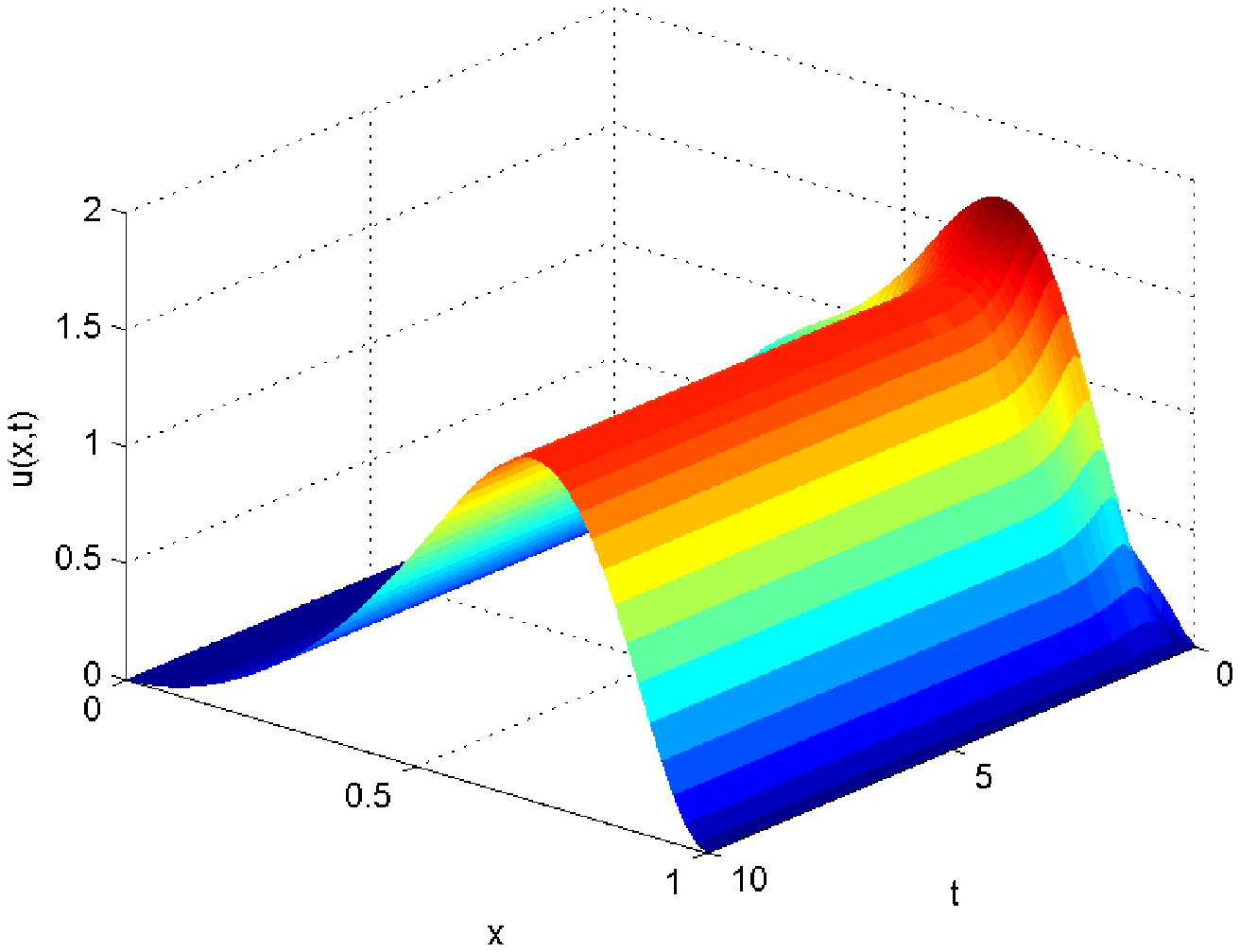}
\label{fig:subfigf}
}
\subfigure[]{
\includegraphics[scale=0.4]{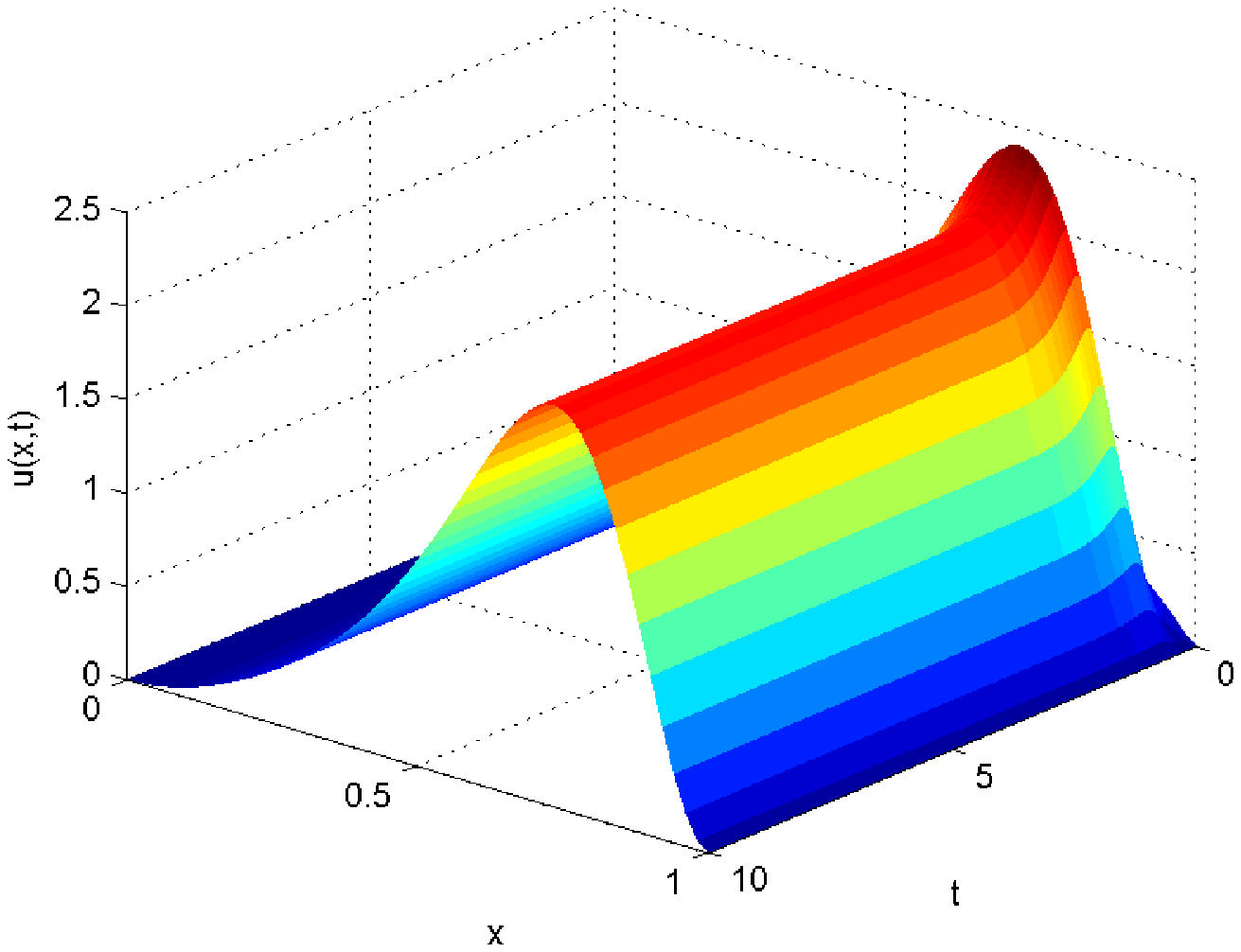}
\label{fig:subfige}
}
\caption{A 3-d landscape of the dynamics of the dispersive equation (\ref{41}) without delay (i.e., $\tau=0$) when $\nu=0.01$, $\mu=0.001$ and $u(x,0)=\sin (\pi x)$ for different functions $a(x)$; (a) $a(x)=0$; (b) $a(x)=-1$; (c) $a(x)=-2$; (d) $a(x)=-3$.}
\label{fig:Chapter4-03}
\end{figure}

\newpage
\vspace*{1.5in}
\begin{figure}[!h]
\centering
\subfigure[]{
\includegraphics[scale=0.4]{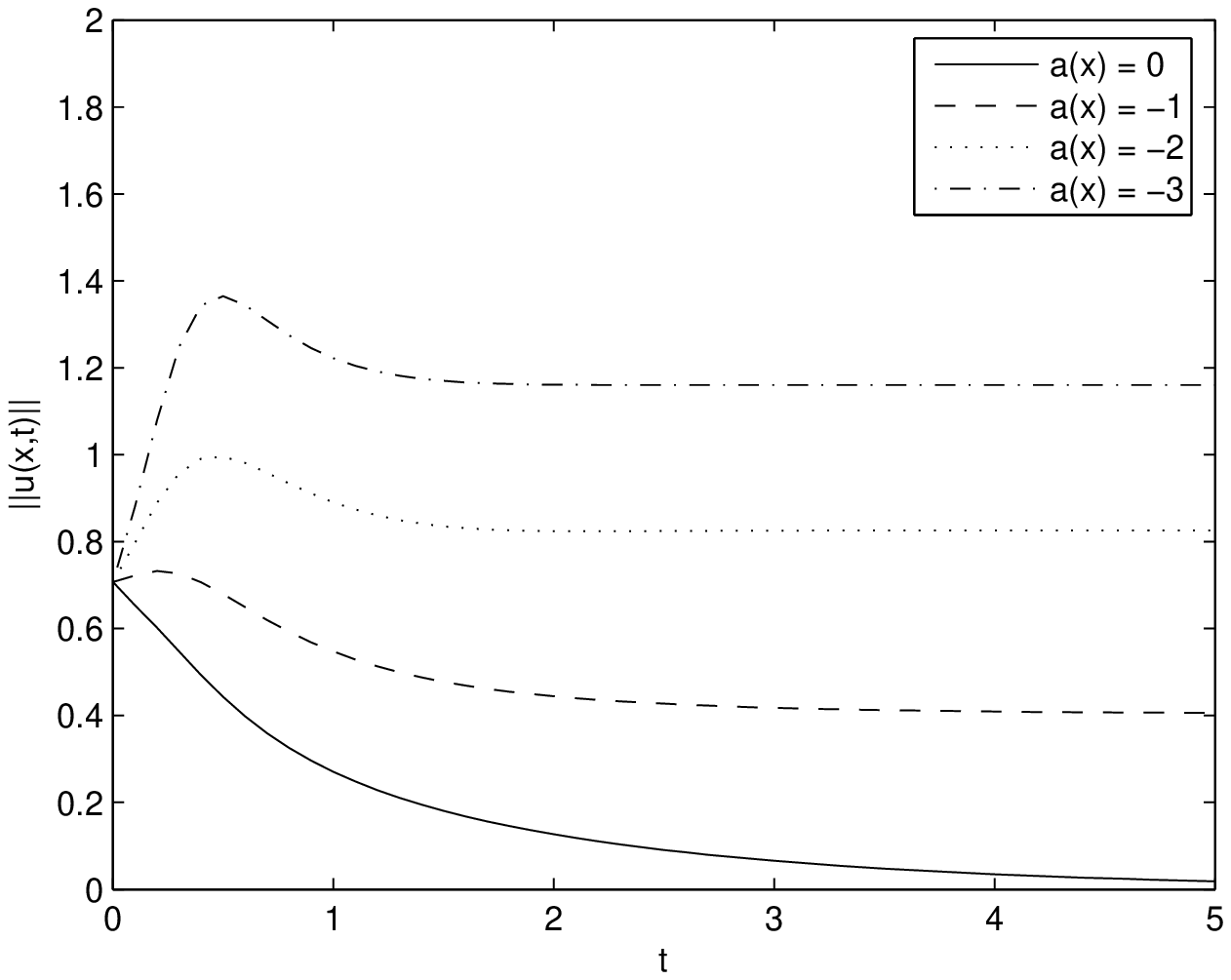}
\label{fig:subfiga}
}
\subfigure[]{
\includegraphics[scale=0.4]{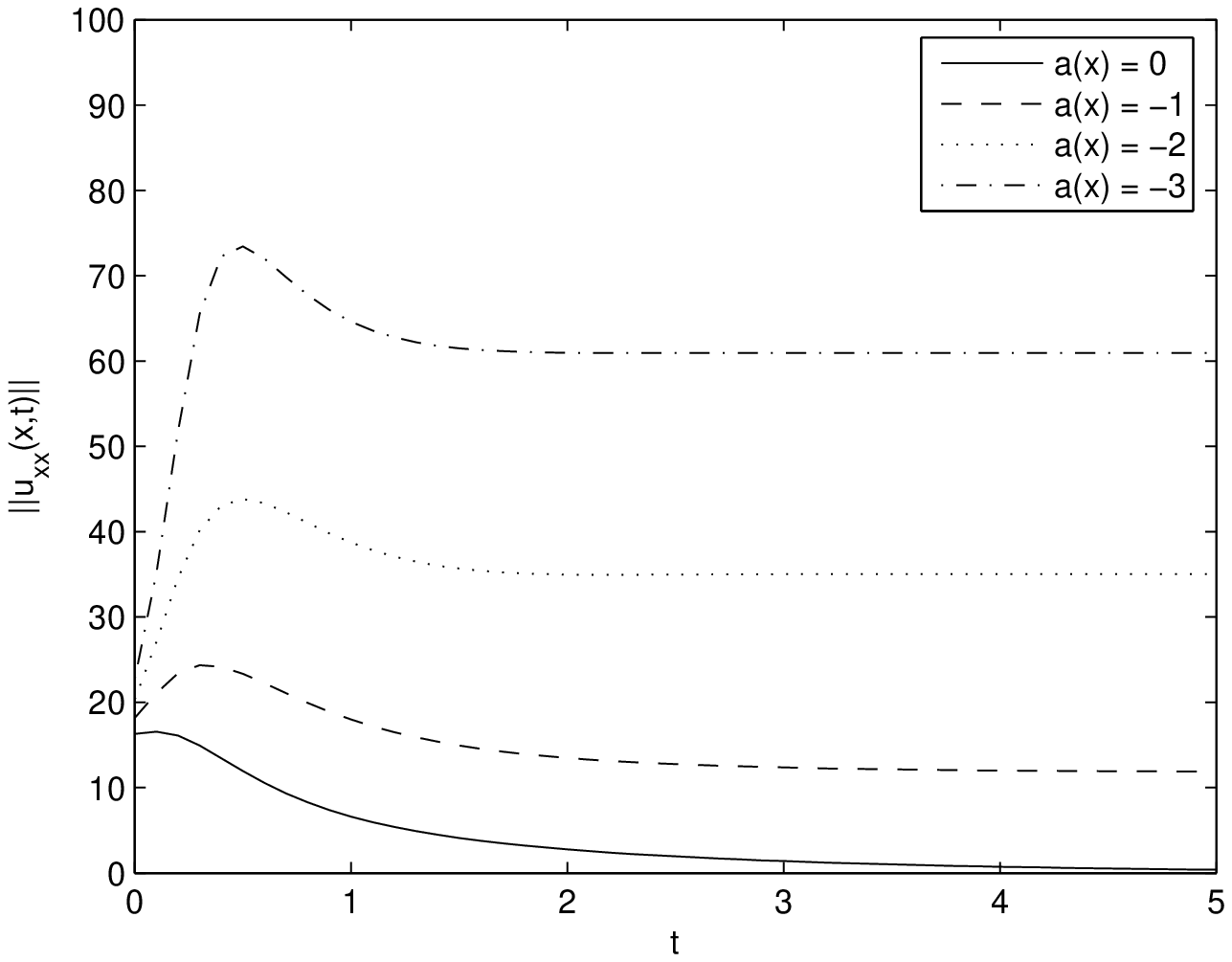}
\label{fig:subfigb}
}
\caption{The $L^2$-norms $||u(x,t)||$ and $||u_{xx}(x,t)||$  without time-delay (i.e., $\tau=0$); (a)  $||u(x,t)||$ vs. time for different  function $a(x)$; (b)   $||u_{xx}(x,t)||$ vs. time for different  function $a(x)$.}
\label{fig:Chapter4-03}
\end{figure}

\newpage
\vspace*{1.5in}
\begin{figure}[!h]
\centering
\subfigure[]{
\includegraphics[scale=0.4]{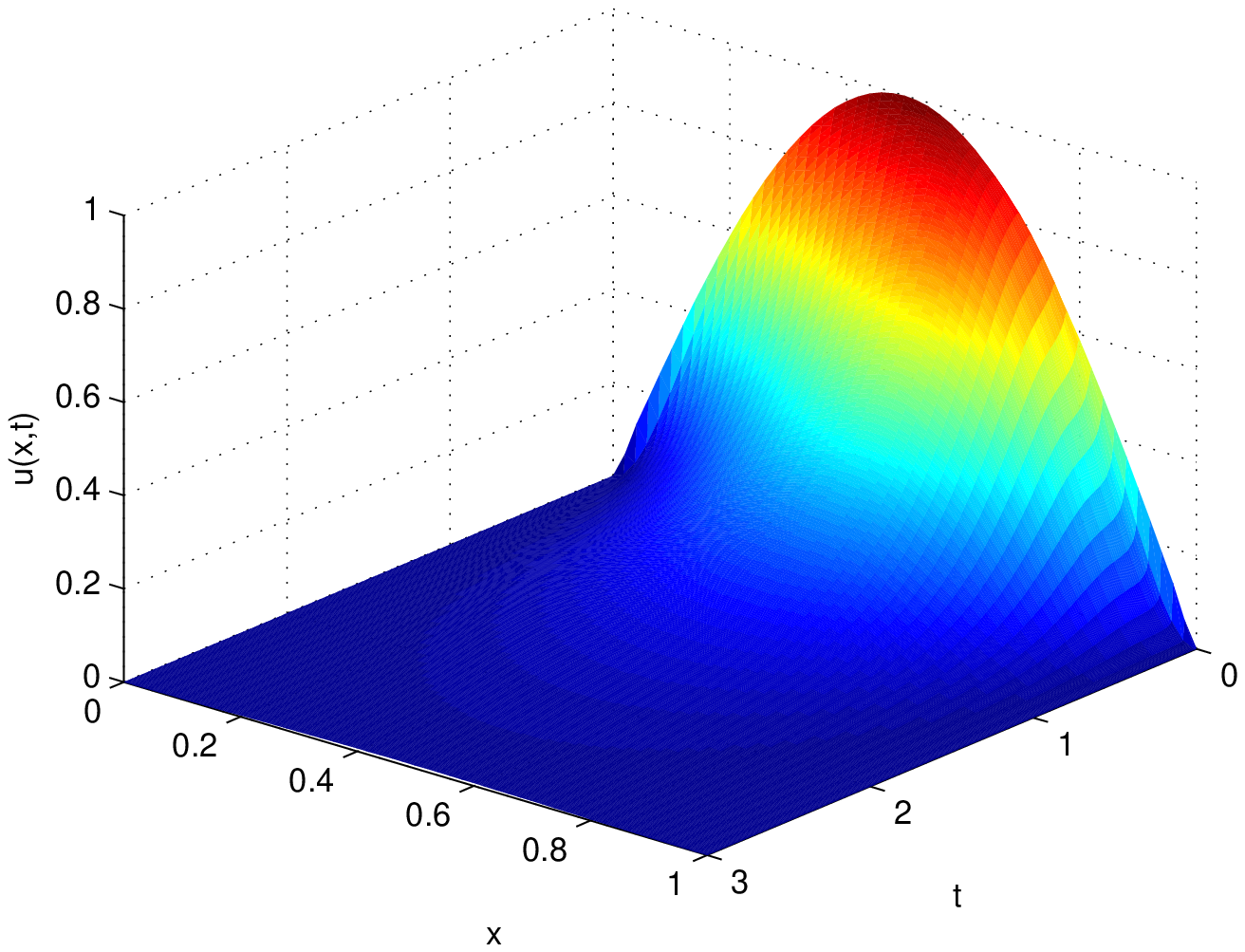}
\label{fig:subfiga}
}
\subfigure[]{
\includegraphics[scale=0.4]{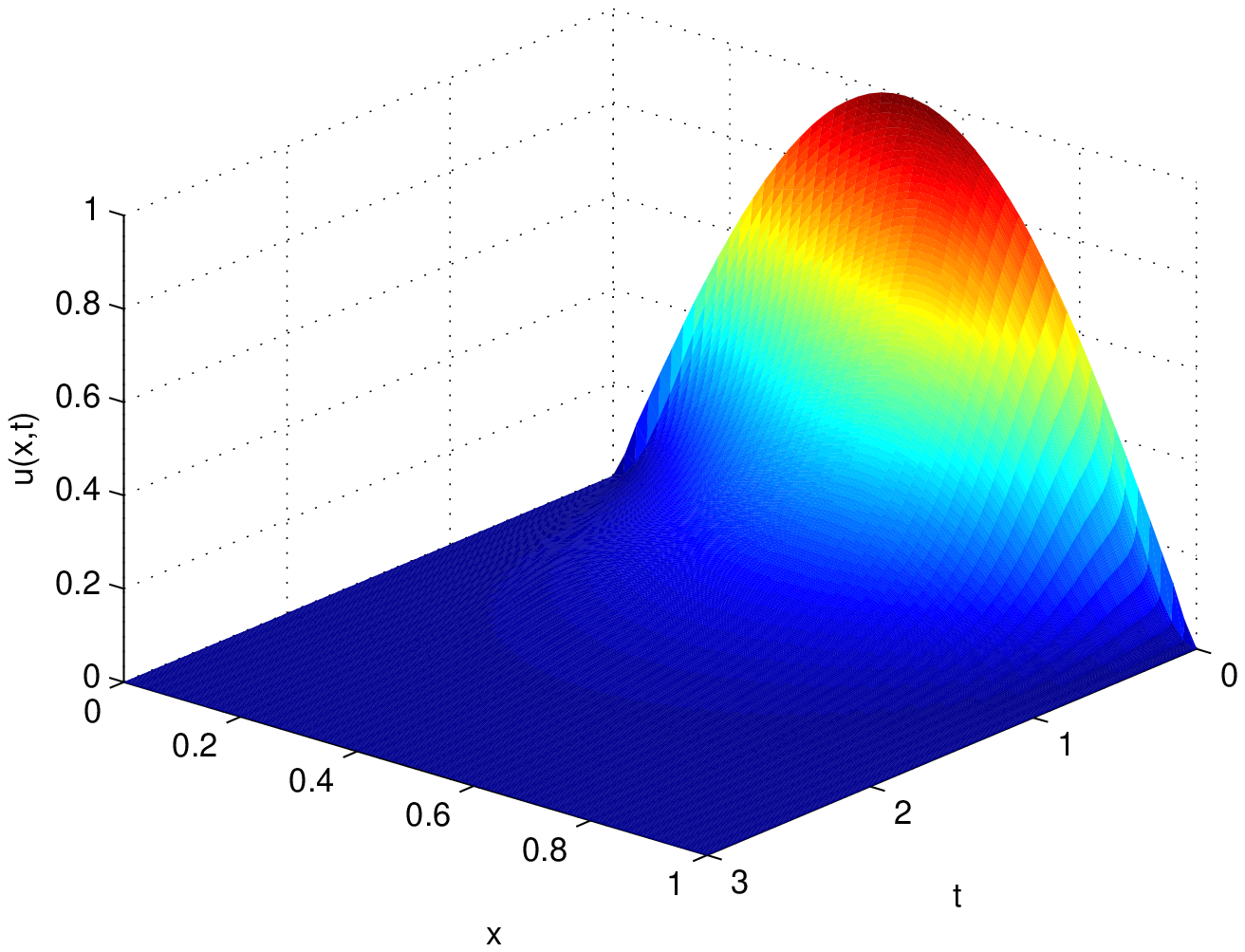}
\label{fig:subfigb}
}
\subfigure[]{
\includegraphics[scale=0.4]{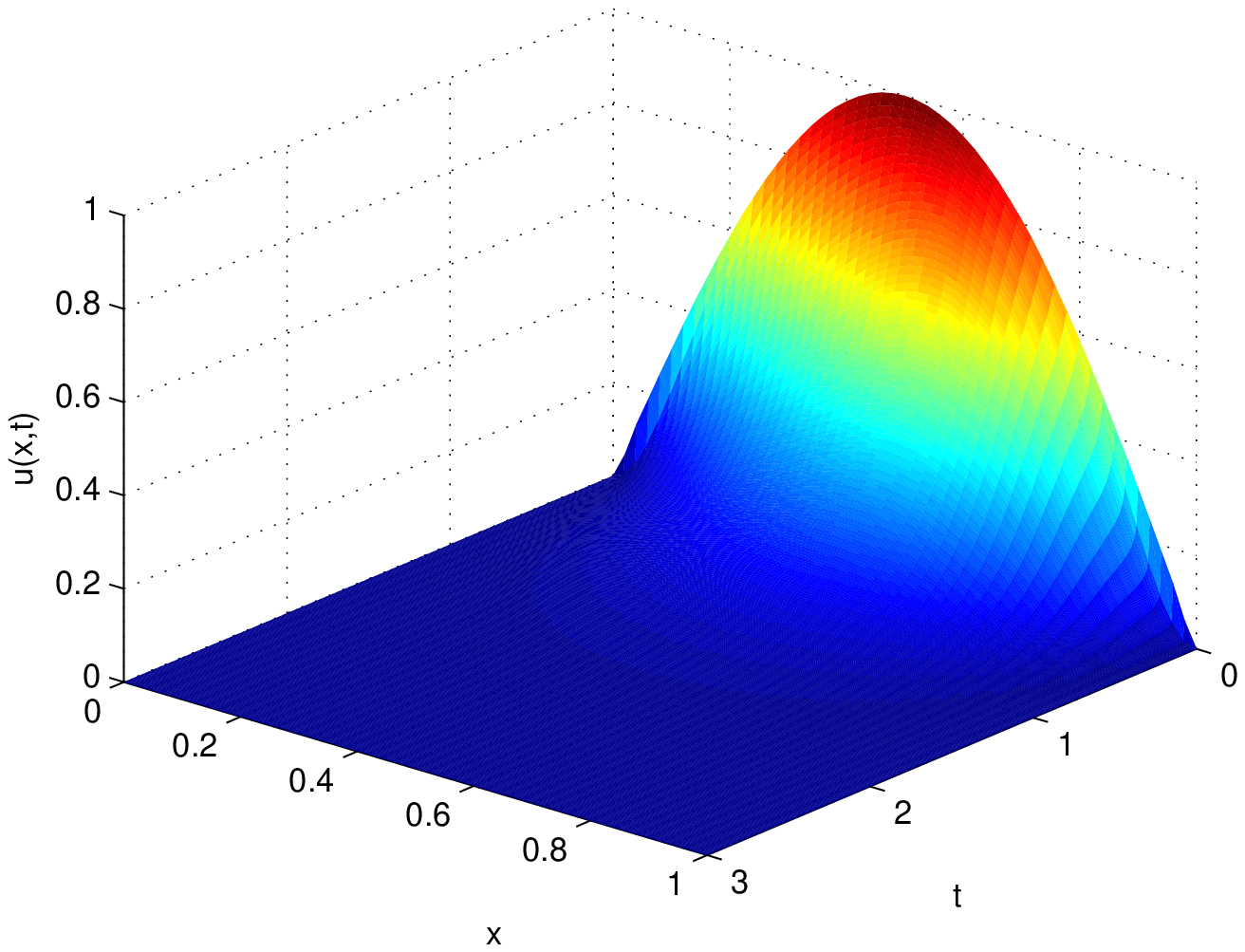}
\label{fig:subfigf}
}
\subfigure[]{
\includegraphics[scale=0.4]{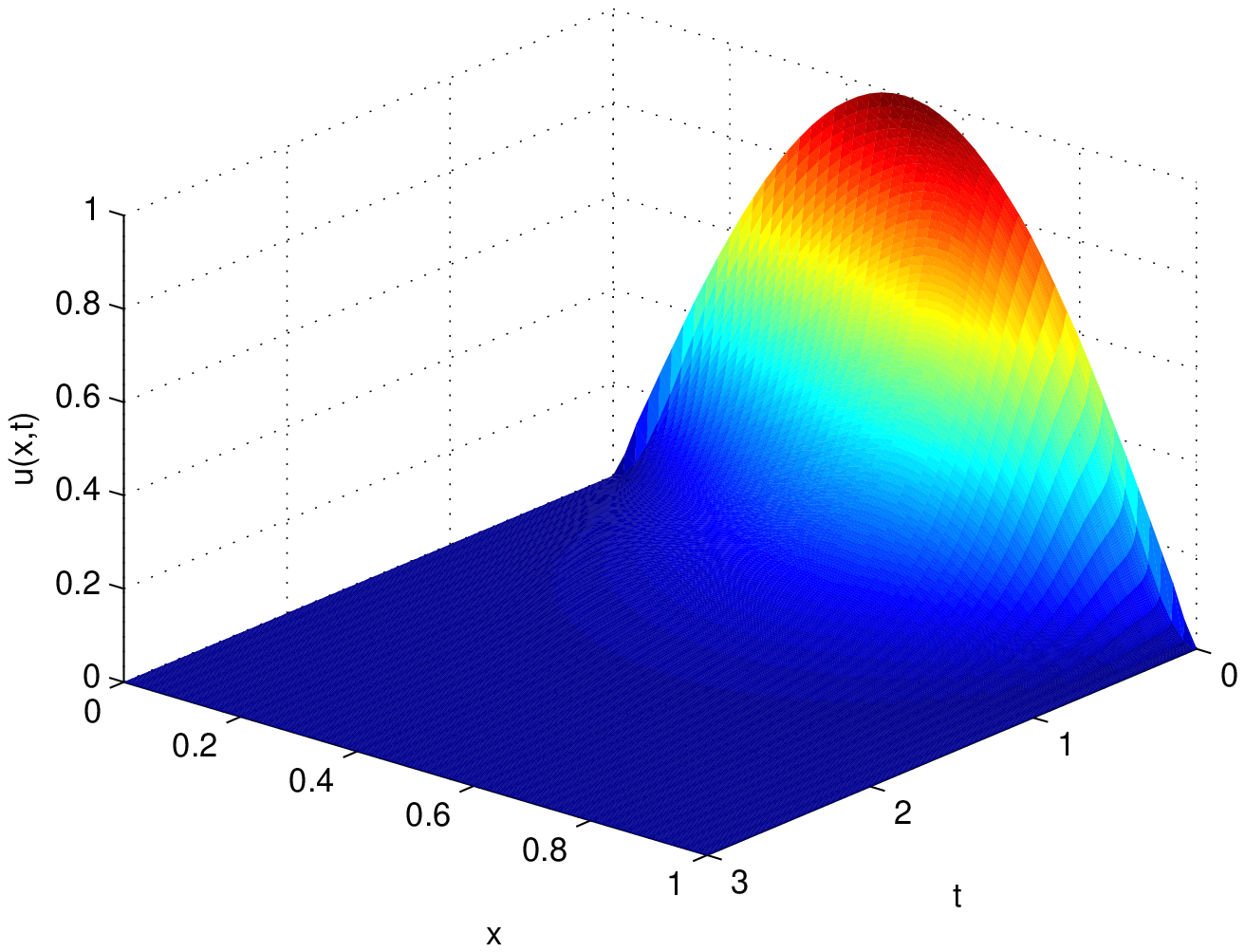}
\label{fig:subfige}
}
\caption{A 3-d landscape of the dynamics of the dispersive equation (\ref{41}) without delay (i.e., $\tau=0$) when $\nu=0.01$, $\mu=0.001$ and $u(x,0)=\sin (\pi x)$ for different functions $a(x)$; (a) $a(x)=1$; (b) $a(x)=1+x$; (c) $a(x)=1+\sin(\pi x)$; (d) $a(x)=1+2x+\sin(2\pi x)$.}
\label{fig:Chapter4-03}
\end{figure}

\newpage
\vspace*{1.5in}
\begin{figure}[!h]
\centering
\subfigure[]{
\includegraphics[scale=0.4]{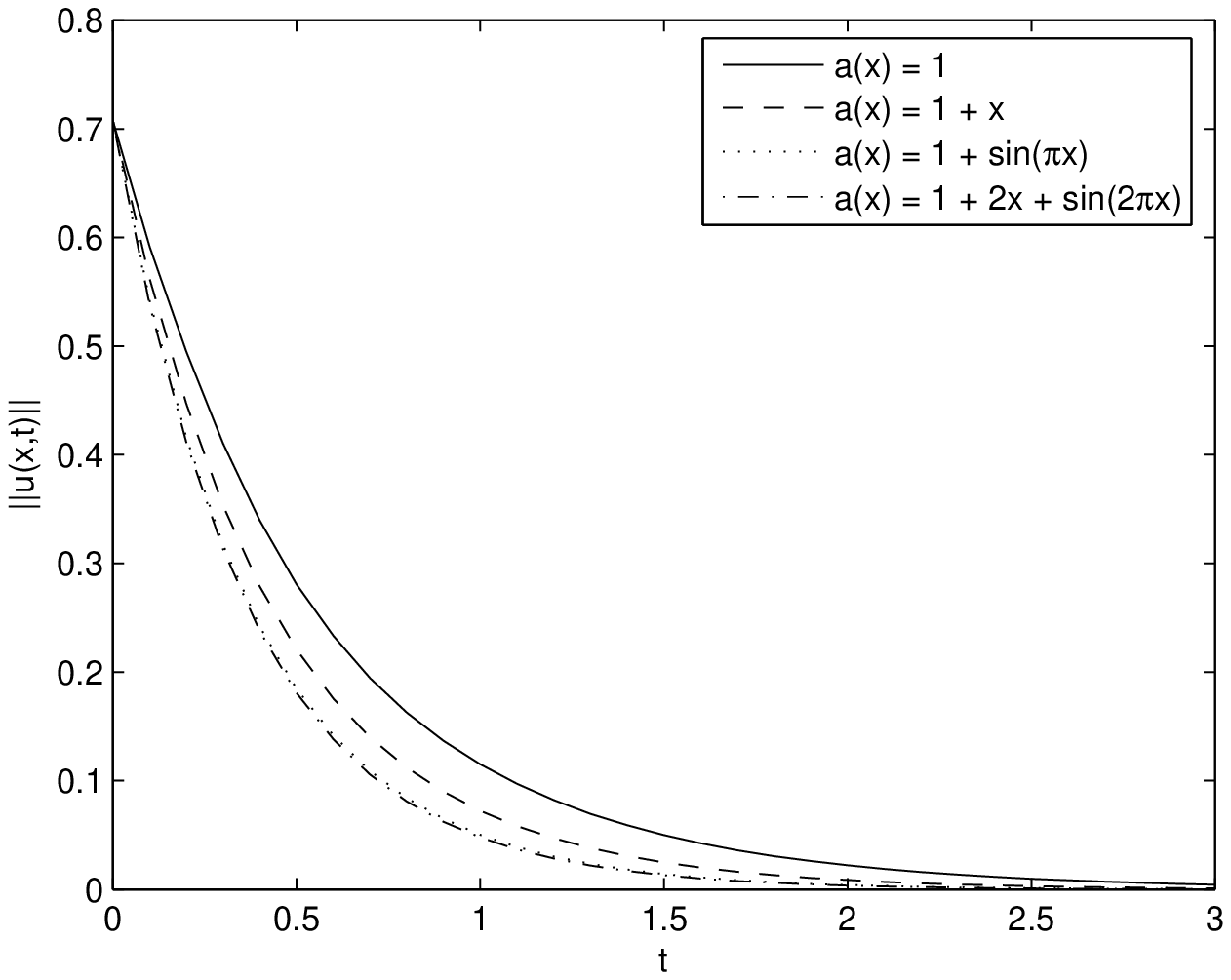}
\label{fig:subfiga}
}
\subfigure[]{
\includegraphics[scale=0.4]{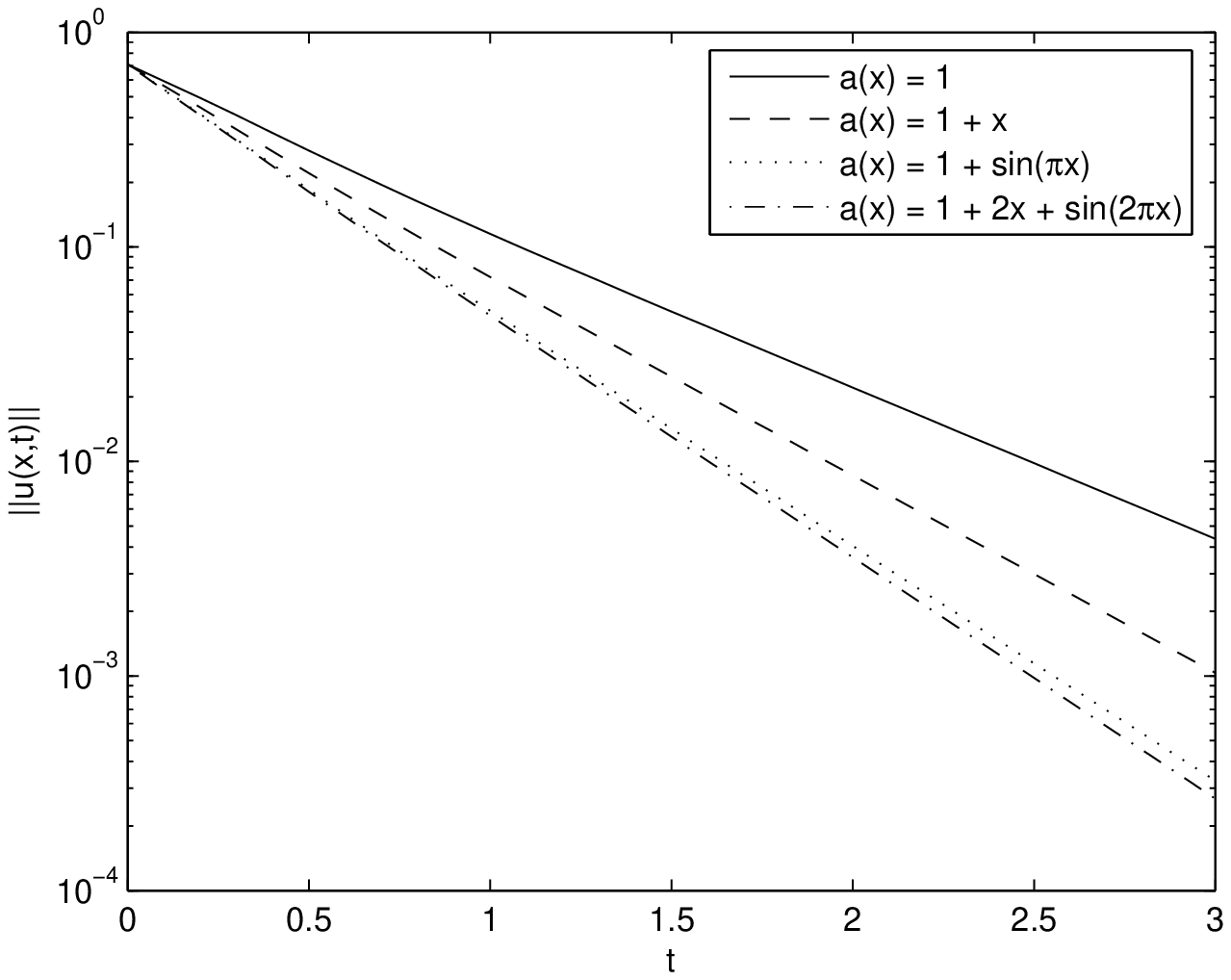}
\label{fig:subfigb}
}
\subfigure[]{
\includegraphics[scale=0.4]{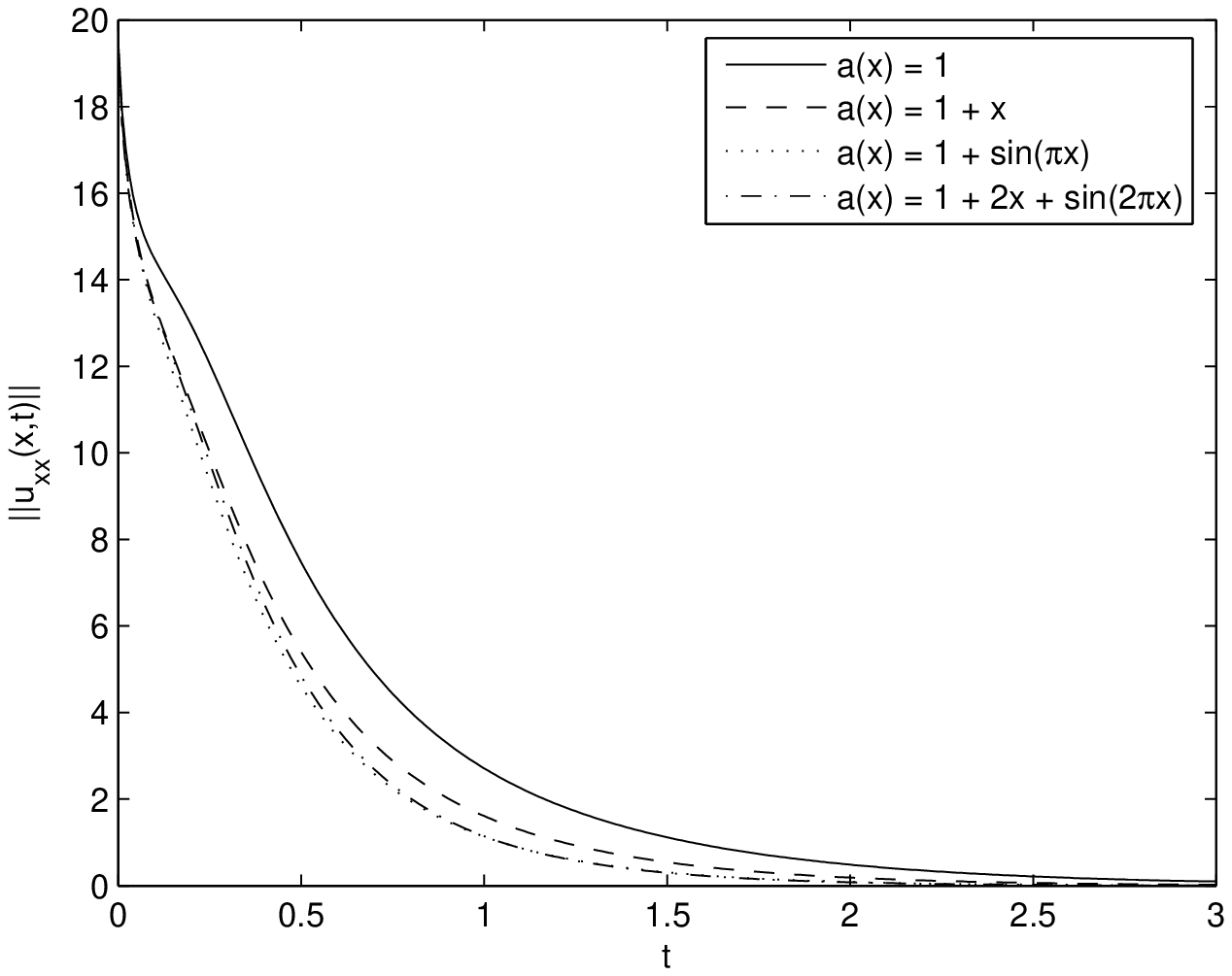}
\label{fig:subfigf}
}
\subfigure[]{
\includegraphics[scale=0.4]{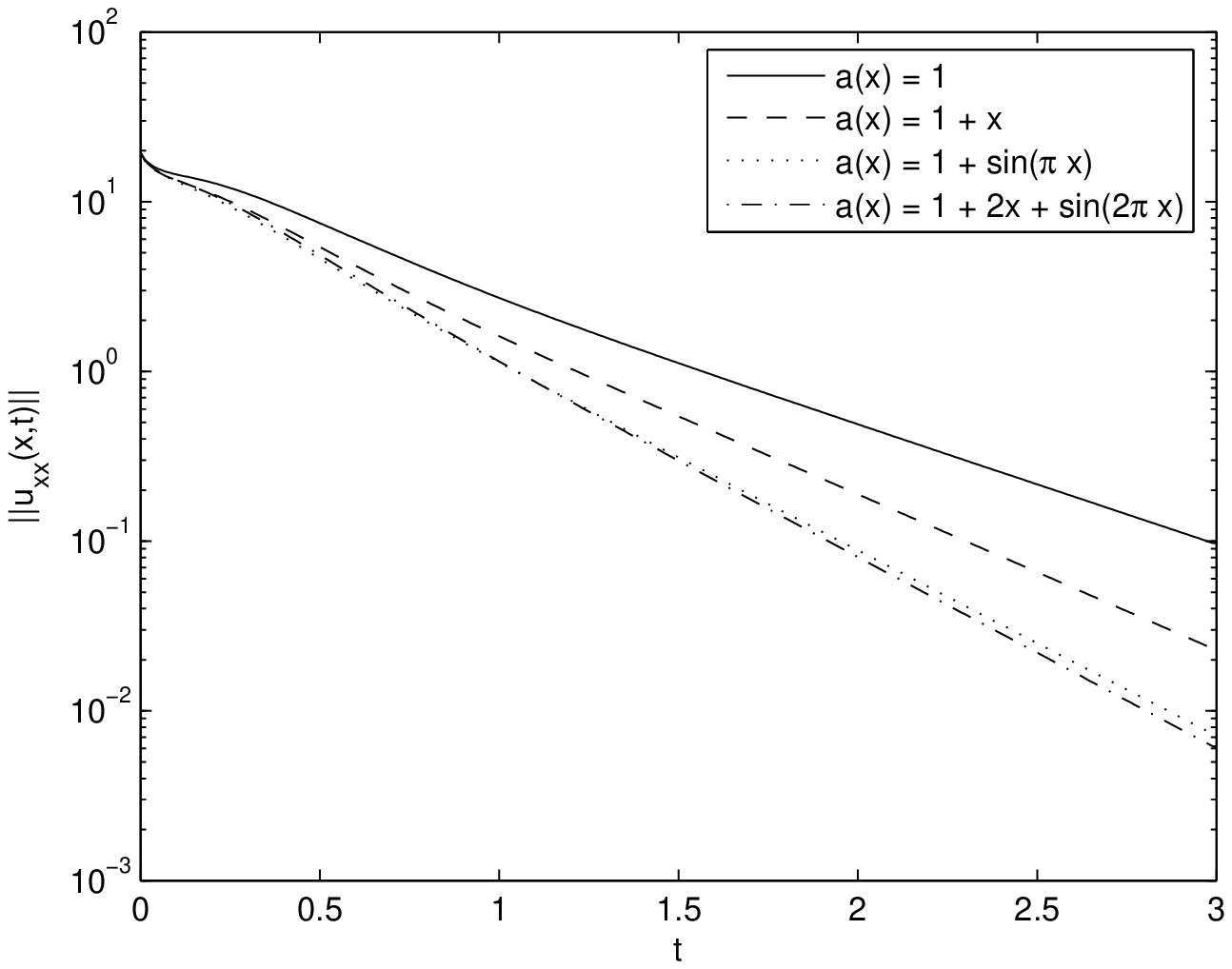}
\label{fig:subfige}
}
\caption{The $L^2$-norms $||u(x,t)||$ and $||u_{xx}(x,t)||$  without time-delay (i.e., $\tau=0$); (a)  $||u(x,t)||$ vs. time for different $a(x)$; (b) A semi-log plot of $||u(x,t)||$ vs. time for different $a(x)$; (c) $||u_{xx}(x,t)||$ vs. time for different $a(x)$; (d) A semi-log plot of  $||u_{xx}(x,t)||$ vs. time for different $a(x)$.}
\label{fig:Chapter4-03}
\end{figure}

\newpage

 \subsection{The dispersive equation with a time-delay}
 In this subsection, we revisit the dispersive equation (\ref{1}) with time-delay.

\,\,\,Throughout this section, we take the physical parameters $\nu=0.01$ and $\mu=0.001$,  the initial condition $v(x,s)=\sin (\pi x)$, and  we consider the following two cases:

{\bf{Case 1:}} {\bf{$a(x)$ is a non-positive function:}} We consider the same four non-positive functions treated in Section 4.1, and simulate the system (\ref{1}) when the time-delay $\tau=1$. Figure 5 presents the time evolution of the solution $u(x,t)$ for these four functions. Figure (5a) depicts that the dynamics is still stable when $a(x)=0$. In turn, the dynamics for each of the other selected negative functions  is unstable.  Furthermore, in each case, the $L^2$-norm $||u(x,t)||$  versus time is plotted (see Figure 6). In this case, it is shown that the time-delay $\tau=1$ destabilizes a stable dynamics when $a(x)$ is negative. On the other hand, when $a(x)=0$, the dynamics of the dispersive equation with a time-delay $\tau=1$ is exponentially stable.

\,\,\,{\bf{Case 2:}} {\bf{$a(x)$ is a positive function:}} We consider the following four positive functions of $a(x)$: i) $a(x)=1$; ii) $a(x) = 1 + x$; iii) $a(x) = 1 + \sin(\pi x)$; and iv) $a(x) = 1 + 2x + \sin(2 \pi x)$, and simulate system (\ref{1}) when the time-delay $\tau=1$. Figure 7 presents the time evolution of the solution $u(x,t)$ for these four cases. The figure indicates that in each case the solution converges to the zero solution. Furthermore, in each case the $L^2$-norms $||u(x,t)||$ and $||u_{xx}(x,t)||$ versus time are plotted in Figures (8a) and (8c), respectively, where it is shown that the two norms converge exponentially to zero as $ t \to \infty $. The exponential convergence is validated  by plotting the semi-log plots of these two norms versus time (see Figures (8b) and (8d)). A careful look at these figures reveals that because of the effect of the time-delay,  the curves of these norms become straight line with negative slopes around $t=1.25$.  The exponential results  are in accordance with the analytical results presented in Section 3. In addition, among the four chosen functions of $a(x)$, the dynamics of the dispersive equation when the time-delay $\tau=1$ corresponding to $a(x) = 1 + 2x + \sin(2 \pi x)$ has the fastest convergence rate; whereas, the dynamics corresponding to $a(x) = 1$ has the slowest convergence rate. This is because $a(x) = 1 + 2x + \sin(2 \pi x)$ is the largest, while $a(x)=1$ is the smallest. This observation is similar to the one noted in Section 4.1 when the time-delay $\tau = 0$.

 \,\,\,Next, we shall study the effect of the choice  of the time-delay $\tau$ on the stability of the system (\ref{1}). To do so, we vary the time-delay $\tau$ and simulate the dynamics of the system. Figures (9a)-(12a) and (9c)-(12c) show that each of the $L^2$-norms of $||u(x,t)||$ and $||u_{xx}(x,t)||$  converges exponentially to zero  for each of the four cases. The rate of convergence of these norms increases slowly as the value of $\tau$ increases. Again, the exponential decay can be confirmed by plotting semi-log plots of these norms versus time  revealing that the curves of these norms are  straight lines with negative slopes (see Figures (9b)-(12b) and (9d)-(12d)).

\newpage
\vspace*{1.5in}
\begin{figure}[!h]
\centering
\subfigure[]{
\includegraphics[scale=0.4]{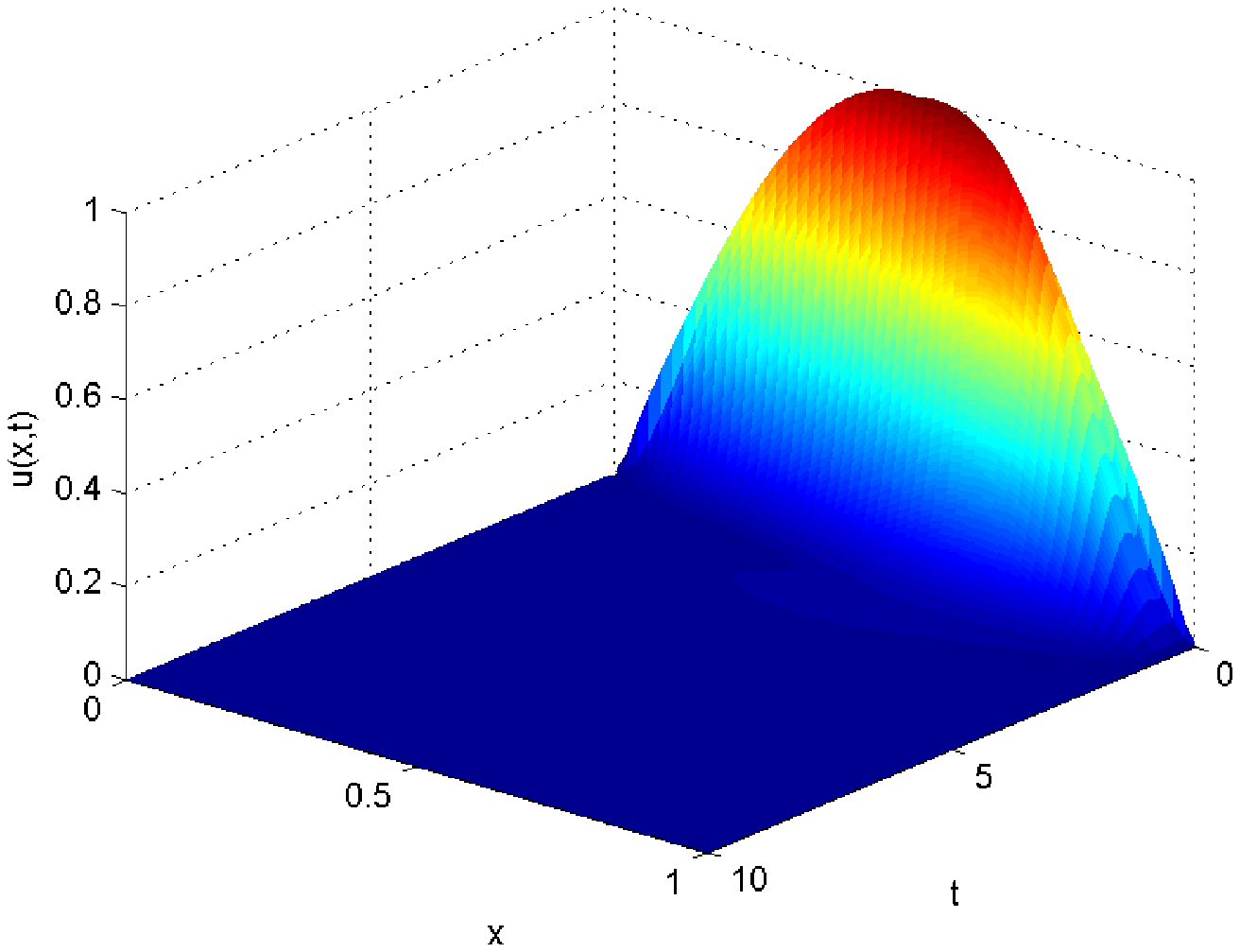}
\label{fig:subfiga}
}
\subfigure[]{
\includegraphics[scale=0.4]{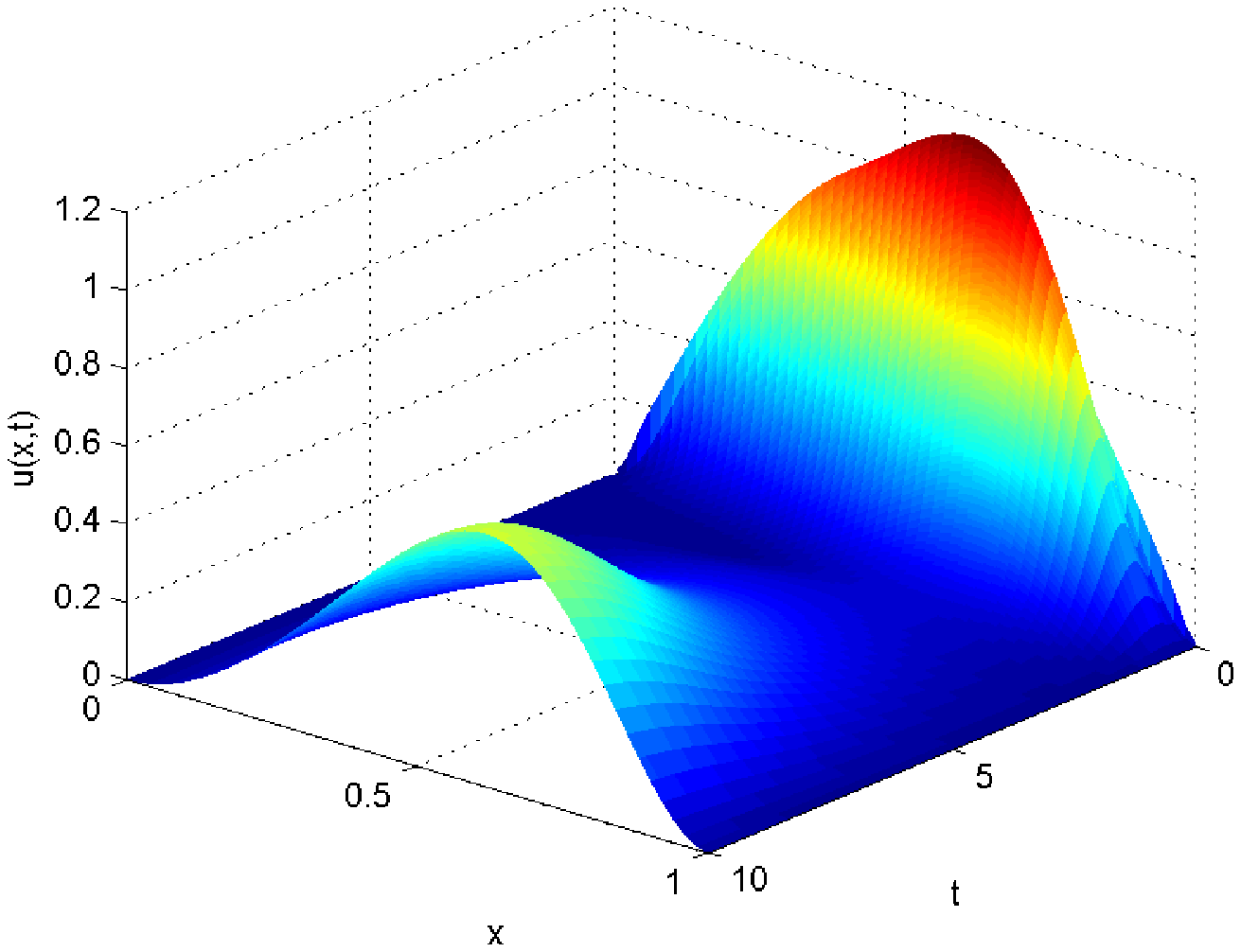}
\label{fig:subfigb}
}
\subfigure[]{
\includegraphics[scale=0.4]{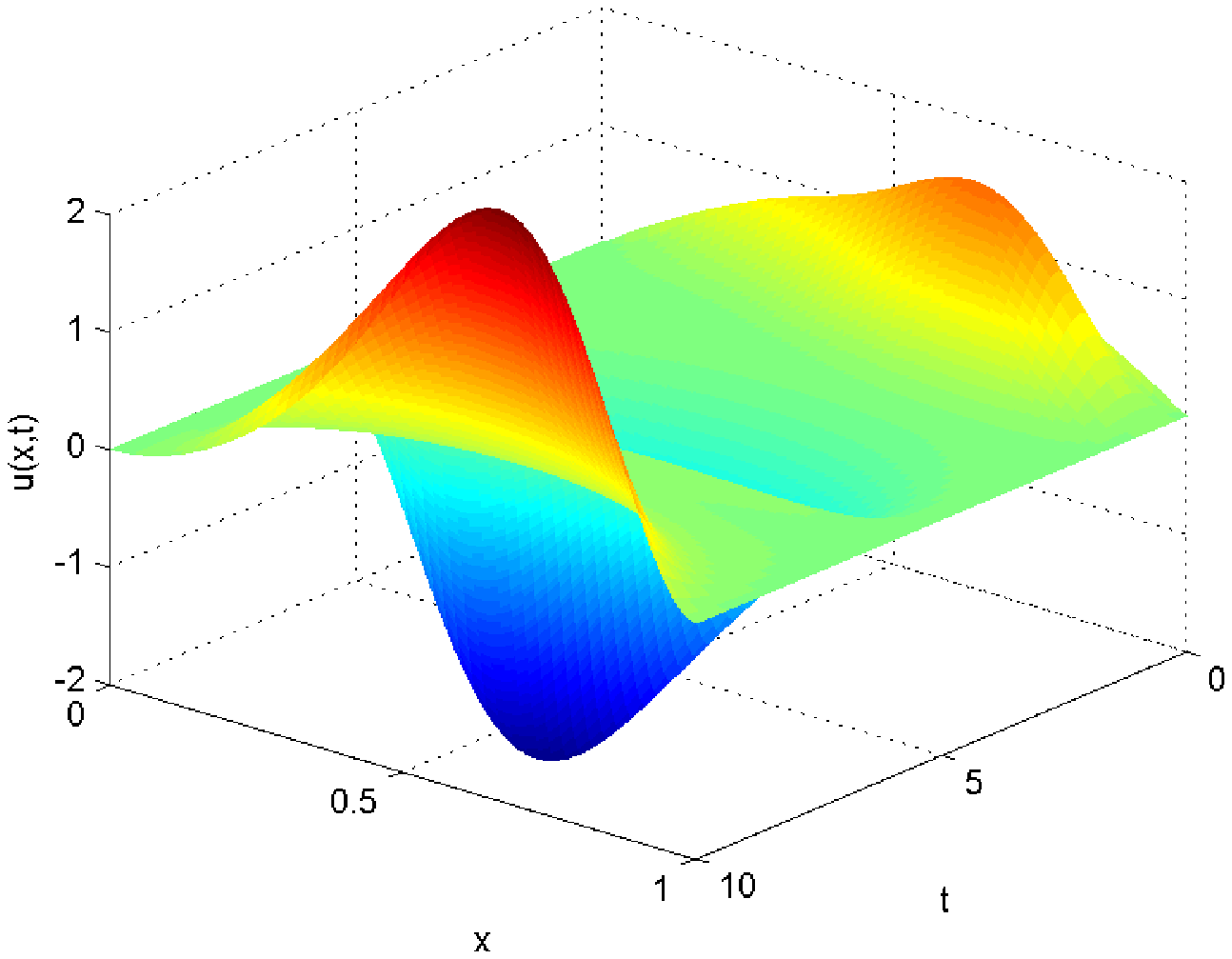}
\label{fig:subfigf}
}
\subfigure[]{
\includegraphics[scale=0.4]{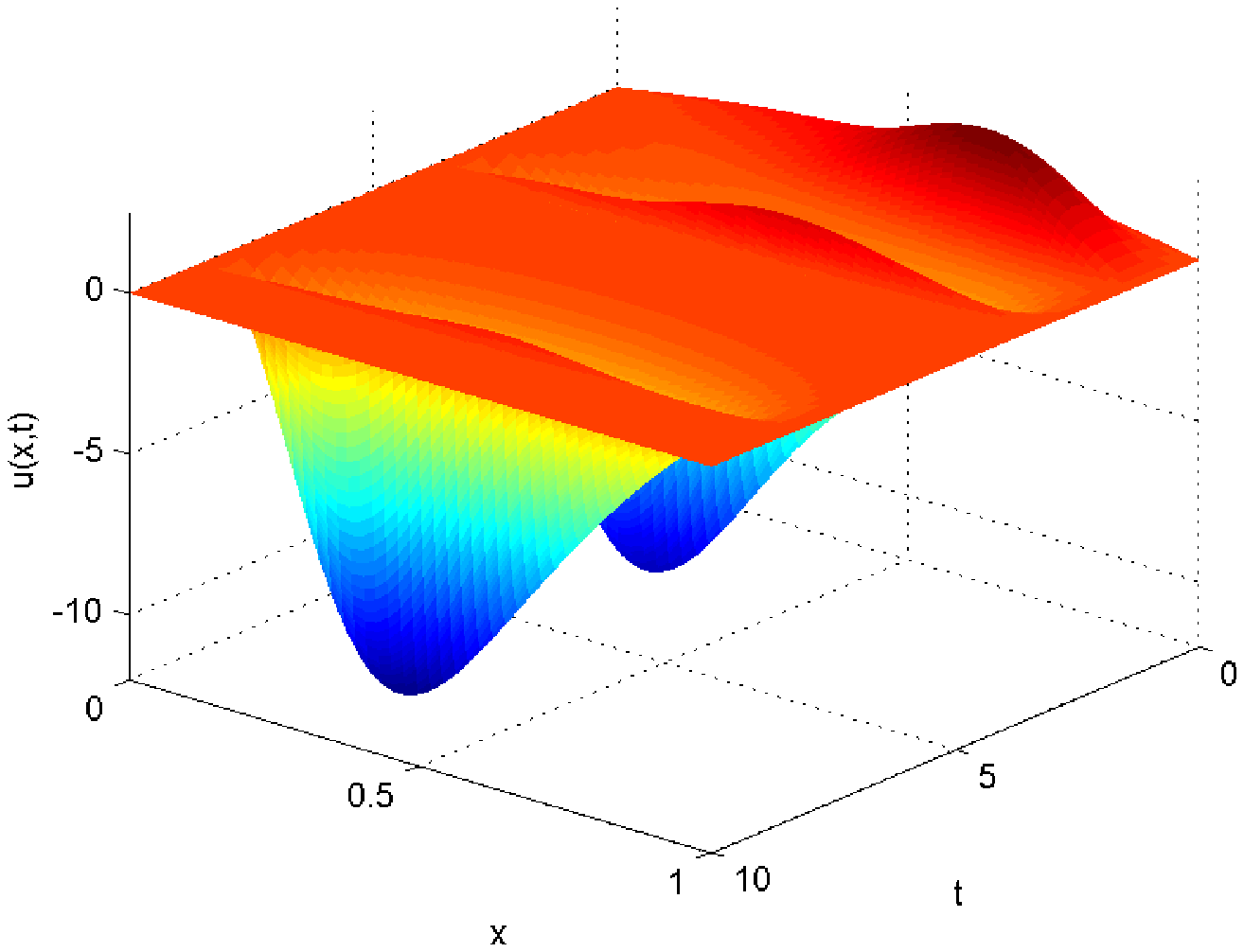}
\label{fig:subfige}
}
\caption{A 3-d landscape of the dynamics of the dispersive equation (1.1) with time-delay, $\tau=1$, when $\nu=0.01$, $\mu=0.001$ and $u(x,0)=\sin (\pi x)$ for different  functions $a(x)$; (a) $a(x)=0$; (b) $a(x)=-1$; (c) $a(x)=-2$; (d) $a(x)=-3$.}
\label{fig:Chapter4-03}
\end{figure}

\newpage
\vspace*{1.5in}
\begin{figure}[!h]
\centering
\subfigure[]{
\includegraphics[scale=0.4]{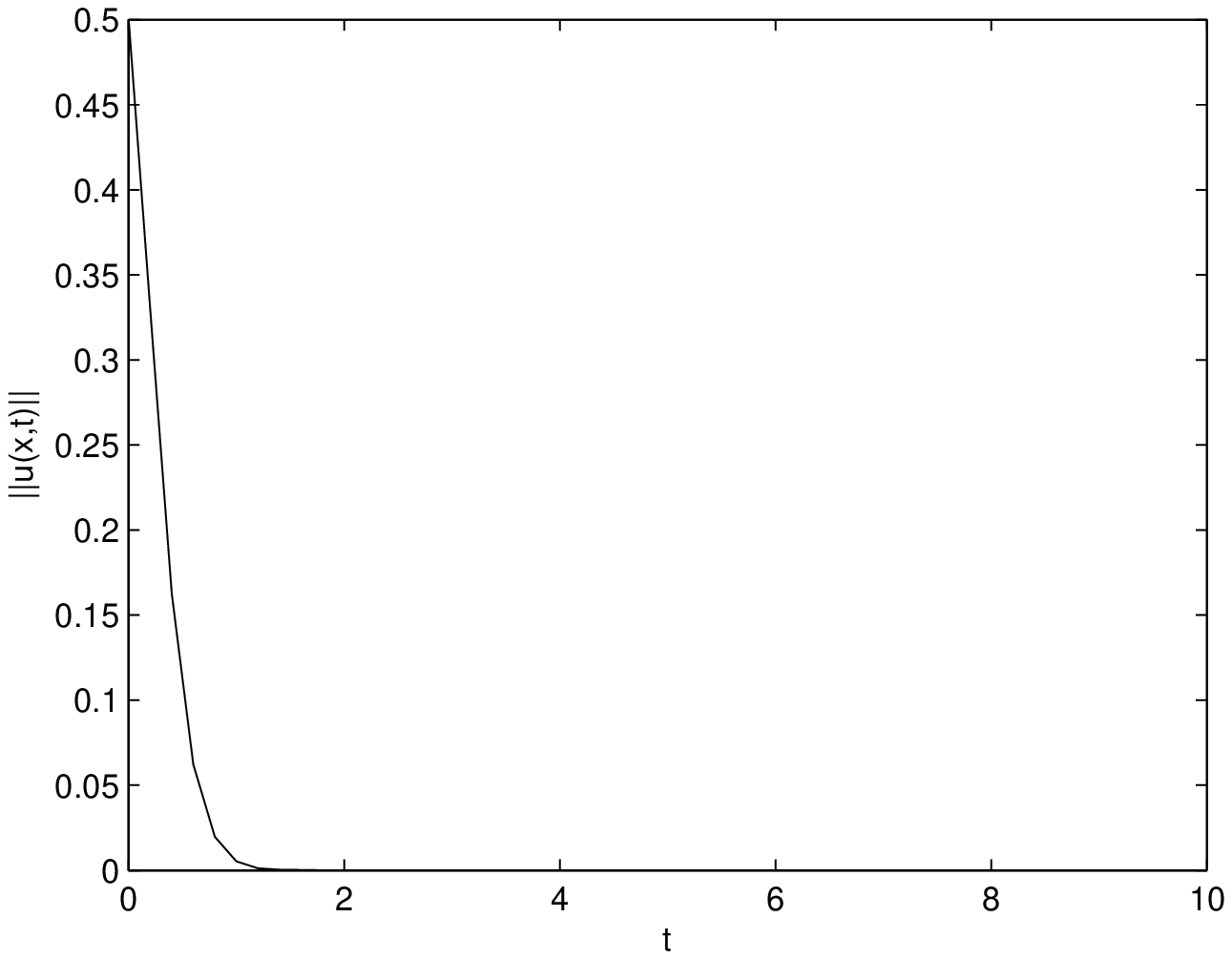}
\label{fig:subfiga}
}
\subfigure[]{
\includegraphics[scale=0.4]{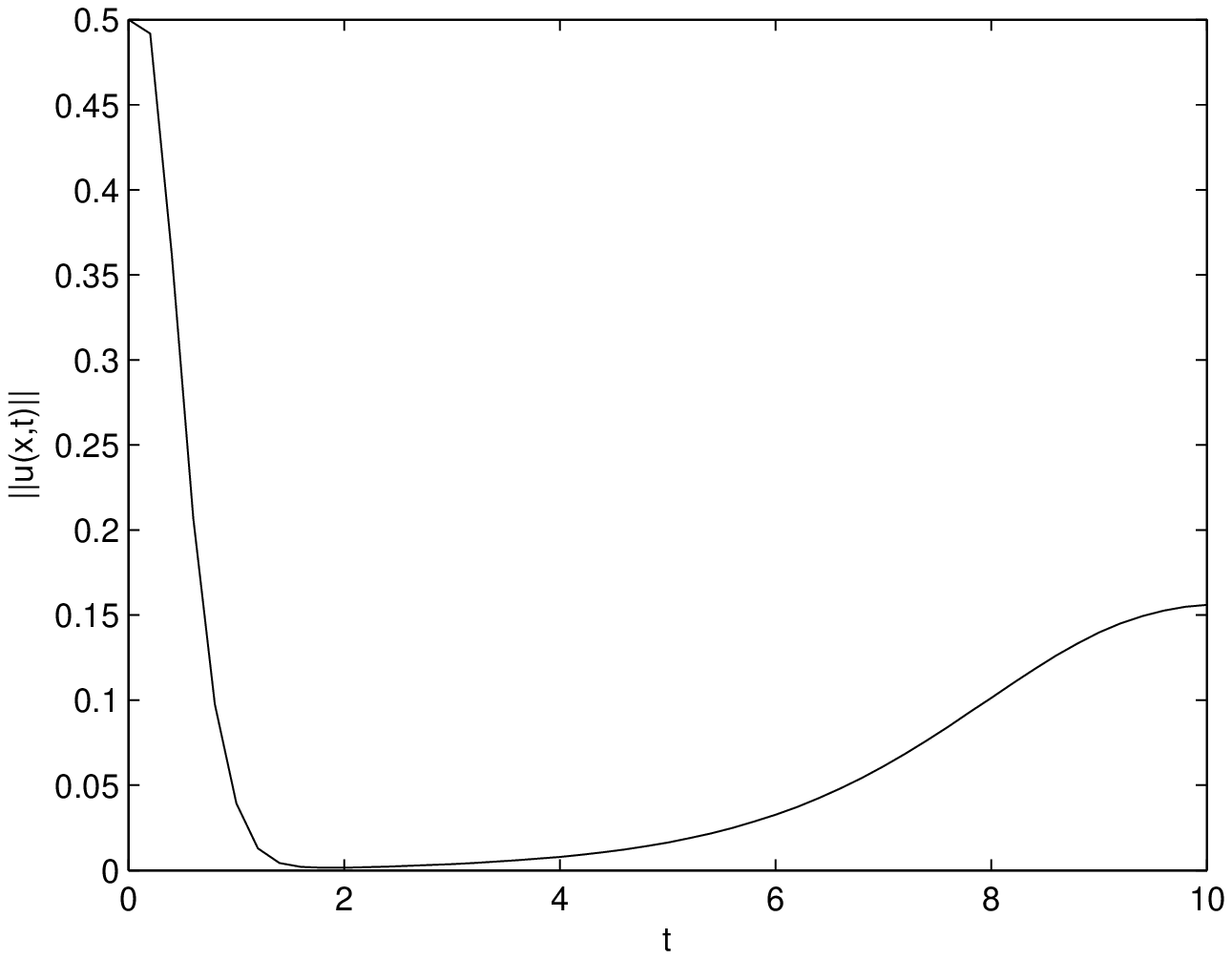}
\label{fig:subfigb}
}
\subfigure[]{
\includegraphics[scale=0.4]{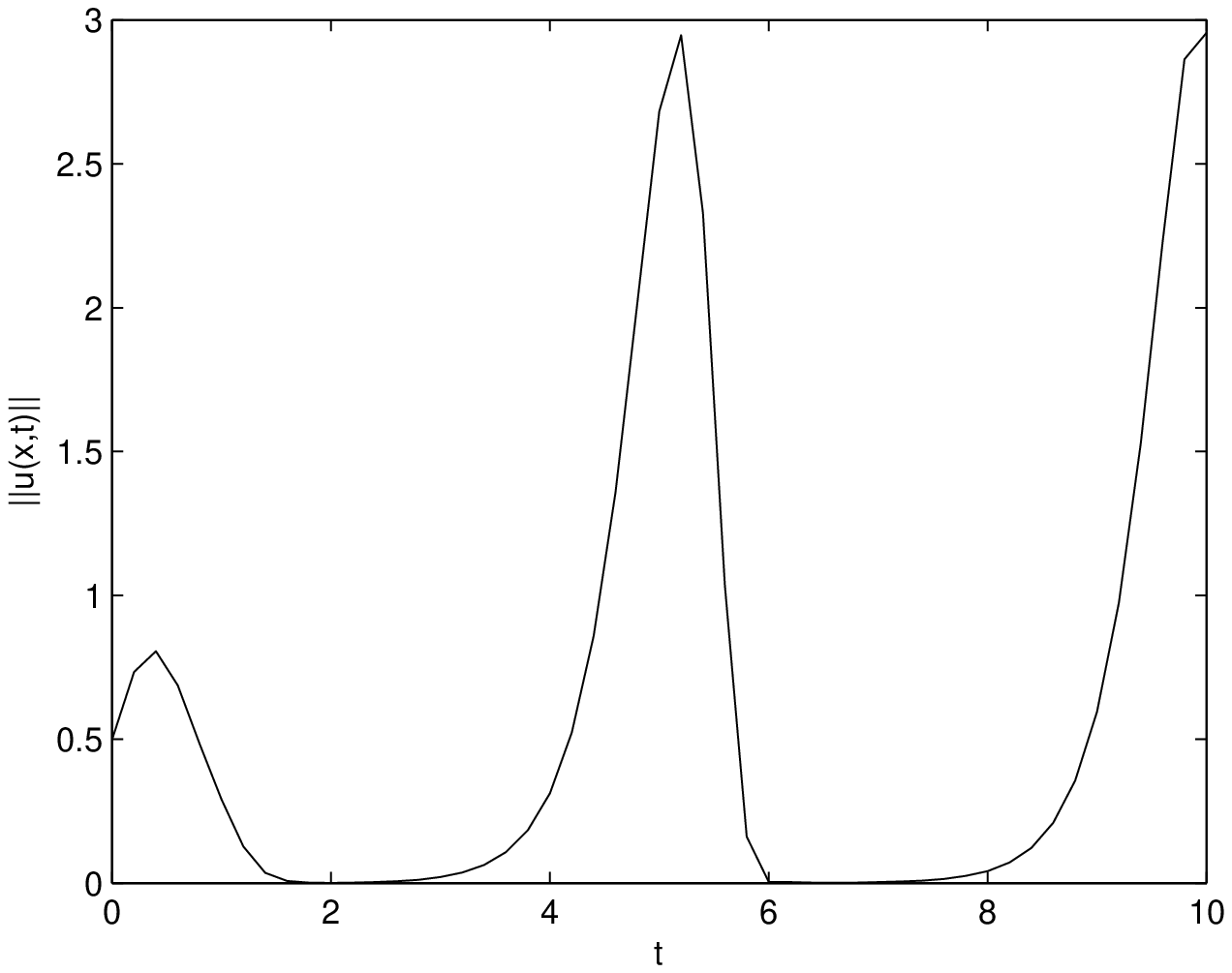}
\label{fig:subfigf}
}
\subfigure[]{
\includegraphics[scale=0.4]{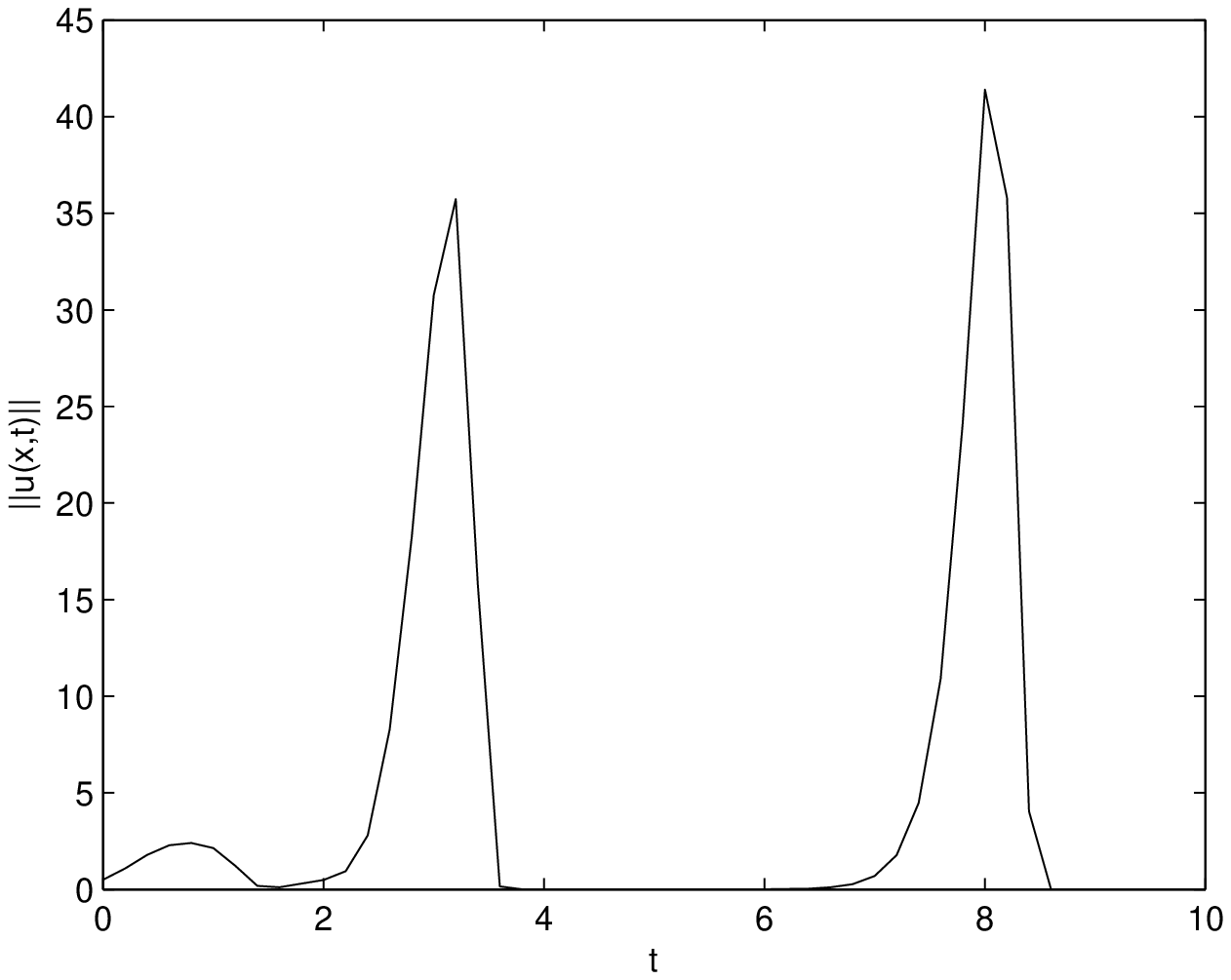}
\label{fig:subfige}
}
\caption{The $L^2$-norms $||u(x,t)||$   when the time-delay $\tau=1$; (a)  $a(x)=0$; $a(x)=-1$; (c) $a(x)=-2$; (d) $a(x)=-3$.}
\label{fig:Chapter4-03}
\end{figure}

\newpage
\vspace*{1.5in}
\begin{figure}[!h]
\centering
\subfigure[]{
\includegraphics[scale=0.4]{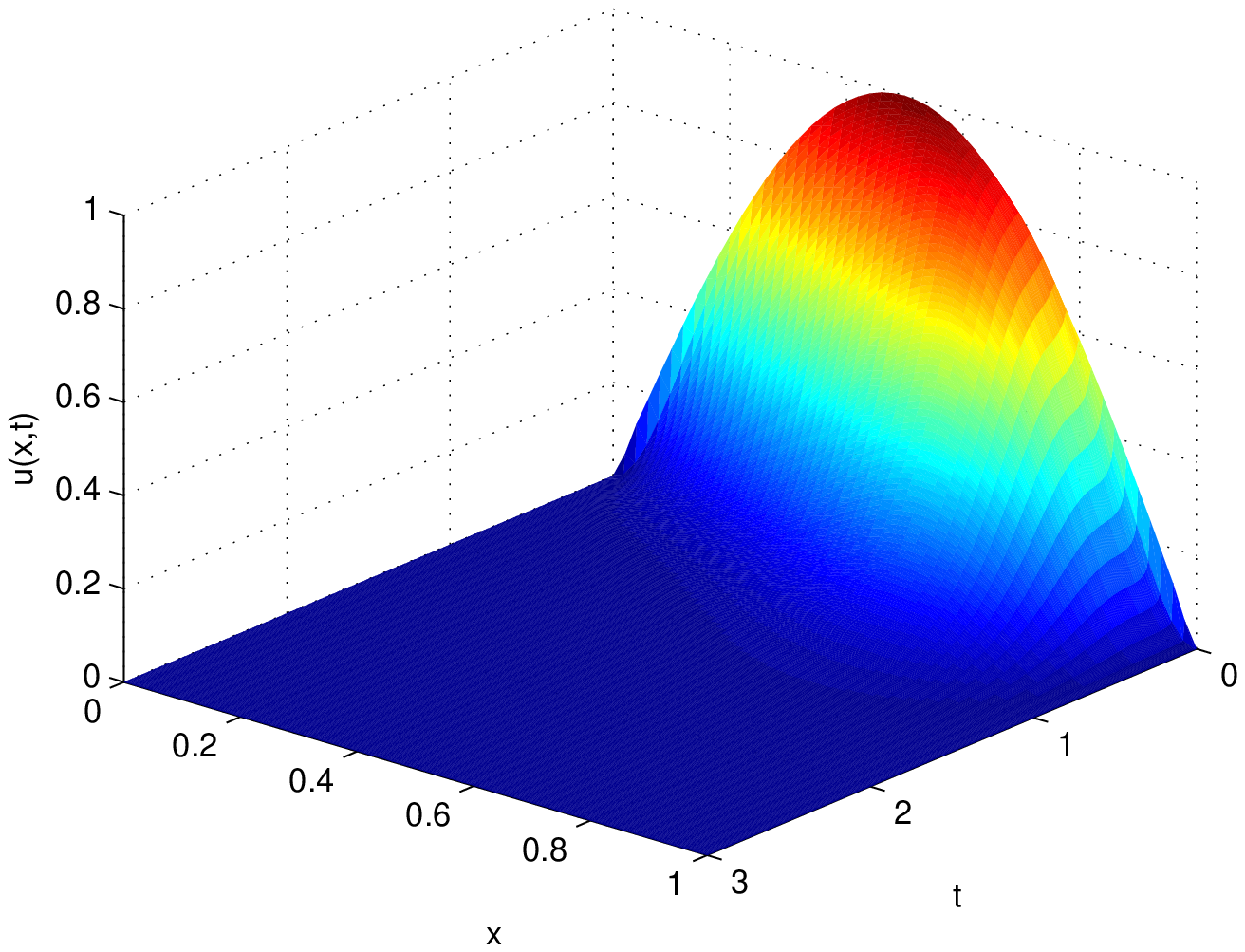}
\label{fig:subfiga}
}
\subfigure[]{
\includegraphics[scale=0.4]{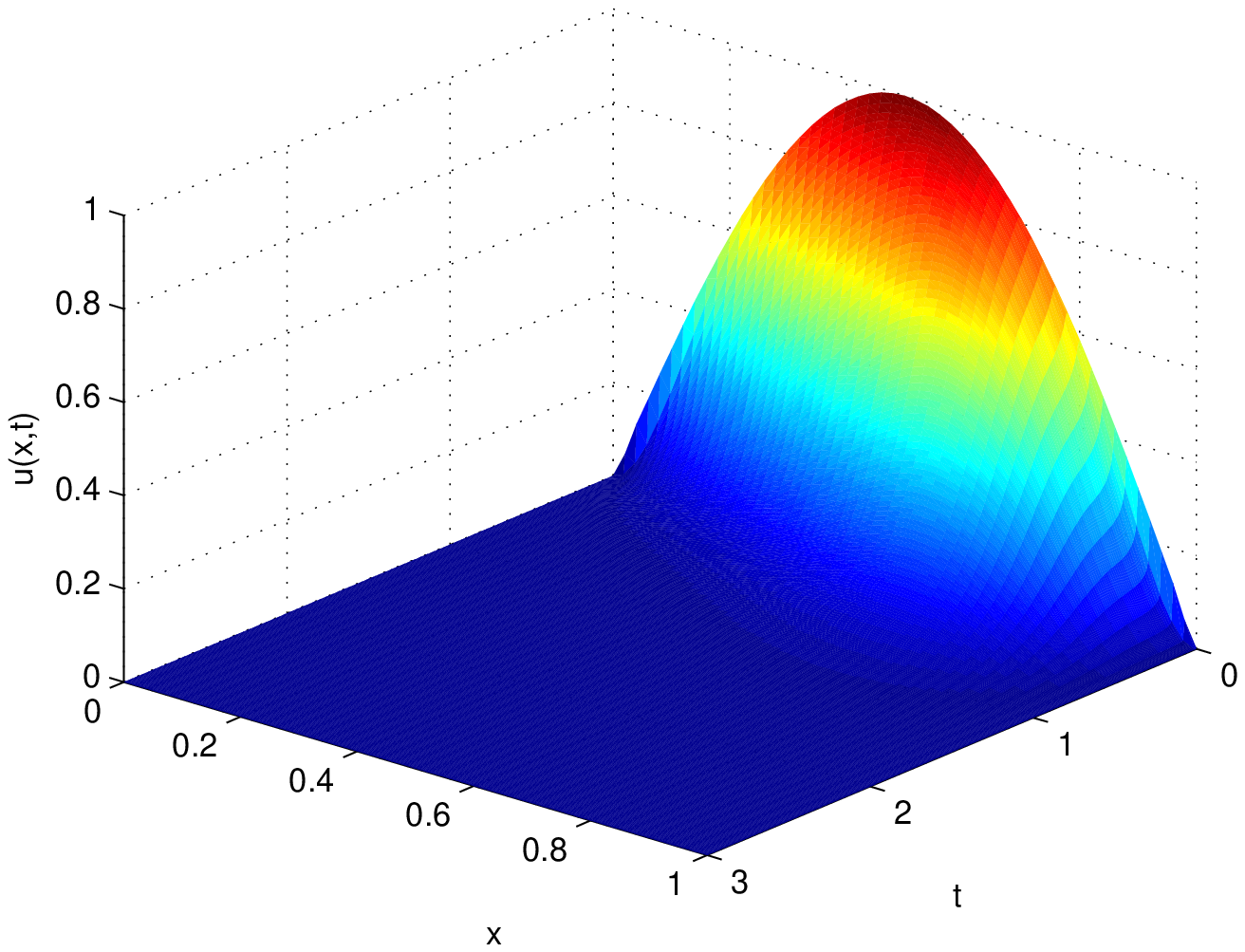}
\label{fig:subfigb}
}
\subfigure[]{
\includegraphics[scale=0.4]{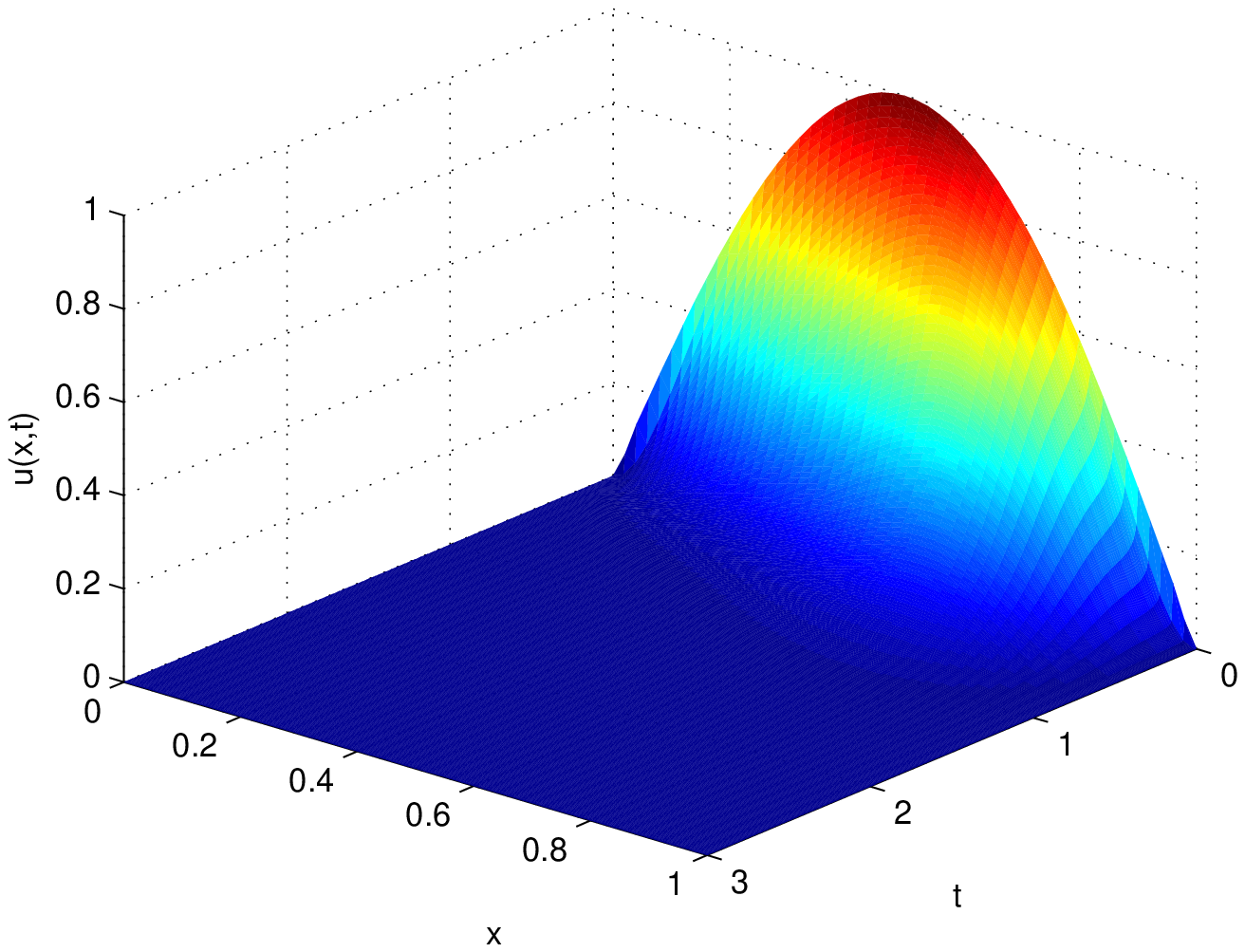}
\label{fig:subfigf}
}
\subfigure[]{
\includegraphics[scale=0.4]{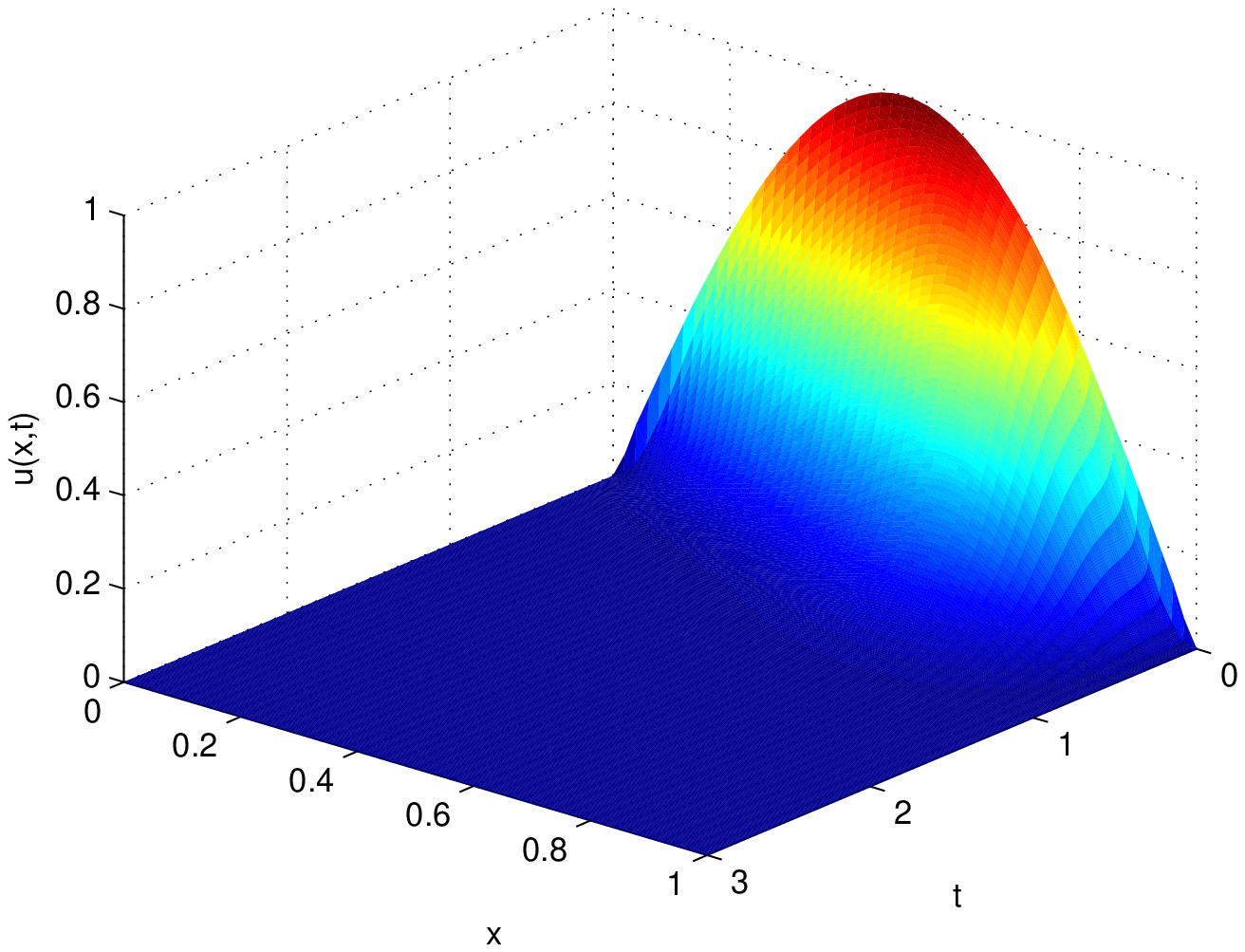}
\label{fig:subfige}
}
\caption{A 3-d landscape of the dynamics of the dispersive equation (1.1) with time-delay, $\tau=1$, when $\nu=0.01$, $\mu=0.001$ and $u(x,0)=\sin (\pi x)$ for different functions $a(x)$; (a) $a(x)=1$; (b) $a(x)=1+x$; (c) $a(x)=1+\sin(\pi x)$; (d) $a(x)=1+2x+\sin(2\pi x)$.}
\label{fig:Chapter4-03}
\end{figure}

\newpage
\vspace*{1.5in}
\begin{figure}[!h]
\centering
\subfigure[]{
\includegraphics[scale=0.4]{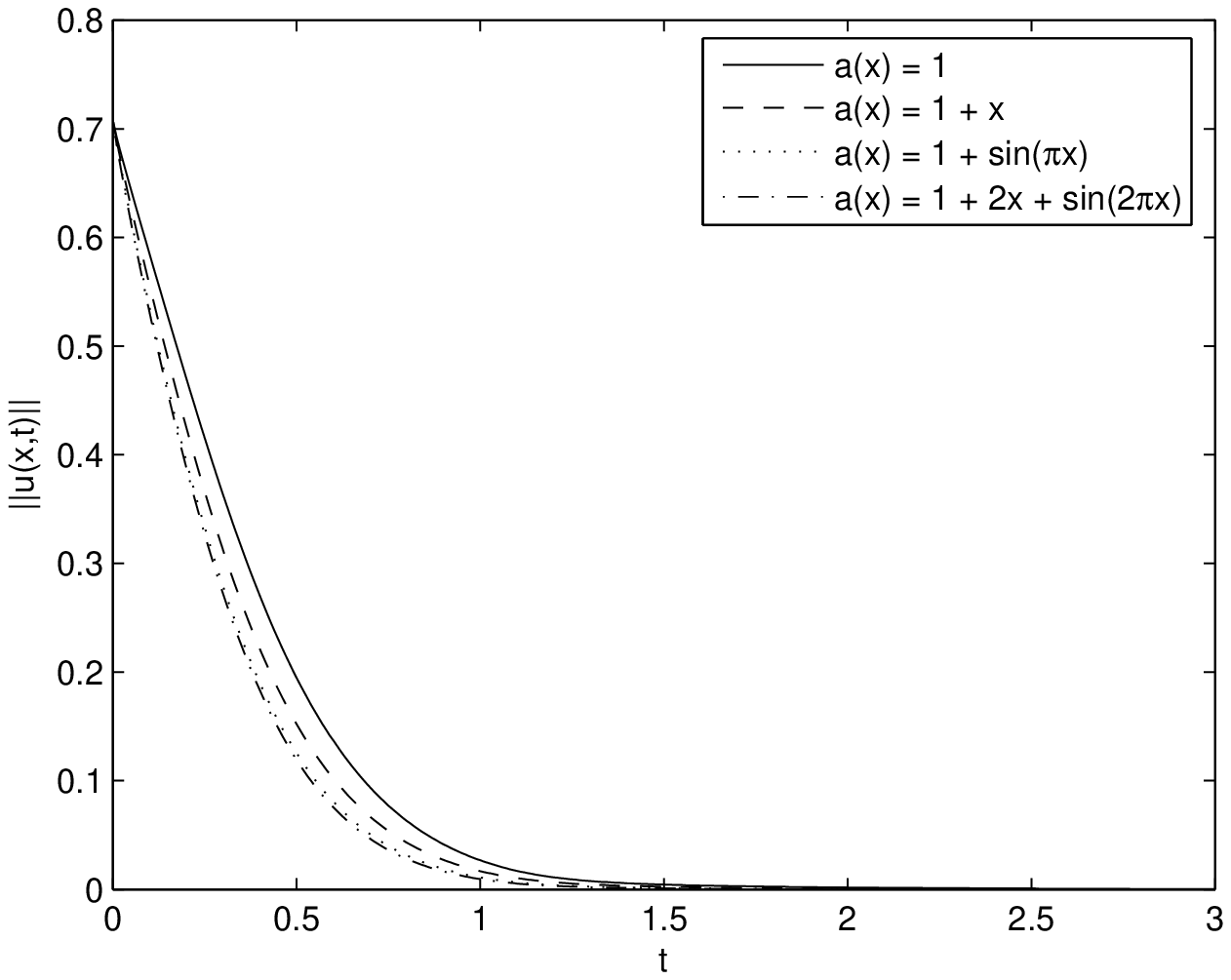}
\label{fig:subfiga}
}
\subfigure[]{
\includegraphics[scale=0.4]{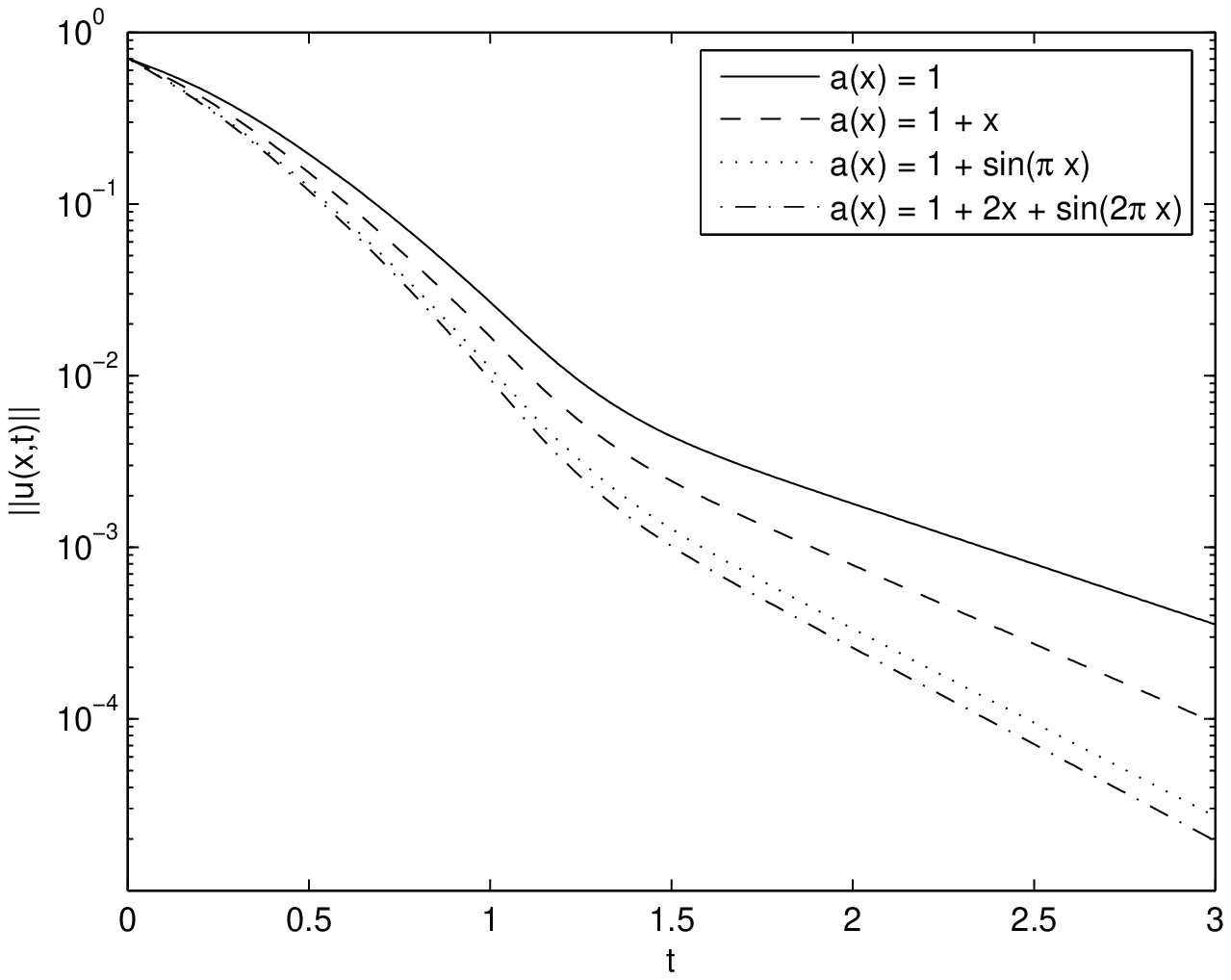}
\label{fig:subfigb}
}
\subfigure[]{
\includegraphics[scale=0.4]{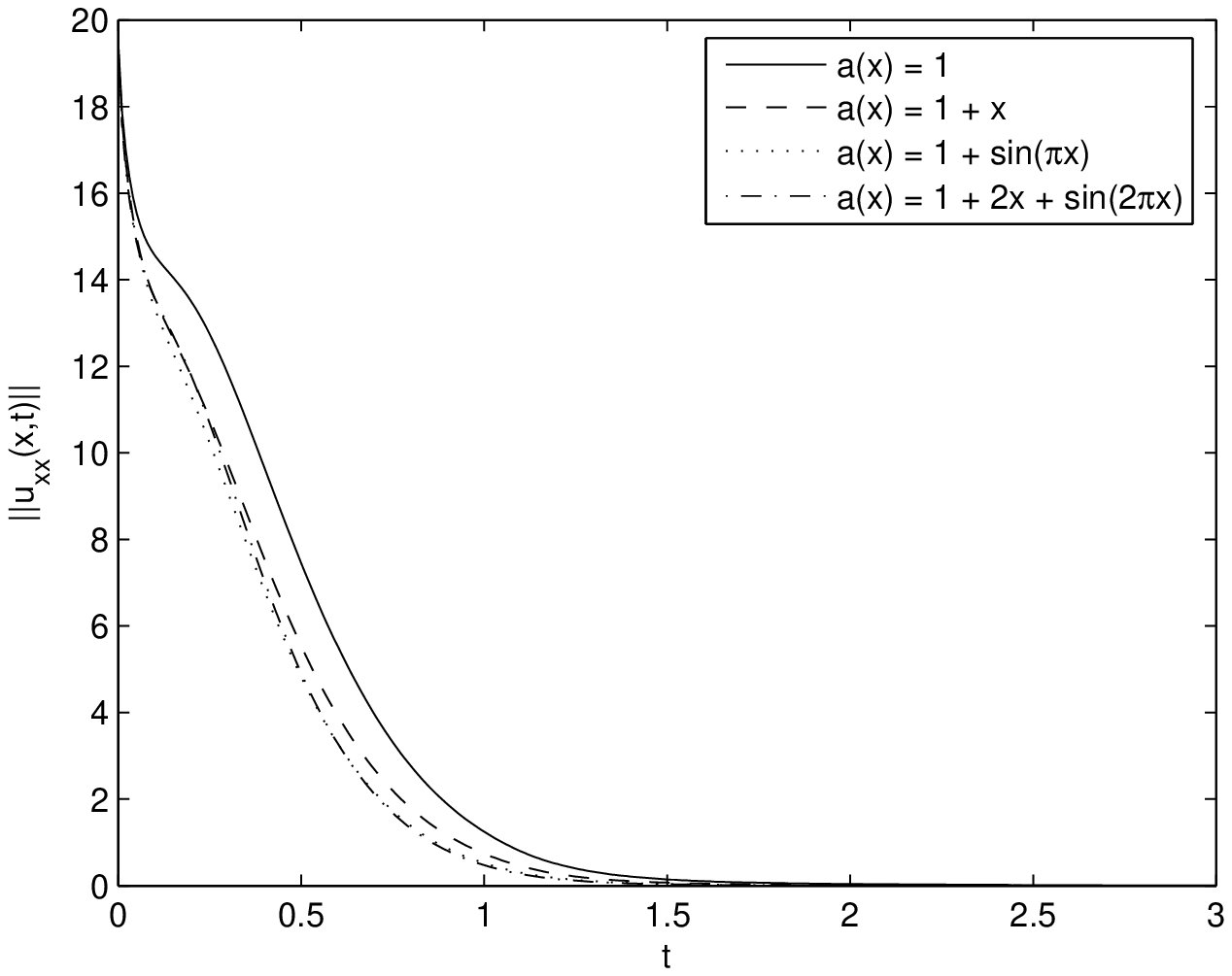}
\label{fig:subfigf}
}
\subfigure[]{
\includegraphics[scale=0.4]{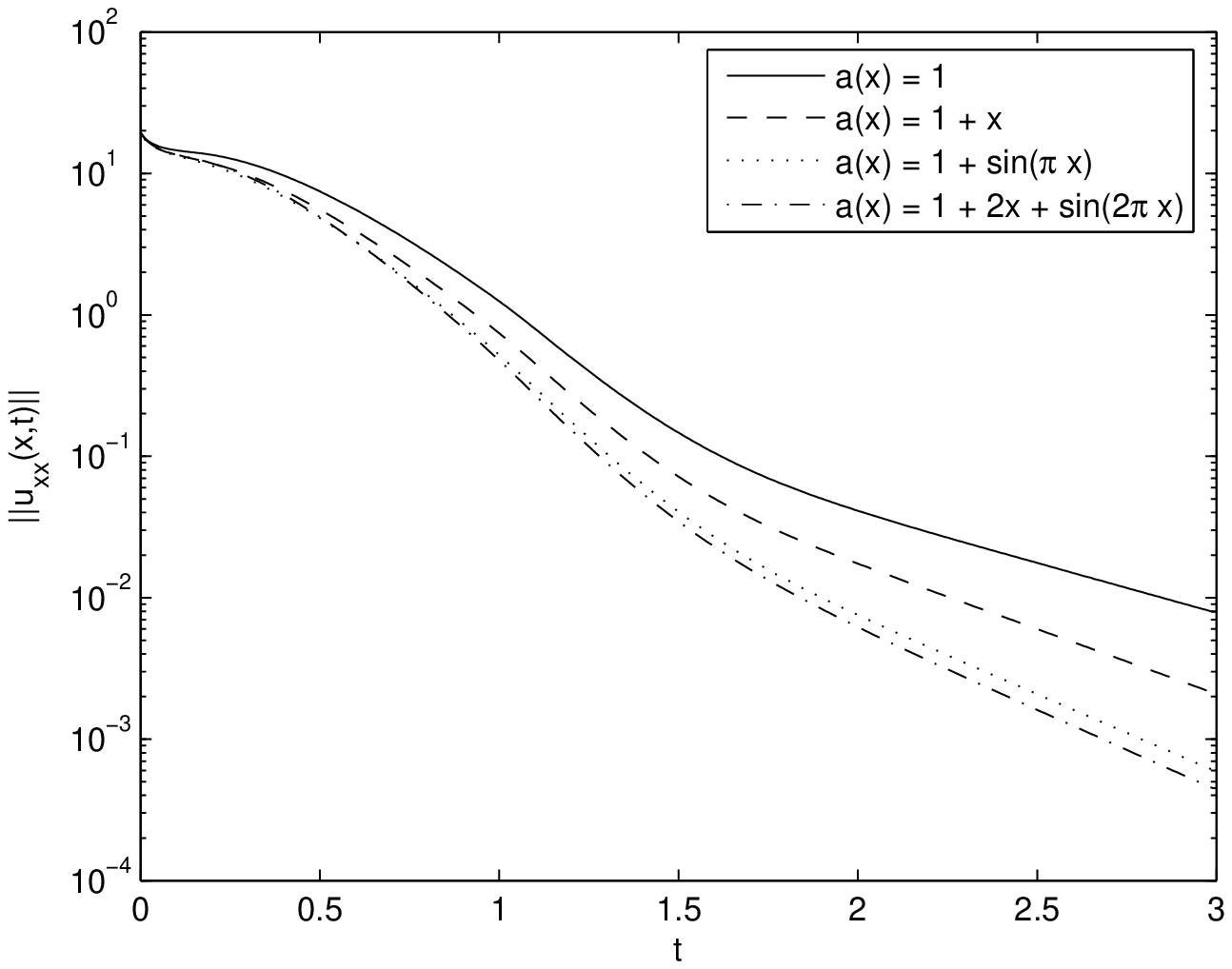}
\label{fig:subfige}
}
\caption{The $L^2$-norms $||u(x,t)||$ and $||u_{xx}(x,t)||$  when the time-delay $\tau=1$; (a)  $||u(x,t)||$ vs. time for different $a(x)$; (b) A semi-log plot of $||u(x,t)||$ vs. time for different $a(x)$; (c)   $||u_{xx}(x,t)||$ vs. time for different $a(x)$; (d) A semi-log plot of  $||u_{xx}(x,t)||$ vs. time for different $a(x)$.}
\label{fig:Chapter4-03}
\end{figure}

\newpage
\vspace*{1.5in}
\begin{figure}[!h]
\centering
\subfigure[]{
\includegraphics[scale=0.4]{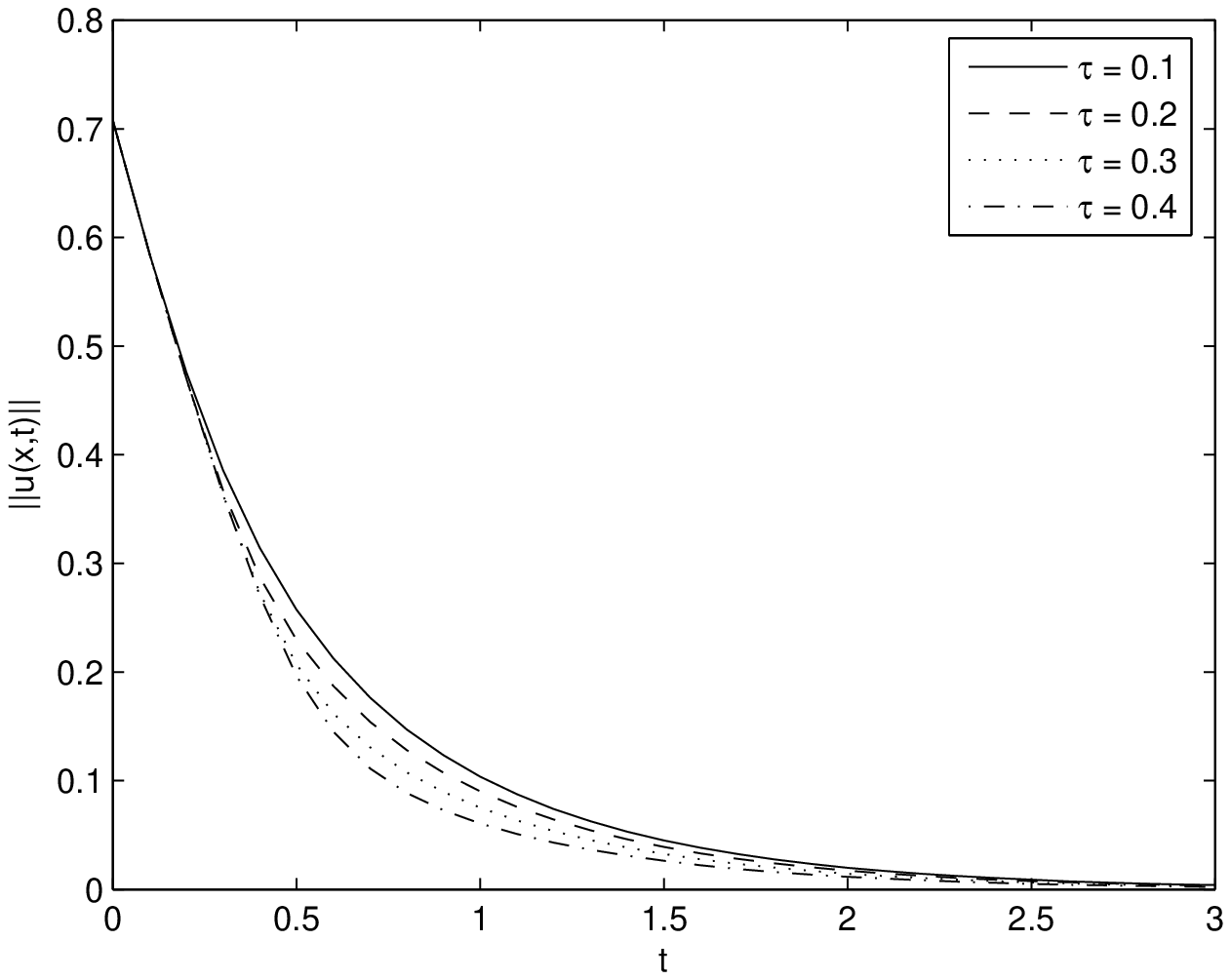}
\label{fig:subfiga}
}
\subfigure[]{
\includegraphics[scale=0.4]{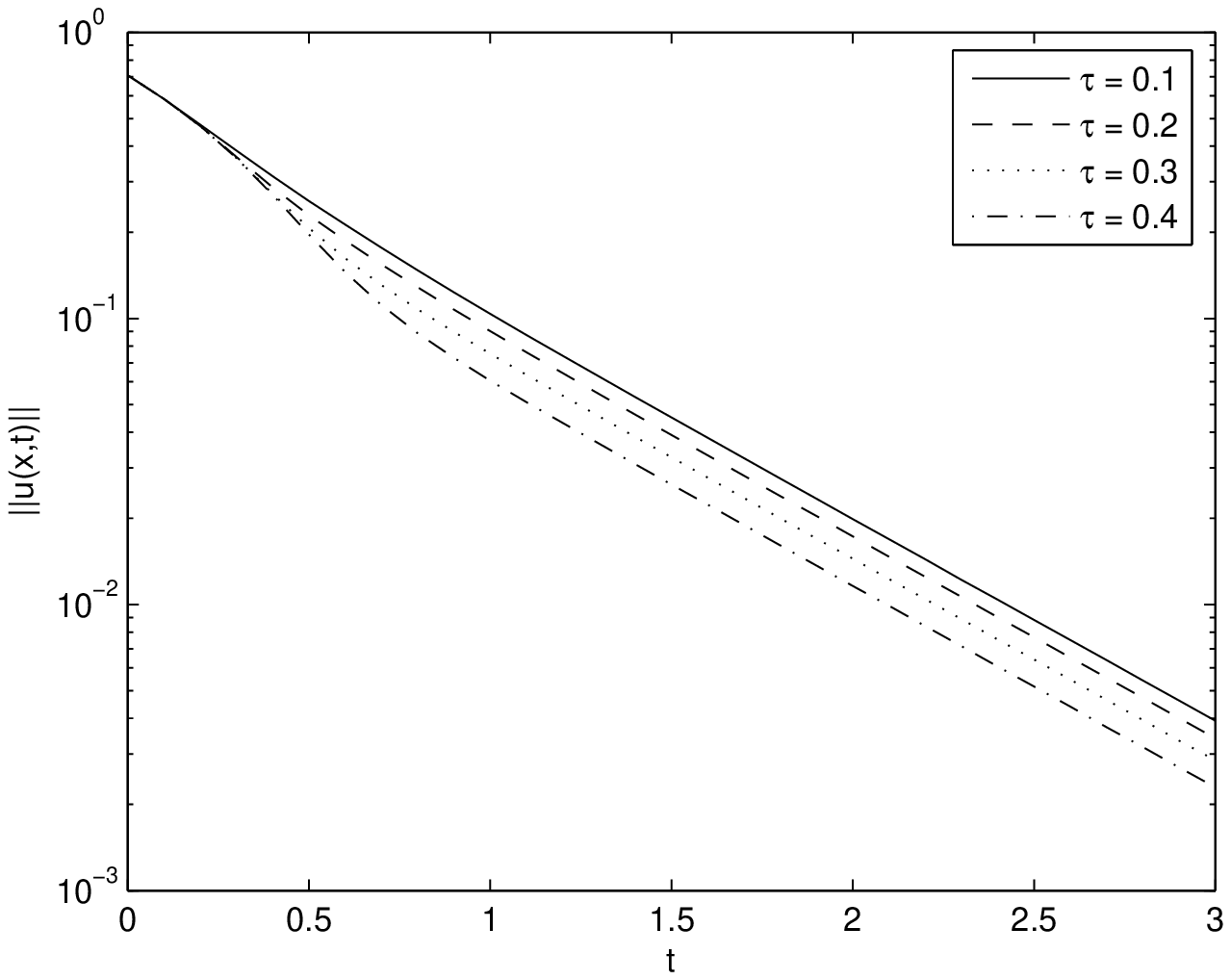}
\label{fig:subfigb}
}
\subfigure[]{
\includegraphics[scale=0.4]{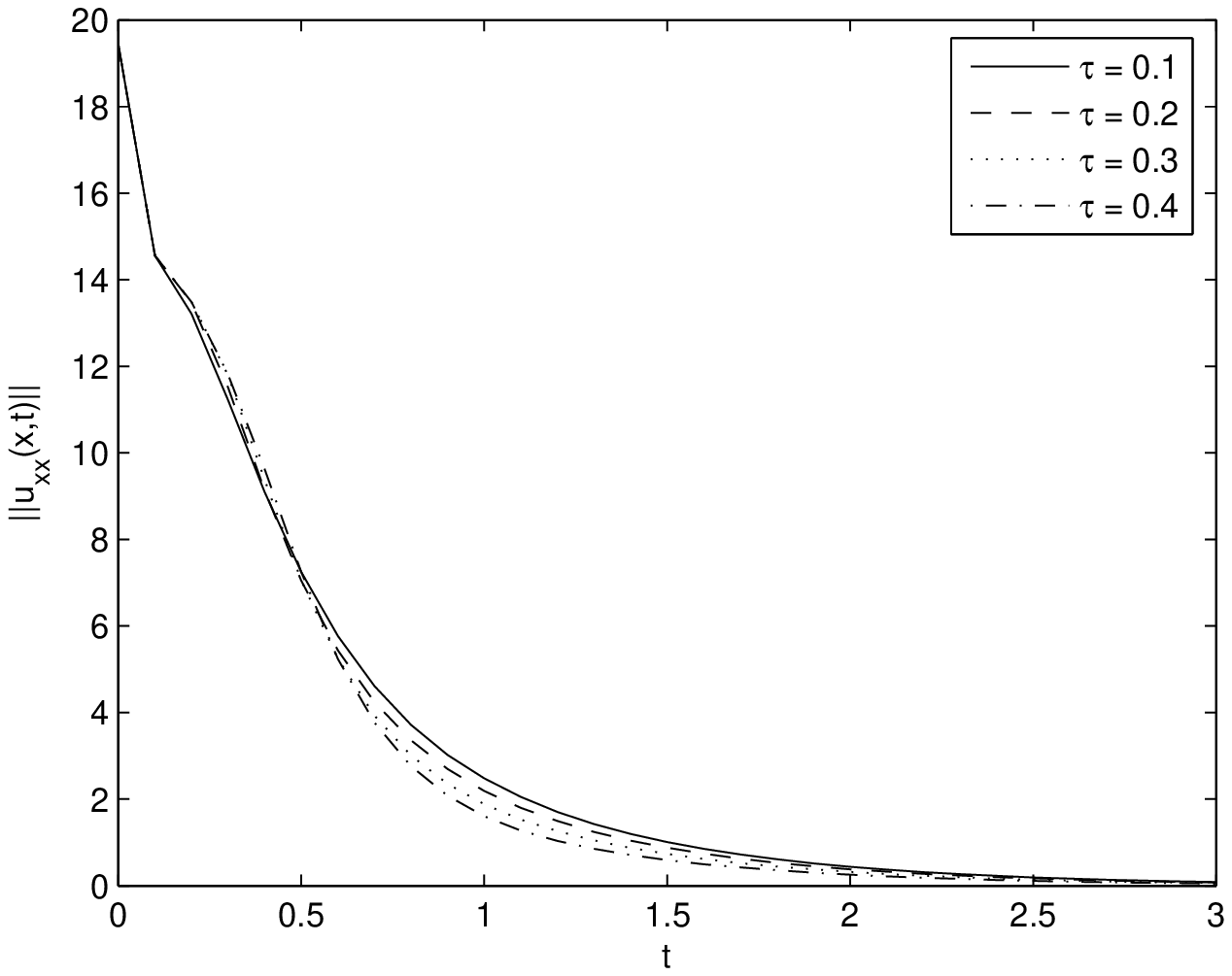}
\label{fig:subfigf}
}
\subfigure[]{
\includegraphics[scale=0.4]{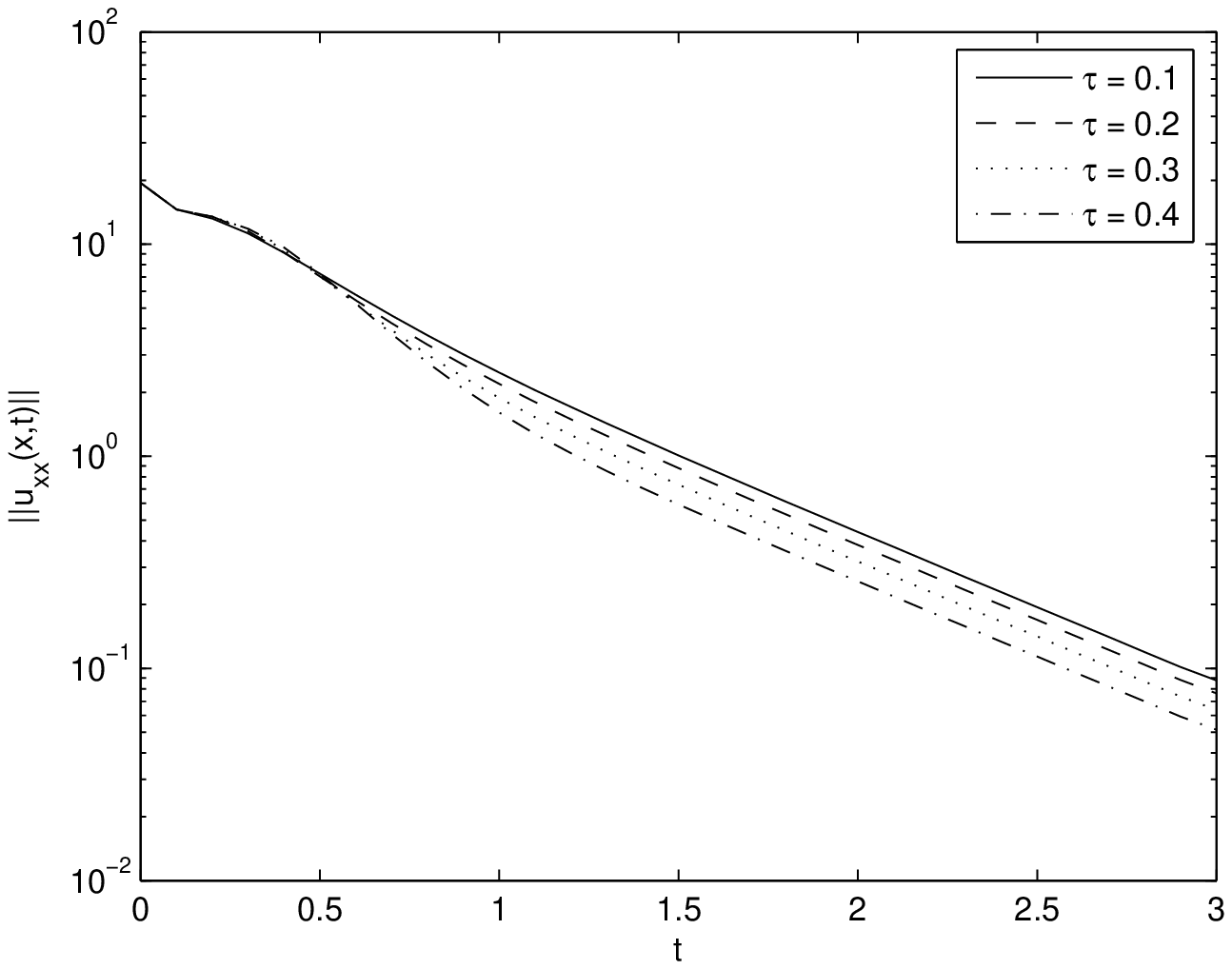}
\label{fig:subfige}
}
\caption{The $L^2$-norms $||u(x,t)||$ and $||u_{xx}(x,t)||$  when $a(x)=1$; (a)  $||u(x,t)||$ vs. time for different $\tau$; (b) A semi-log plot of $||u(x,t)||$ vs. time for different $\tau$; (c) $||u_{xx}(x,t)||$ vs. time for different $\tau$; (d) A semi-log plot of  $||u_{xx}(x,t)||$ vs. time for different $\tau$.}
\label{fig:Chapter4-03}
\end{figure}

\newpage
\vspace*{1.5in}
 \begin{figure}[!h]
\centering
\subfigure[]{
\includegraphics[scale=0.4]{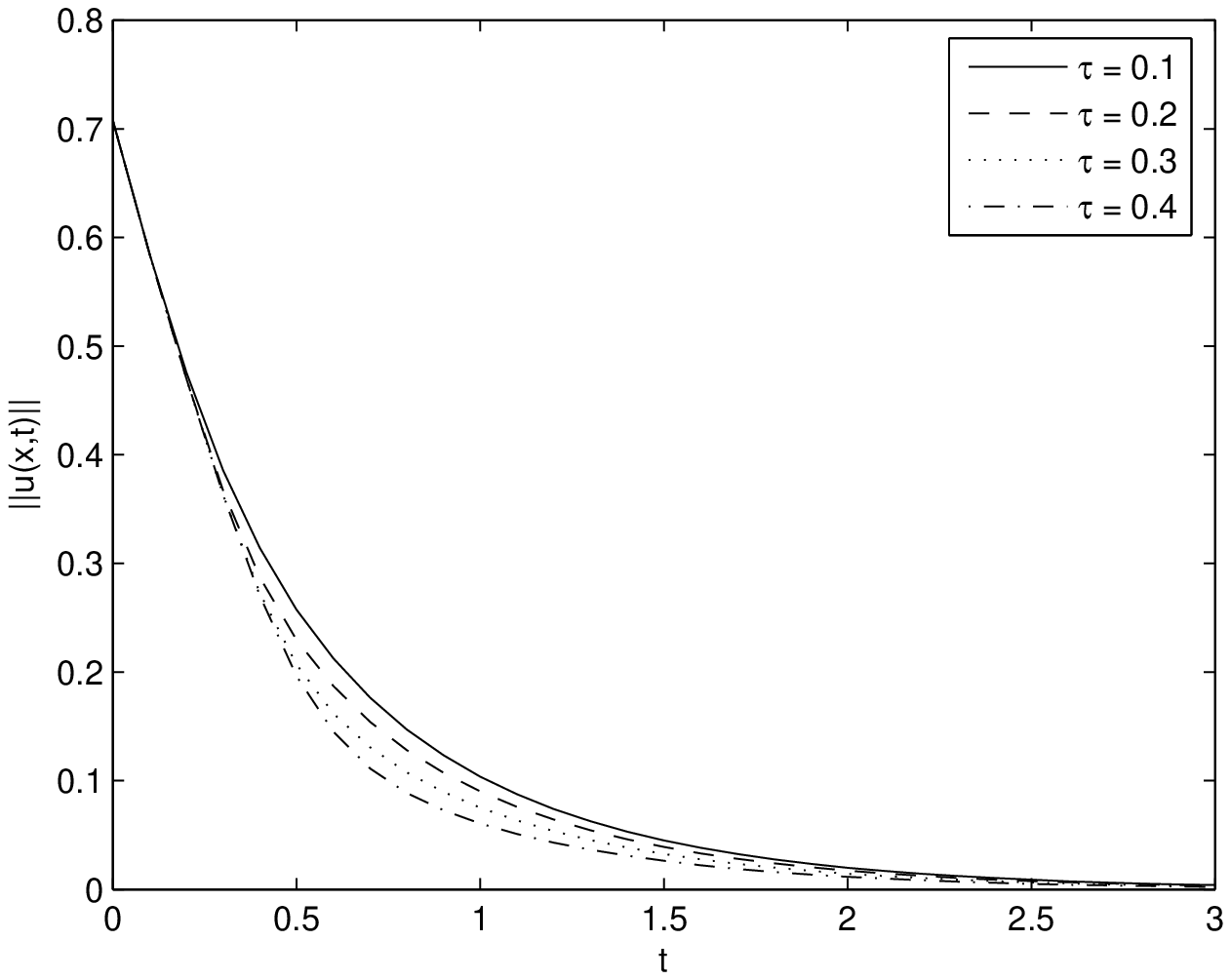}
\label{fig:subfiga}
}
\subfigure[]{
\includegraphics[scale=0.4]{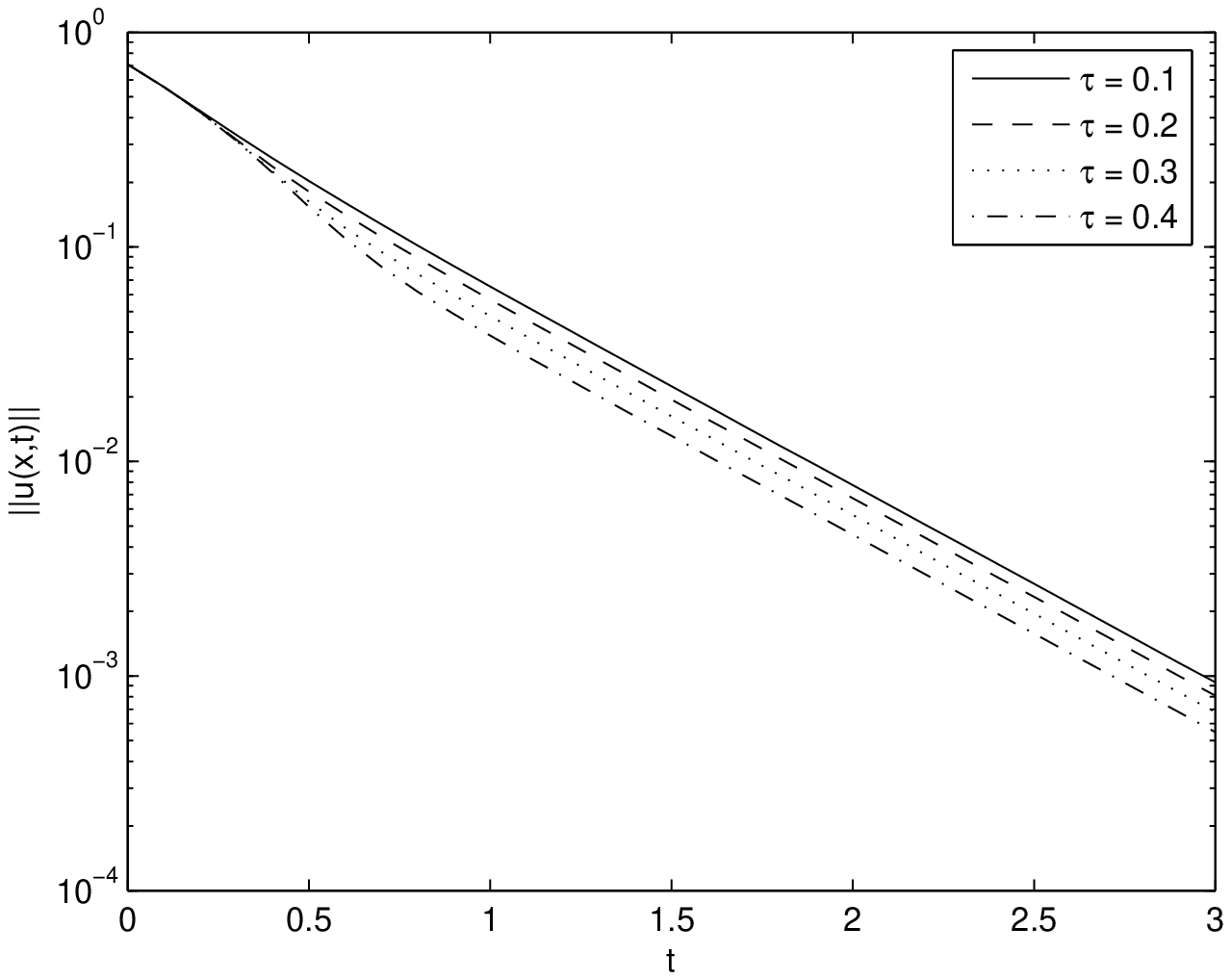}
\label{fig:subfigb}
}
\subfigure[]{
\includegraphics[scale=0.4]{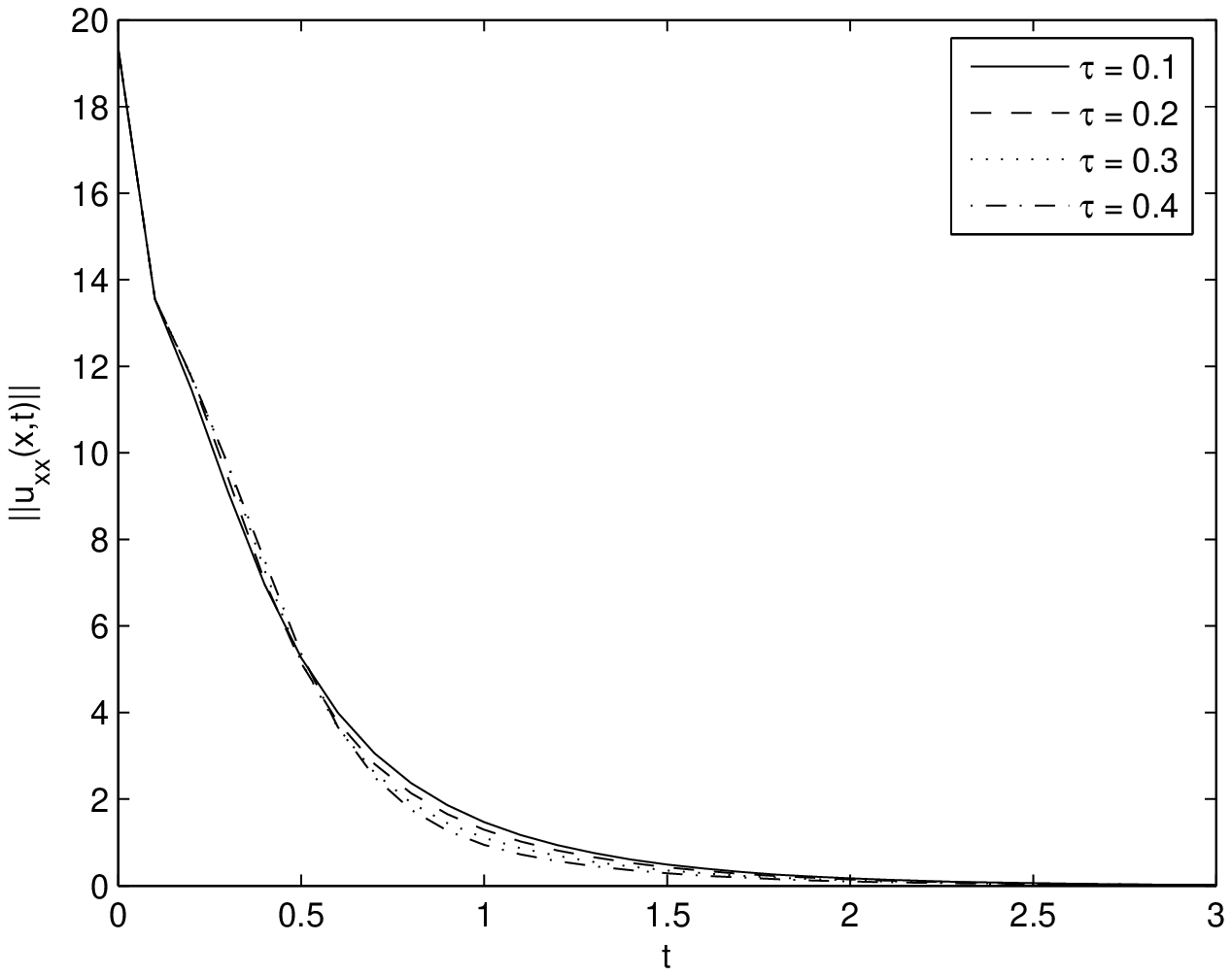}
\label{fig:subfigf}
}
\subfigure[]{
\includegraphics[scale=0.4]{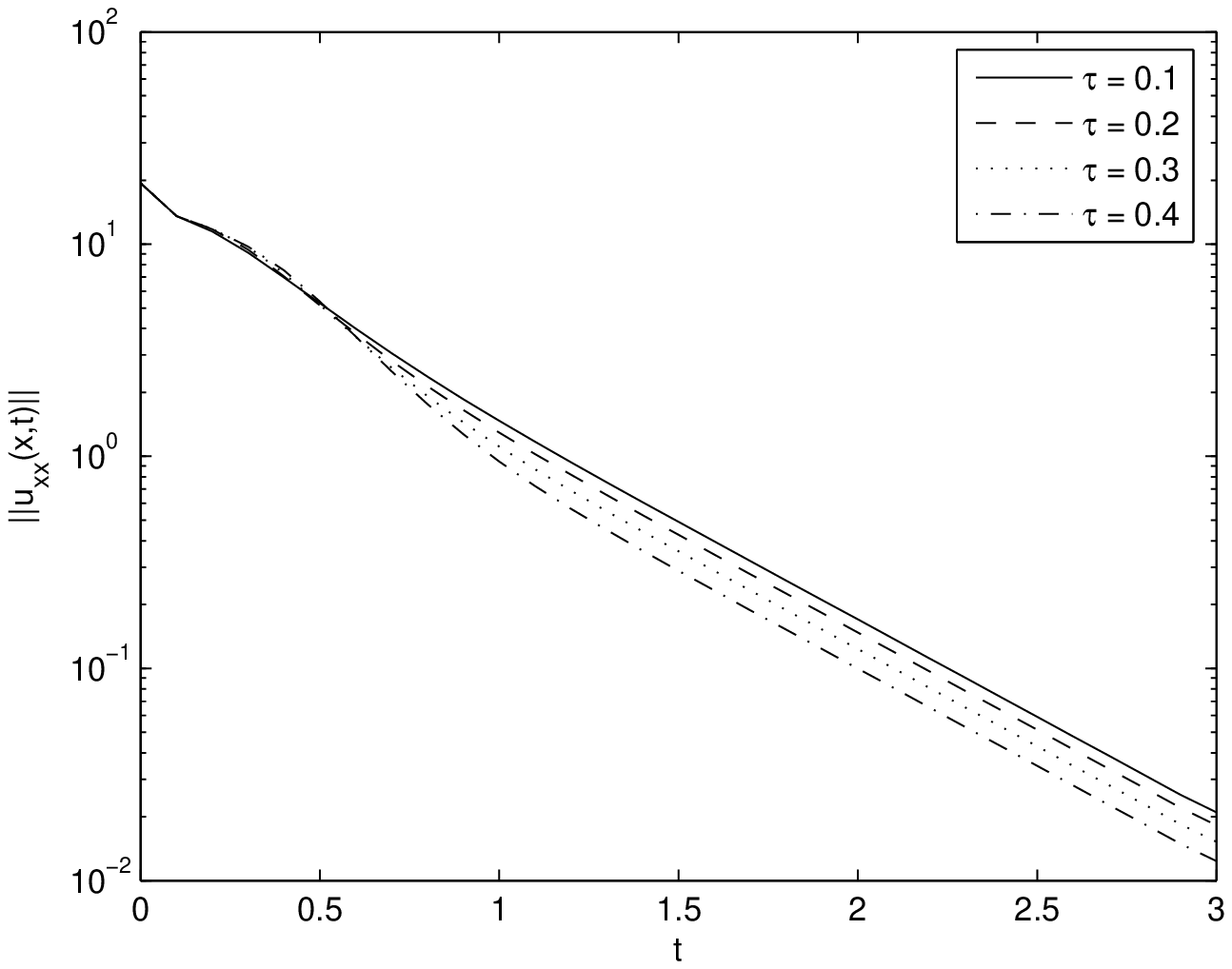}
\label{fig:subfige}
}
\caption{The $L^2$-norms $||u(x,t)||$ and $||u_{xx}(x,t)||$  when $a(x)=1+x$; (a)  $||u(x,t)||$ vs. time for different $\tau$; (b) A semi-log plot of $||u(x,t)||$ vs. time for different $\tau$; (c) $||u_{xx}(x,t)||$ vs. time for different $\tau$; (d) A semi-log plot of  $||u_{xx}(x,t)||$ vs. time for different $\tau$.}
\label{fig:Chapter4-03}
\end{figure}

\newpage
\vspace*{1.5in}
\begin{figure}[!h]
\centering
\subfigure[]{
\includegraphics[scale=0.4]{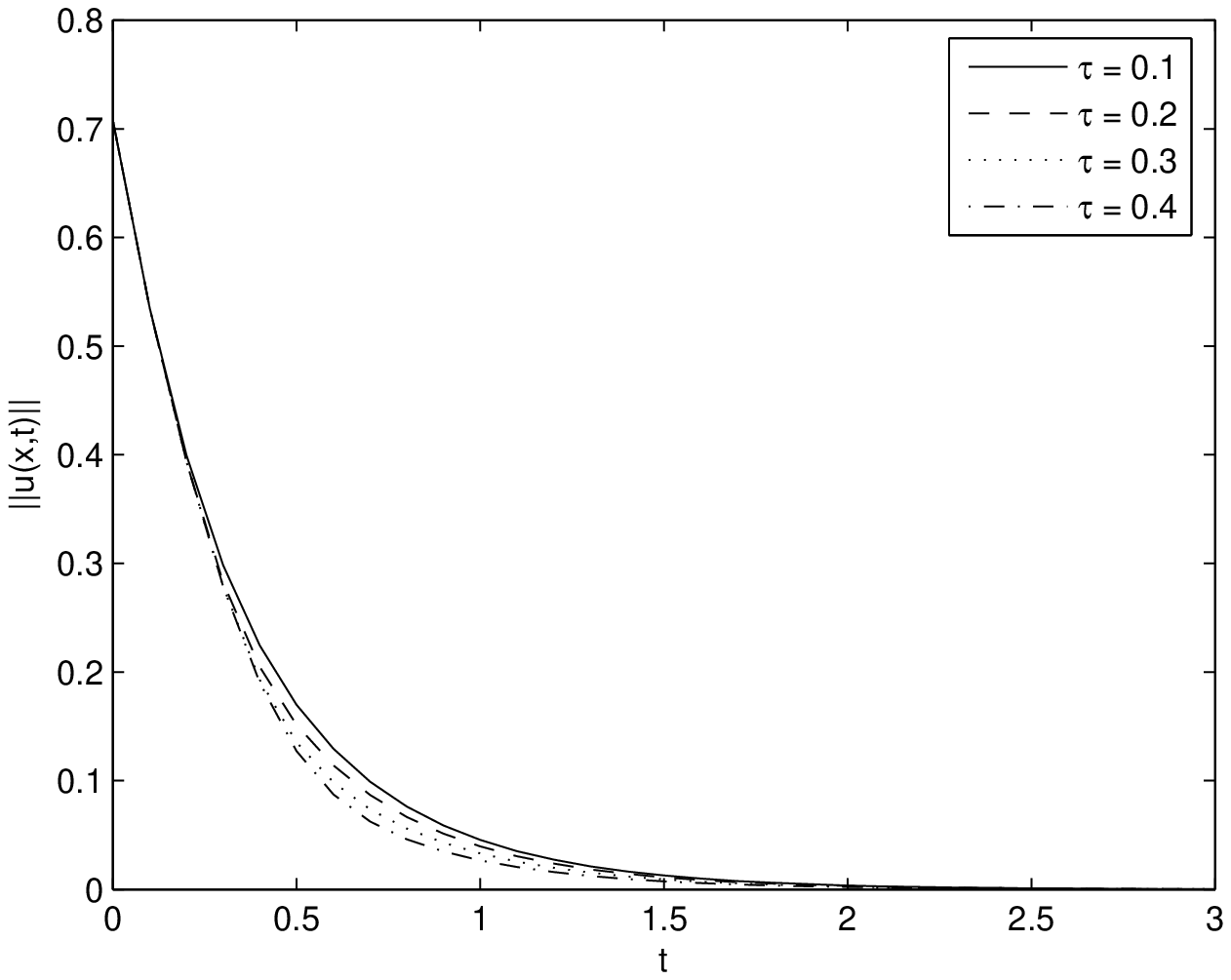}
\label{fig:subfiga}
}
\subfigure[]{
\includegraphics[scale=0.4]{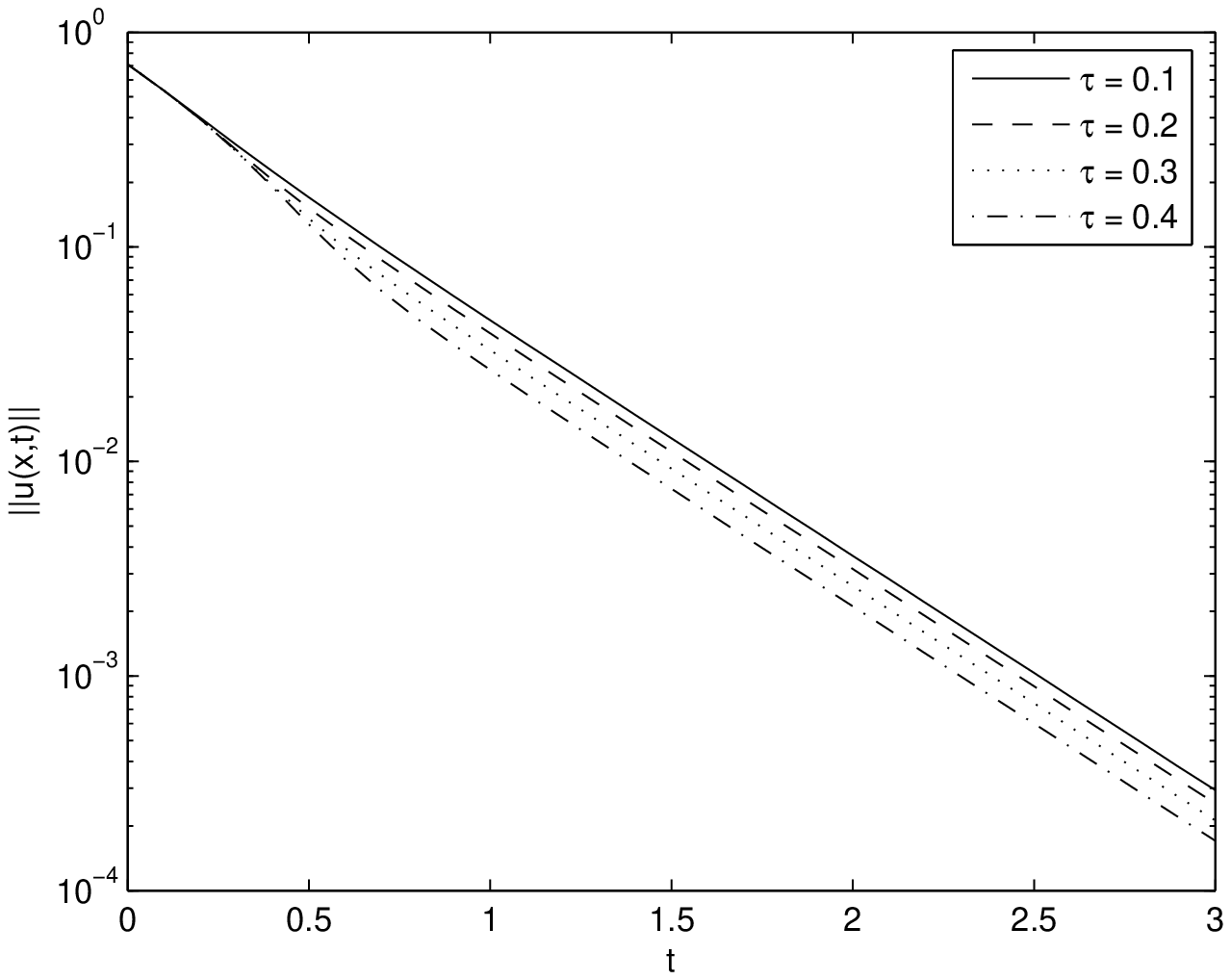}
\label{fig:subfigb}
}
\subfigure[]{
\includegraphics[scale=0.4]{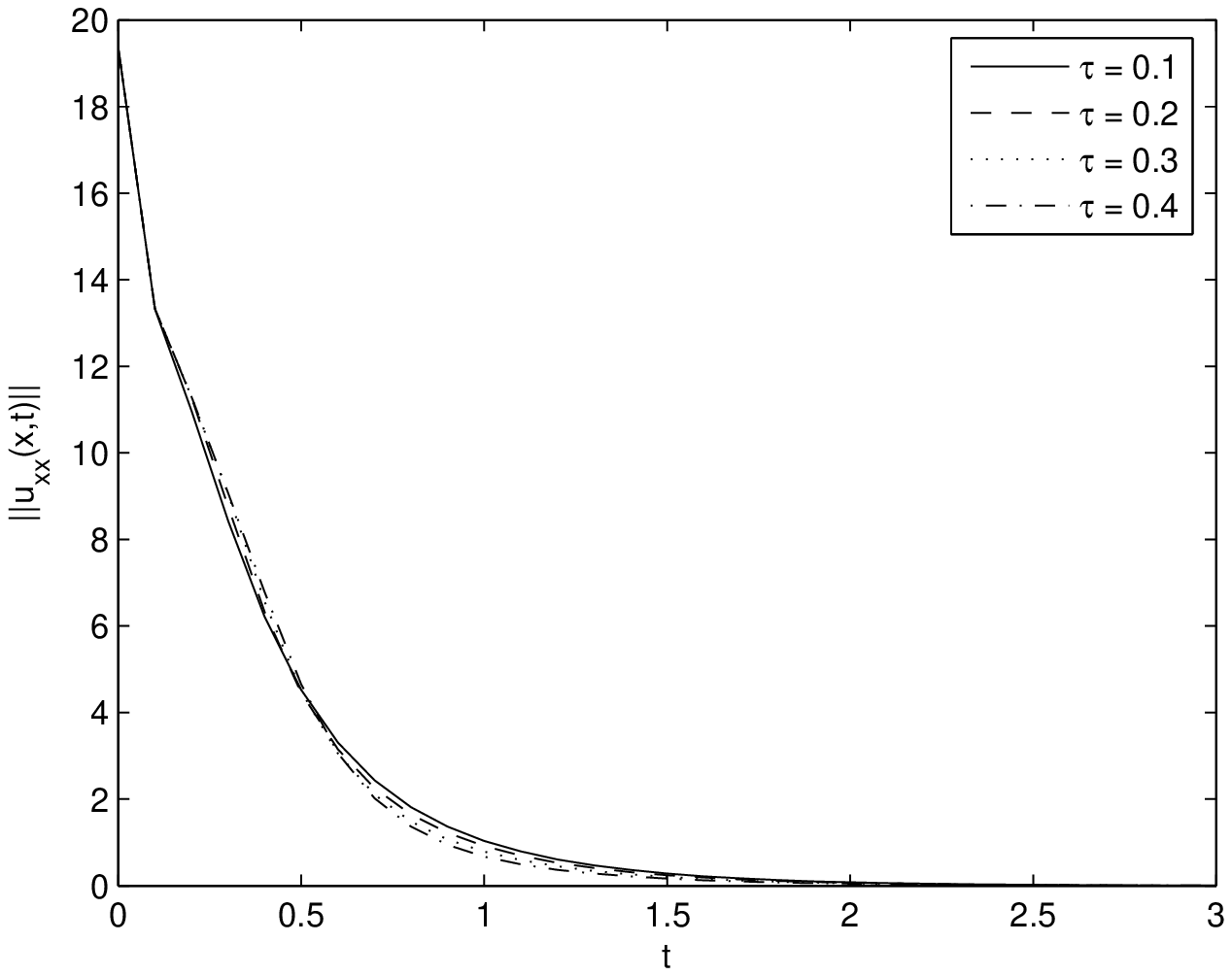}
\label{fig:subfigf}
}
\subfigure[]{
\includegraphics[scale=0.4]{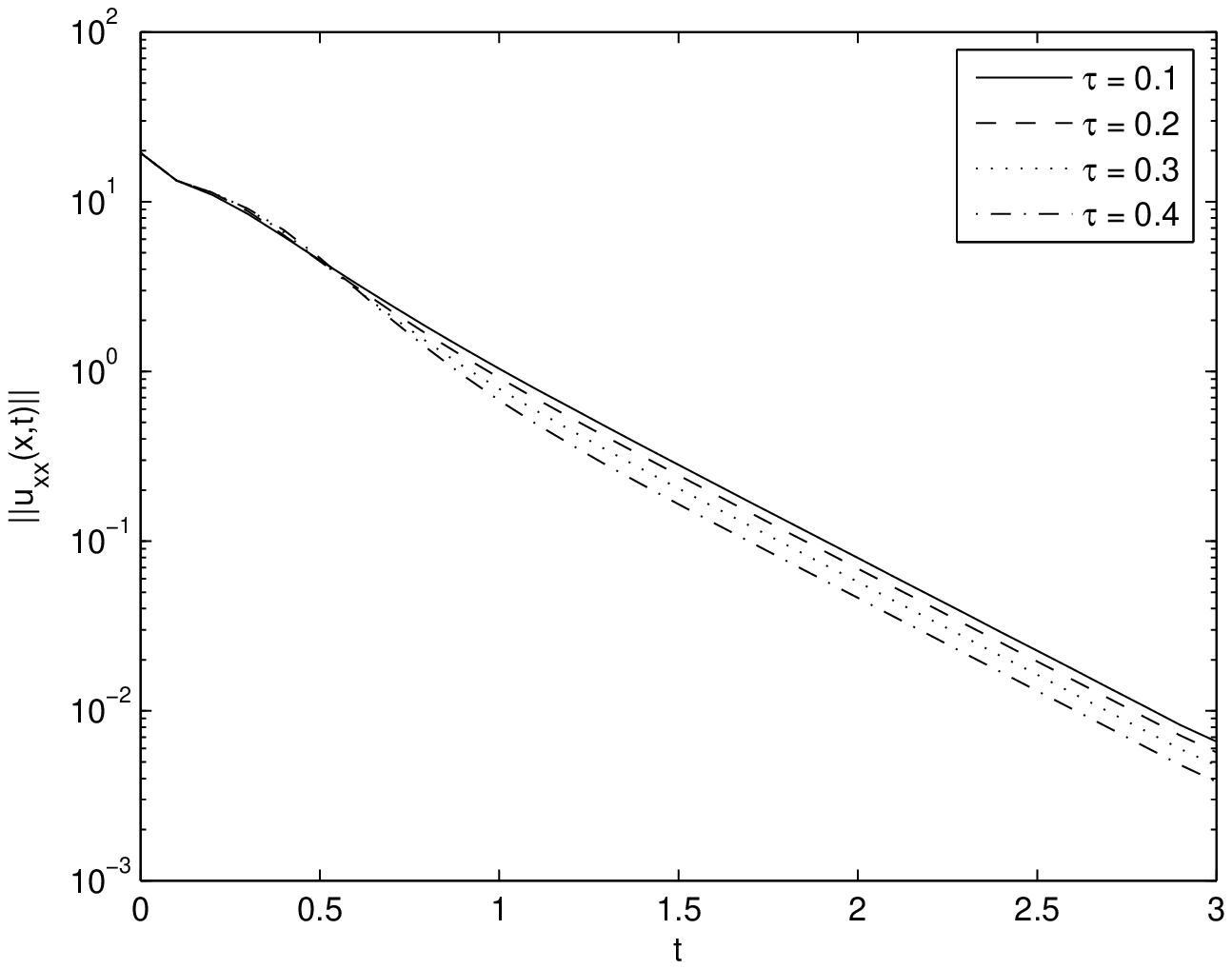}
\label{fig:subfige}
}
\caption{The $L^2$-norms $||u(x,t)||$ and $||u_{xx}(x,t)||$  when $a(x)=1+\sin(\pi x)$; (a)  $||u(x,t)||$ vs. time for different $\tau$; (b) A semi-log plot of $||u(x,t)||$ vs. time for different $\tau$; (c) $||u_{xx}(x,t)||$ vs. time for different $\tau$; (d) A semi-log plot of  $||u_{xx}(x,t)||$ vs. time for different $\tau$.}
\label{fig:Chapter4-03}
\end{figure}

\newpage
\vspace*{1.5in}
\begin{figure}[!h]
\centering
\subfigure[]{
\includegraphics[scale=0.4]{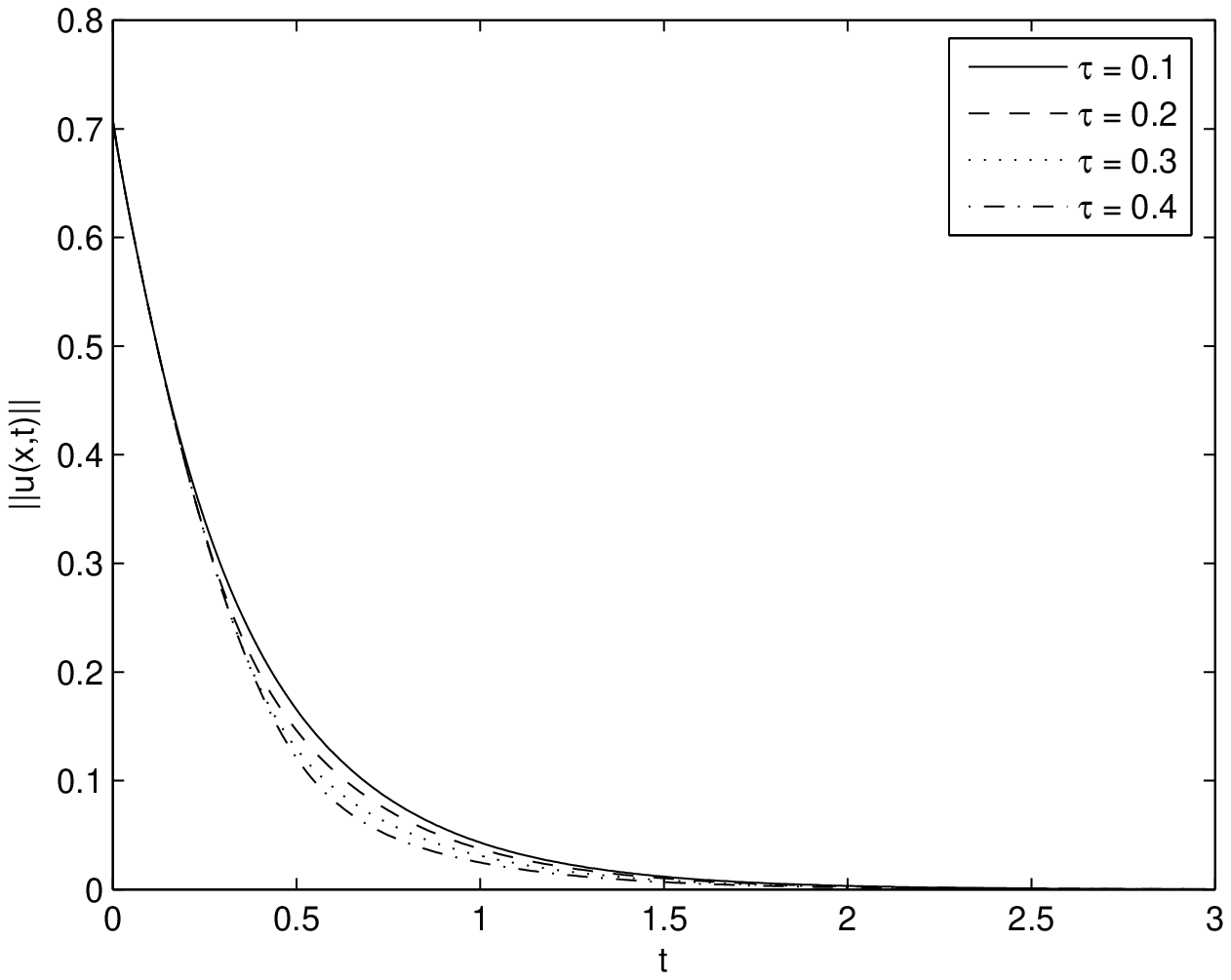}
\label{fig:subfiga}
}
\subfigure[]{
\includegraphics[scale=0.4]{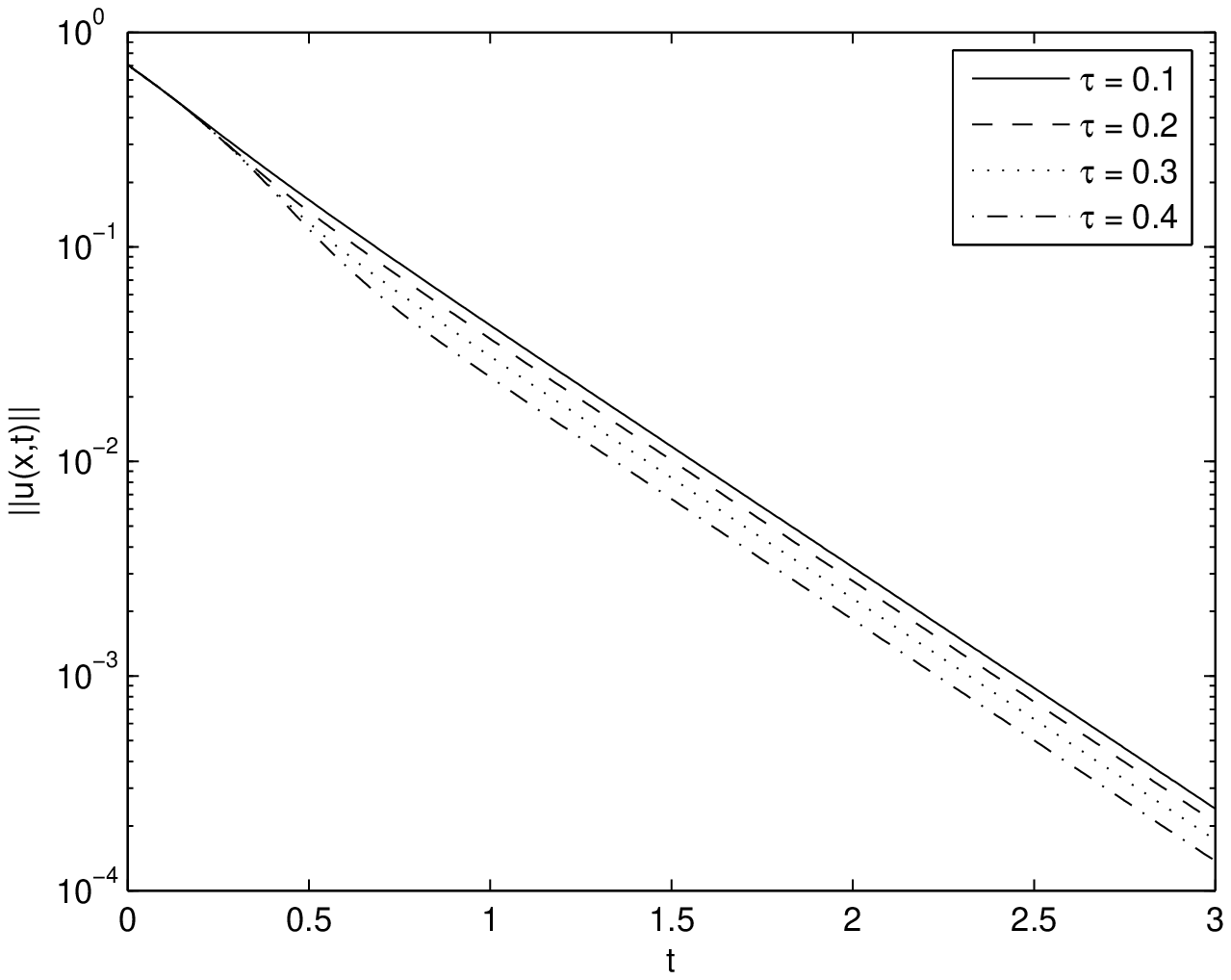}
\label{fig:subfigb}
}
\subfigure[]{
\includegraphics[scale=0.4]{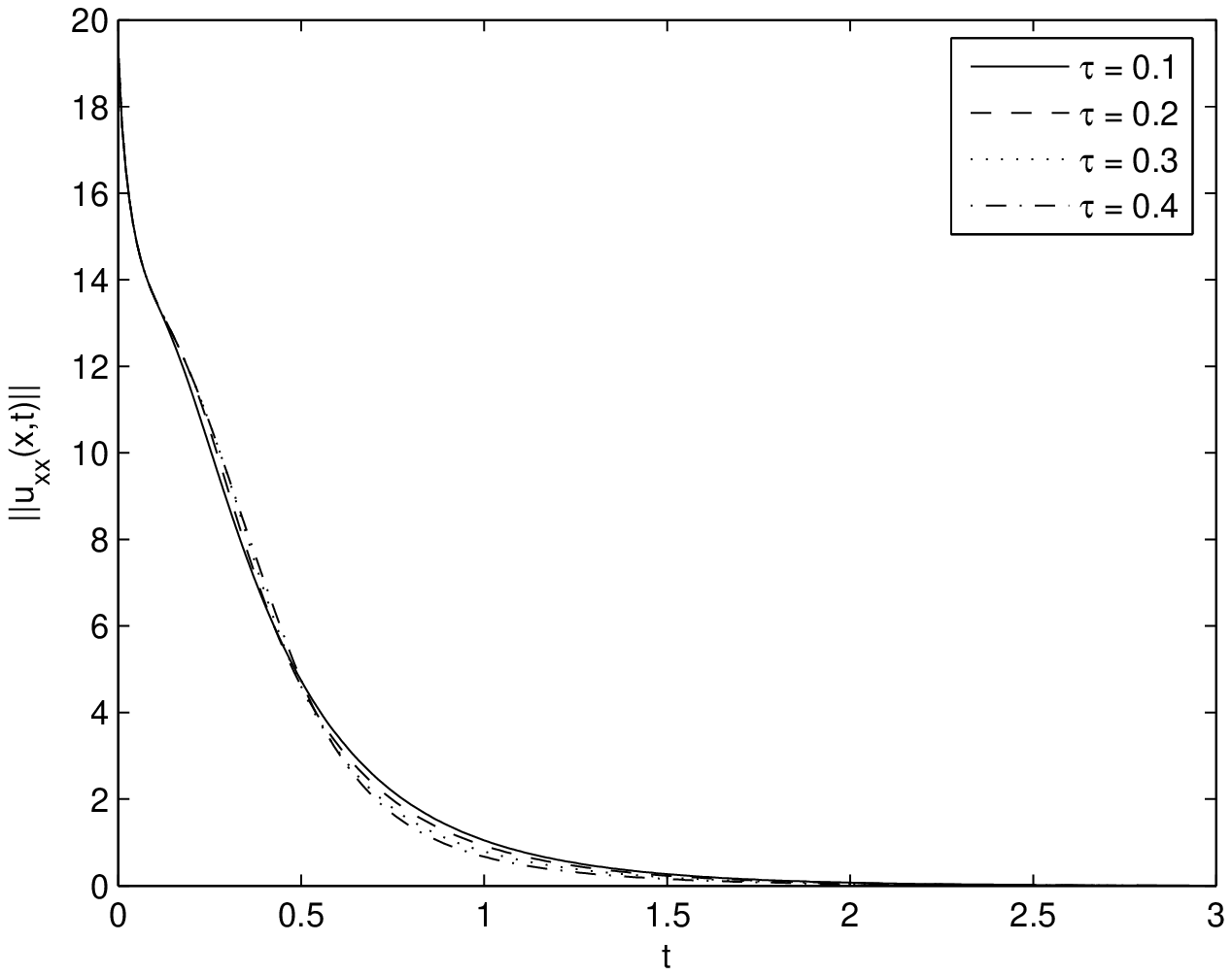}
\label{fig:subfigf}
}
\subfigure[]{
\includegraphics[scale=0.4]{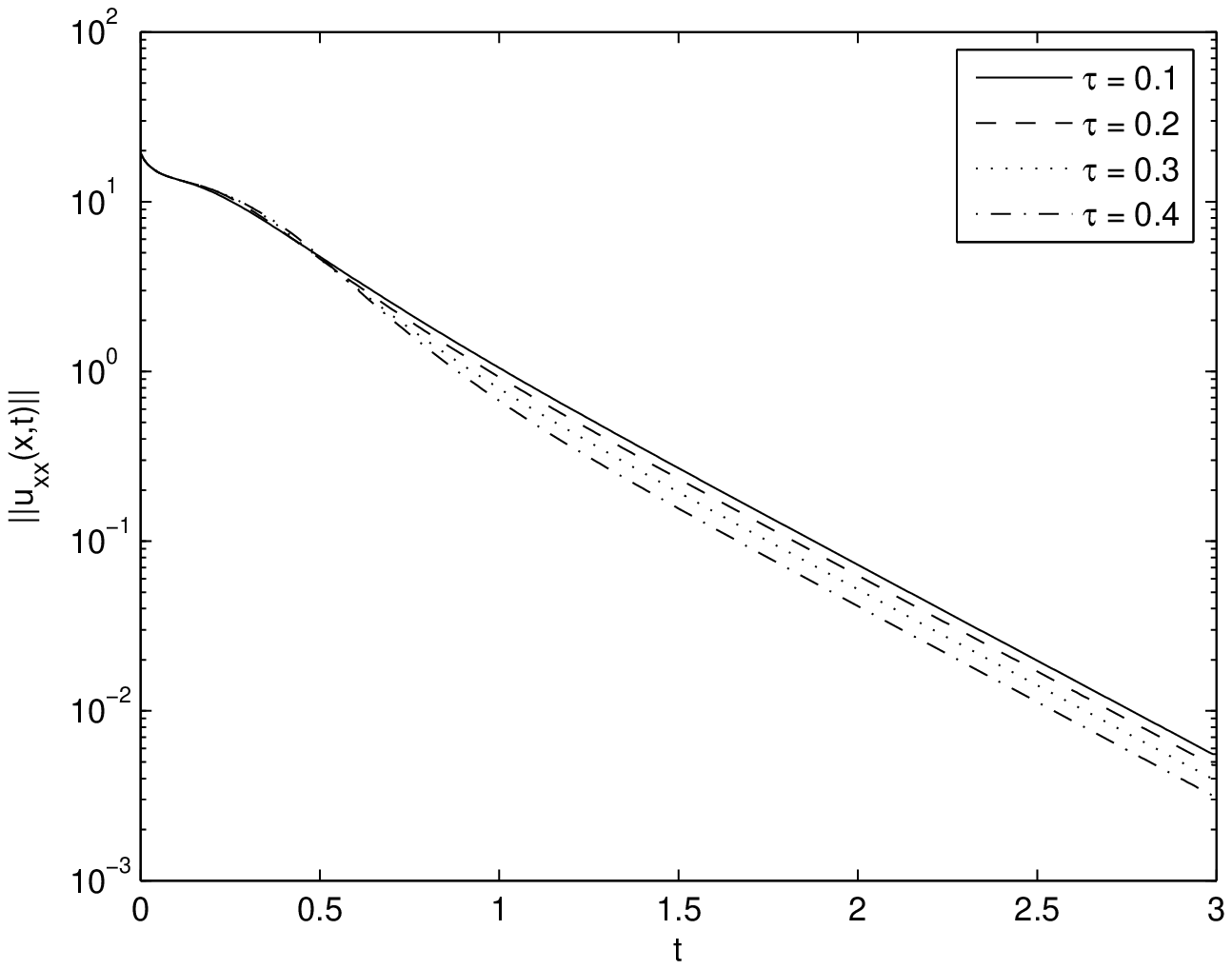}
\label{fig:subfige}
}
\caption{The $L^2$-norms $||u(x,t)||$ and $||u_{xx}(x,t)||$  when $a(x)=1+2x+\sin(2 \pi x)$; (a)  $||u(x,t)||$ vs. time for different $\tau$; (b) A semi-log plot of $||u(x,t)||$ vs. time for different $\tau$; (c) $||u_{xx}(x,t)||$ vs. time for different $\tau$; (d) A semi-log plot of  $||u_{xx}(x,t)||$ vs. time for different $\tau$.}
\label{fig:Chapter4-03}
\end{figure}

\newpage

\section{Conclusion}
In this paper, a nonlinear dispersive equation with time-delay has been considered in a bounded interval. A well-posedness result has been established in an appropriate functional space. Moreover, the exponential stability of the solutions are shown provided that the delay is small enough. Finally, the theoretical results are illustrated through numerical simulations.

In future works, we aspire to investigate the well-posedness and stability of the same equation but with higher nonlinearity $u^{\alpha} (x,t-\tau)  u_x(x,t)$, where $\alpha >1$.

\section*{Acknowledgment}
This work was supported and funded by Kuwait University, Research Grant No. SM05/18.

\section*{Conflict of Interest}  The authors declare that they have no conflict of interest.


\begin{thebibliography}{90}

\bibitem{AABM} E. M. Ait Benhassi, K. Ammari, S. Boulite and L.  Maniar, Feedback stabilization of a class of evolution equations with delay, {\em J. Evol. Equ.,} 9 (2009), 103--121.

\bibitem{usm} K. Al-Khaled, N. Haynes, W. Schiesser and M. Usman, Eventual periodicity of the forced oscillations for a Korteweg--de Vries type equation on a bounded domain using a sinc collocation method, {\em Journal of Computational and Applied Mathematics,} 330 (2018) 417--428.

\bibitem{c2} F. Al-Musallam, K. Ammari, and  B. Chentouf, Asymptotic analysis of a 2D overhead crane with input delays in the boundary control, {\em Zeitschrift fur Angewandte Mathematik und Mechanik}, 98 (2018),  1103--1122.

\bibitem{c1} K. Ammari and  B. Chentouf,  On the exponential and polynomial  convergence  for  a  delayed wave  equation without  displacement, {\em Applied Mathematics Letters}, 86 (2018), 126--133.

\bibitem{amcrep} K. Ammari and E. Cr\'epeau, Feedback stabilization and boundary controllability  of the  Korteweg-de Vries equation on a star-shaped network, {\em SIAM Journal on Control and Optimization,}, 56 (2018), 1620--1639.

\bibitem{amcrepbbm} K. Ammari and E. Cr\'epeau, Well-posedness and stabilization of the Benjamin-Bona-Mahony equation on star-shaped networks, {\em Systems Control Lett.,} 127 (2019), 39--43.


\bibitem{ammarinicaise} K. Ammari and S. Nicaise, Stabilization of elastic systems by collocated feedback, {\em Lecture Notes in Mathematics,} 2124, Springer, Cham, 2015.



\bibitem{ANP1} K. Ammari, S. Nicaise and C. Pignotti, Stability of an abstract-wave equation with delay and a Kelvin-Voigt damping, {\em Asymptot. Anal.,} 95 (2015), 21--38.

\bibitem{ANP2} K. Ammari, S. Nicaise and C. Pignotti, Stabilization by switching time-delay, {\em Asymptot. Anal.,} 83 (2013), 263--283.

 \bibitem{ANP3} K. Ammari, S. Nicaise and C. Pignotti, Feedback boundary stabilization of wave equations with interior delay, {\em Systems Control Lett.,} 59 (2010), 623--628.

\bibitem{ba1} A. Balogh, D. S. Gilliam and V. I. Shubov, Stationary solutions for a boundary controlled Burgers' equation,  {\em Mathematical and Computer Modeling}, 33 (2001), 21--37.


\bibitem{ba2} A. Balogh and M. Krstic, Global boundary stabilization and regularization of Burgers' equation, {\em Proceedings of the American Control Conference}, San Diego, California, 1712--1716, 1999.


\bibitem{ba3} A. Balogh and M. Krstic, Burgers' equation with nonlinear boundary feedback: Stability, well-Posedness and simulation, {\em Mathematical Problems in Engineering}, 6 (2000), 189--200.

\bibitem{bi1} P. Biler,  Asymptotic behavior in time of solutions to some equations generalizing the Korteweg-de Vries-Burgers equation, {\em Bulletin of the Polish Academy of Sciences, Mathematics}, 32 (1984), 275-282.

\bibitem{bi2} P. Biler, Large-time behavior of periodic solutions to dissipative equations of Korteweg-de Vries-Burgers type, {\em Bulletin of the Polish Academy of Sciences, Mathematics}, 32 (1984), 401--405.

\bibitem{bo1} J. L. Bona, V. A. Dougalis, O. A. Karakashian, and W. R. McKinney, Computations of blow-up and decay for periodic solutions of the generalized Korteweg-de Vries Burgers equation, {\em Applied Numerical Mathematics}, 10 (1992), 335--355.


\bibitem{bo2} J. L. Bona and L. Luo, Decay of solutions to nonlinear, dispersive wave equations, {\em Differential and Integral Equations}, 6 (1993), 961--980.

\bibitem{bo3} J. L. Bona and L. Luo,   More results on the decay of solutions to nonlinear dispersive wave equations, {\em Discrete and Continuous Dynamical Systems}, 1 (1995), 151--193.


\bibitem{bo4} J. L. Bona , V.A. Dougalis, A. Karakashian and W. R. McKinney,  The effect of dissipation on solutions of the generalized Korteweg--de Vries equation, {\em Journal of Computational and Applied Mathematics} 74 (1996) 127--154.


\bibitem {bou} J. Boussinesq, Essai sur la th\'eorie des eaux courantes, M\'emoires Pr\'esent\'es par Divers Savants \`a l'Acad. des Sci. Inst. Nat. France,  23 (1877), 1--680.


\bibitem {br} H. Brezis, {Functional Analysis, Sobolev Spaces and Partial Differential Equations}, Universitex, Springer, 2011.

\bibitem{capz} R. A. Capistrano-Filho and B. Y. Zhang, Initial boundary value problem for Korteweg-de Vries equation: a review and open problems, \emph{S\~{a}o Paulo J. Math. Sciences}, 13 (2019), 402--417.

\bibitem{Cerpa} E. Cerpa, Control of a Korteweg-de Vries equation: a tutorial, {\em Math. Control Relat. Fields}, { 4} (2014), 45--99.

\bibitem{Cerpa_Crepeau} E. Cerpa and E. Cr\'epeau, Boundary controllability for the nonlinear Korteweg-de Vries equation on any critical domain, {\em Annales de l'Institut Henri Poincare (C) Non Linear Analysis }, {  26} (2009), 457--475.


\bibitem {cc}  H. C. Chang, Nonlinear waves on liquid film surfaces-II.  Flooding in a vertical tube, \emph{Chem. Eng. Sci.}, 41 (1986),  2463--2476.

\bibitem{c5}  B. Chentouf, Compensation of the interior delay effect for a rotating disk-beam system, {\em  IMA Journal of Math. Control and Information},  33(4) (2016), 963--978.


\bibitem {mo1}	B. Chentouf, N. Smaoui and A. Alalabi, Nonlinear Adaptive Boundary Control of the Modified Generalized Korteweg-de Vries-Burgers Equation, \emph{Complexity}, vol. 2020 (2020), Article ID 4574257, 1--18.

\bibitem {4}  B. I. Cohen, J. A. Krommes, W. M. Tang, and M. N. Rosenbluth, Nonlinear saturation of the dissipative trapped-ion mode by mode coupling, {\it Nuclear Fusion}, 16 (1976), 971--992.


\bibitem{CoCre} J. M. Coron and E. Cr\'epeau, Exact boundary controllability of a nonlinear KdV equation with critical lengths, {\em Journal of the European Mathematical Society}, { 6} (2004), 367-398.

\bibitem{Crepeau} E. Cr\'epeau, Exact boundary controllability of the Korteweg-de Vries equation with a piecewise constant main coefficient, {\em  Systems $\&$ Control Letters}, { 97} (2016), 157--162.

\bibitem {erd} M. B. Erdo\u{g}an and N. Tzirakis, \emph{Dispersive Partial Differential Equations}, Cambridge University Press, 2016.

\bibitem{H} G. H.  Hardy, J. E. Littlewood  and G. P\' olya, {\it Inequalities}, 2nd ed. Cambridge, England: Cambridge University Press, 1988.

\bibitem {jk} A. Jeffrey and T. Kakutani,  Weak nonlinear dispersive waves: A discussion centered around the Korteweg--De Vries equation, \emph{SIAM Rev.}, 14 (1972), 582--643.

\bibitem {kdv} D. J. Korteweg and G. de Vries, On the change of form of long waves advancing in a rectangular canal, and on a new type of long stationary waves, \emph{Philos. Mag.} 39 (1895), 422--443.

\bibitem{kr} M. Krstic, On global stabilization of Burgers' equation by boundary control, {\em Systems and Control Letters}, 37 (1999), 123--142.

\bibitem {kt} Y. Kuramoto, T. Tsuzuki, On the formation of dissipative structures in reaction-diffusion systems, \emph{Progr. Theoret. Phys.}, 54 (1975), 687--699.


\bibitem {lig}  M. J. Lighthill, On waves generated in dispersive systems to travelling effects, with applications to the dynamics of rotating fluids, \emph{J. Fluid Mech.}, 27 (1967),  725--752.

\bibitem {lina} F. Linares and A. F. Pazoto, On the exponential decay of the critical generalized Korteweg-de Vries with localized damping, \emph{Proc. Amer. Math. Soc.}, 135 (2007), 1515--1522.

\bibitem {lipo} F. Linares and G. Ponce, \emph{Introduction to Nonlinear Dispersive Equations}, , Springer-Verlag, New York, 2009.


\bibitem{liu2} W. J. Liu  and M. Krstic,  Adaptive Control of Burgers' Equation with Unknown Viscosity, {\em International Journal of Adaptive Control and Signal Processing}, 15 (2001), 745--766.


\bibitem{liu} W. J. Liu, Asymptotic behavior of solutions of time-delayed Burgers equation, {\em Discrete Continuous Dynam. Systems-B}, { 2} (2002), 47-56.

\bibitem{NPSicon06} S. Nicaise and C. Pignotti, Stability and instability results of the wave equation with a delay term in the boundary or internal feedbacks, {\em SIAM J. Control Optim.}, { 45} (2006), 1561--1585.

\bibitem{Pazoto} A. F. Pazoto, Unique continuation and decay for the Korteweg-de Vries equation with localized damping, {\em  ESAIM: Control, Optimization and Calculus of Variations},{ 11} (2005), 473--486.

\bibitem{Perla} G. Perla Menzala, C. F. Vasconcelos and E. Zuazua, Stabilization of the Korteweg-de Vries equation with localized damping, {\em Quarterly of applied Mathematics}, { 60} (2002), 111--129.

\bibitem{Rosier} L. Rosier, Exact boundary controllability of the Korteweg-de Vries equation on a bounded domain, {\em ESAIM: COCV}, { 2} (1997), 33--55.

\bibitem{ro3} L. Rosier and B. Y Zhang, Global stabilization of the generalized Korteweg-de Vries equation posed on a finite domain, \emph{SIAM J. Control Optim.}, 45 (2006), 927--956.

\bibitem{roz} L. Rosier and B. Y Zhang, Control and stabilization of the Korteweg-de Vries equation: Recent progresses, \emph{J. Syst. Sci. Complex.}, 22 (2009), 647--682.

\bibitem{r1} G. Sivashinsky,  Nonlinear analysis for hydrodynamic instability in Laminar flames. Derivation of basic equations, \emph{Acta Astronautica}, {4} (1977), 1177--1206.

\bibitem{s7} N. Smaoui, Controlling the dynamics of Burgers equation with a high-order nonlinearity, \emph{International Journal of Mathematics and Mathematical Sciences}, {62} (2004), 3321-3332.

\bibitem{s8} N. Smaoui, Nonlinear boundary control of the Generalized Burgers Equation, \emph{Nonlinear Dynamics}, {37} (2004), 75-86.


\bibitem{s1} N. Smaoui and R. Al-Jamal, A nonlinear boundary control for the dynamics of the generalized Korteweg-de Vries-Burgers equation, {\em Kuwait Journal of Science and Engineering}, 34 (2007), 57--76.

\bibitem{s2} N. Smaoui and R. Al-Jamal, Boundary control of the generalized Korteweg-de Vries-Burgers equation, {\em Nonlinear Dynamics}, 51 (2008), 439--446.

\bibitem{s3} N. Smaoui, A. El-Kadri, and M. Zribi, Adaptive boundary control of the forced generalized Korteweg-de Vries-Burgers equation, {\em European Journal of control}, 16 (2010) 72--84.

\bibitem{s4} N. Smaoui, A. El-Kadri, and M. Zribi,  Nonlinear boundary control of the unforced generalized Korteweg-de Vries-Burgers equation, {\em Nonlinear Dynamics}, 60 (2010), 561-574.

\bibitem{mo2} N. Smaoui, B. Chentouf  and A. Alalabi, Boundary linear stabilization of the modified generalized Korteweg-de Vries-Burgers equation,  \emph{Advances in Difference Equations}, 2019, (2019), Article number: 457, 17 pages.

\bibitem{s5} N. Smaoui and M. Mekkaoui, The generalized Burgers equation with and without a time-delay, {\em Journal of Applied Mathematics and Stochastic Analysis}, 1 (2004), 73--96.

\bibitem{s6} N. Smaoui and M. Zribi, A finite dimensional control of the dynamics of the generalized Korteweg-de Vries Burgers equation, {\em Applied Mathematics and Information Sciences-An International Journal}, 3 (2009), 207--221.

\bibitem{wang} Y. Tang and M. Wang,  A remark on exponential stability of time-delayed Burgers equation, {\em Discrete Contin. Dyn. Syst. Ser. B.,} { 12} (2009), 219--225.

\bibitem {whi1} G. B. Whiham, Non-linear dispersive waves, \emph{Proc. Roy. Soc. Ser. A}, 283 (1965), 238--261.

\bibitem {whi2}    G. B. Whitham, \emph{Linear and Nonlinear Waves}, Pure and Applied Mathematics, John Wiley $\&$ Sons, New York-London-Sydney-Toronto, 1974.

\bibitem {za}  N. J. Zabusky and M. D. Kruskal, Interaction of ``solitons'' in a collisionless plasma and the recurrence of initial states, \emph{Phys. Rev. Letters}, 15 (1965), 240--243.

\bibitem {zou} X. Zou, Delay induced traveling wave fronts in reaction diffusion equations of KPP--Fisher type, \emph{Journal of Computational and Applied Mathematics}, 146 (2002) 309--321.

\end{thebibliography}
\end{document}